\documentclass[twoside,leqno,11pt]{article}
\usepackage{PhDstylefile}
\usepackage{graphicx}

\newtheorem*{thm1}{Theorem 1}
\newtheorem*{thm2}{Theorem 2}
\newtheorem*{thm3}{Theorem 3}

\begin{document}

\begin{titlepage}

\title{\textbf{Tannaka duality over ring spectra}}
\author{James Wallbridge\\
Kavli IPMU (WPI), UTIAS, University of Tokyo\\
5-1-5 Kashiwanoha, Kashiwa, 277-8583, Japan\\
james.wallbridge@ipmu.jp}
\date{}
\maketitle

\begin{center}
\textbf{Abstract}
\end{center}

We prove a Tannaka duality theorem for $(\infty,1)$-categories.  This is a duality between certain derived group stacks, or more generally certain derived gerbes, and symmetric monoidal $(\infty,1)$-categories endowed with particular structure.  This duality theorem is defined over commutative ring spectra and subsumes the classical statement.  We show how the classical theory arises as a special case of our more general theory.  The application to perfect complexes is explored.

\end{titlepage}

\tableofcontents


\section*{Introduction}

Classical Tannaka duality is a duality between certain groups and certain monoidal categories endowed with particular structure.  Tannaka duality for $(\infty,1)$-categories refers to a duality between certain derived group stacks and certain monoidal $(\infty,1)$-categories endowed with particular structure.  This duality theorem subsumes the classical statement.  Our starting point is the philosophy developed by Grothendieck which is to consider fundamental groupoids (ie. 1-truncated homotopy types) as automorphism groupoids of certain ``fiber" functors.  

This philosophy began with Grothendieck's study of Galois theory axiomatically using purely categorical methods \cite{SGA}.  He introduced the notion of a \textit{Galois category}, that is, a category $C$ endowed with a fiber functor satisfying conditions such that $C$ is equivalent to the category of representations of a profinite group.  More precisely, let $(C,\omega)$ be a Galois category and define the fundamental group of $C$ at the base point $\omega$ to be the collection
\[ \pi_1(C,\omega):=\Aut(\omega) \]
of automorphisms of the fiber functor.  Then $\pi_1(C,\omega)$ is a profinite group and the functor 
\[     C\ra \pi_1(C,\omega)\text{-}\tu{FSet}       \]
from $C$ to the category $\tu{FSet}$ of finite sets endowed with an action of $\pi_1(C,\omega)$, is an equivalence of categories.  This is the \textit{Galois duality} statement.  By looking at the problem categorically Grothendieck was able to transfer the study of $1$-truncated homotopy types to contexts where such a notion was previously difficult to define.  In this way he defined a new topological invariant - the \'etale fundamental group.

An analogous notion in the case of compact topological groups was initiated much earlier by Tannaka \cite{Ta} who showed that a compact group can be reconstructed from its category of representations.  The group arises as the group of tensor preserving automorphisms of the forgetful fiber functor from the category of representations to its underlying category of vector spaces.  In \cite{Kr}, Krein characterised those categories of the form $\Rep(G)$ which arise in this way.

The passage from Galois theory to Tannaka theory is the linearization process of replacing sets by vector spaces.  Following the Galois philosophy of Grothendieck above, Saavedra developed a Tannaka duality theory for affine group schemes where the abstract category dual is termed a \textit{neutralized Tannakian  category} \cite{Sa}.  The \textit{neutralized Tannaka duality} statement is then that the automorphism group of fiber functors is an affine group scheme and the Tannakian category is equivalent to the category of representations of this affine group scheme.  

More precisely, let $k$ be a field.  Then 
\[  \Rep_*:\tu{AffGp}_k^{op}\ra(\Tan_k)_*  \]
is an equivalence of categories where $(\Tan_{k})_*$ is the category of pairs $(T,\omega)$ where $T$ is a $k$-Tannakian category and $\omega$ is a fiber functor.  The category $\tu{AffGp}_k$ on the left hand side is the category of affine $k$-group schemes.  

Let $(T,\omega)$ be a neutralized Tannakian category and define the algebraic fundamental group of $T$ at the base point $\omega$ to be the collection
\[ \pi_1(T,\omega)^{alg}:=\Aut^{\otimes}(\omega) \]
of tensor-perserving automorphisms of the fiber functor.  Then $\pi_1(T,\omega)^{alg}$ is an affine group scheme and the functor 
\[    T\ra \Rep(\pi(T,\omega)^{alg})     \]
is an equivalence of categories.  These affine group schemes are considered as algebraic versions of $1$-truncated homotopy types and lead to Deligne's definition of the algebraic fundamental group \cite{D1}. 

More generally, Saavedra wrote a non-neutral Tannaka statement which characterises those categories $T$ which are equivalent to the category of representations of the stack $\Fib(T)$ of fiber functors on $T$.  In \cite{D2}, Deligne completed the proof that this stack is an affine gerbe (in the $ffqc$-topology).  More precisely they showed that the functor
\[  \Rep:\Ger(k,ffqc)^{op}\ra\Tan_k  \]
is an equivalence of categories where $\Tan_{k}$ is the category of $k$-Tannakian categories and $\Ger(k,ffqc)$ is the category of affine gerbes over $\Spec(k)$ in the $ffqc$-topology.  When $\Fib(T)$ is the neutral gerbe of $G$-torsors, for $G$ an affine group scheme, we recover the neutral Tannaka statement.  

In order to study \textit{higher} homotopy types it is necessary to move to a higher categorical generalisation of the above ideas \cite{Gr}.  Work in this direction began in \cite{T1} by To\"en.  It involves the use of $(\infty,1)$-categorical techniques developed in work by Joyal \cite{Jo} and Lurie \cite{Lu} and in the theory of derived algebraic geometry by Lurie \cite{L1} and To\"en and Vezzosi \cite{TVI,TVII}.  

Informally, the passage from categories to $(\infty,1)$-categories involves replacing the category of sets by the $(\infty,1)$-category of spaces (topological spaces, Kan complexes or one such equivalent model).  Indeed, one possible (although not so convenient) model for an $(\infty,1)$-category is a simplical category.  This forces generalisations of other familiar categorical concepts.  The category of abelian groups is replaced by the $(\infty,1)$-category of spectra, an abelian category is replaced by a stable $(\infty,1)$-category, commutative rings are replaced by commutative ring spectra (called $\Einfty$-rings) and a rigid category is replaced by a rigid $(\infty,1)$-category.  More examples are discussed througout the text.   

In the spirit of the above, we prove the following pointed or neutralized Tannaka duality statement for $(\infty,1)$-categories.   

\begin{thm1}[Neutralized $(\infty,1)$-Tannaka duality~: see Theorem~\ref{pointedinftytannakaduality}]
Let $R$ be a (connective, bounded connective) $\Einfty$-ring and $\tau$ a subcanonical topology.  Then the map
\[    \Perf_{*}:\TGp(R,\tau)^{op}\ra(\Tens^{\rig}_R)_*   \]
is fully faithful.  Moreover, the adjunction $\Fib_*\dashv\Perf_*$ induces the following~:
\begin{enumerate}
\item Let $R$ be a $\Einfty$-ring.  Then $(T,\omega)$ is a pointed finite $R$-Tannakian $\infty$-category if and only if it is of the form $\Perf_*(G)$ for $G$ a finite $R$-Tannakian group stack.
\item Let $R$ be a connective $\Einfty$-ring.  Then $(T,\omega)$ is a pointed flat $R$-Tannakian $\infty$-category if it is of the form $\Perf_*(G)$ for $G$ a flat $R$-Tannakian group stack.
\item Let $R$ be a bounded connective $\Einfty$-ring.  Then $(T,\omega)$ is a pointed positive $R$-Tannakian $\infty$-category if it is of the form $\Perf_*(G)$ for $G$ a positive $R$-Tannakian group stack.
\end{enumerate}
\end{thm1}

We will briefly explain the objects in this statement.  The category $\TGp(R,\tau)$ is the $(\infty,1)$-category of $R$-Tannakian group stacks.  These are affine group stacks which are weakly rigid in an appropriate sense.  The category $(\Tens_R^\rig)_*$ is the $(\infty,1)$-category of pointed rigid $R$-tensor $(\infty,1)$-categories.  The objects in this category are rigid stable $R$-linear symmetric monoidal $(\infty,1)$-categories together with an $R$-linear symmetric monoidal functor to the $(\infty,1)$-category of rigid $R$-modules.  

We will introduce three topologies on the $(\infty,1)$-category of $R$-algebras for an $\Einfty$-ring $R$ called the finite, flat and positive topologies (denoted by $fin$, $fl$ and $\geq 0$ respectively).  A rigid $R$-tensor $(\infty,1)$-category will be called pointed Tannakian with respect to one of these topologies if it is equipped with a fiber functor satisfying certain properties that reflect this topology (see Definition~\ref{pointeddfn}).  A $\tau$-$R$-Tannakian group stack for $\tau\in\{fin,fl,\geq 0\}$ is an $R$-Tannakian group stack such that its associated Hopf $R$-algebra reflects the topology $\tau$ (see Definition~\ref{tannakiangroupstack}).  Note that we prove an equivalence of $(\infty,1)$-categories in the finite case but we rest on showing that the functor $\Perf_*$ on Tannakian objects is fully faithful in the flat and positive cases.

The upshot of the above theorem is that we can define full homotopy types categorically as follows.  Let $(T,\omega)$ be a pointed $\tau$-$R$-Tannakian $\infty$-category and define the \textit{algebraic homotopy type} of $T$ at the base point $\omega$ to be the collection
\[ \pi(T,\omega)^{alg}:=\Aut^{\otimes}(\omega) \]
of $\tau$-fiber functors.  Then $\pi(T,\omega)^{alg}$ is a $\tau$-$R$-Tannakian group stack and the functor 
\[    T\ra \Rep(\pi(T,\omega)^{alg})     \]
is an equivalence of $(\infty,1)$-categories.  

We also have the more general neutral Tannaka duality statement for $(\infty,1)$-categories.

\begin{thm2}[Neutral $(\infty,1)$-Tannaka duality~: see Theorem~\ref{inftytannakaduality}]
Let $R$ be a (connective, bounded connective) $\Einfty$-ring and $\tau$ a subcanonical topology.  Then the map
\[    \Perf:\TGer(R,\tau)^{op}\ra\Tens^{\rig}_R  \]
is fully faithful.  Moreover, the adjunction $\Fib\dashv\Perf$ induces the following~:
\begin{enumerate}
\item Let $R$ be a $\Einfty$-ring.  Then $T$ is a finite $R$-Tannakian $(\infty,1)$-category if and only if it is of the form $\Perf(G)$ for $G$ a neutral finite $R$-Tannakian gerbe.
\item Let $R$ be a connective $\Einfty$-ring.  Then $T$ is a flat $R$-Tannakian $(\infty,1)$-category if it is of the form $\Perf(G)$ for $G$ a neutral flat $R$-Tannakian gerbe.
\item Let $R$ be a bounded connective $\Einfty$-ring.  If $(T,\omega_1)$ and $(T,\omega_2)$ are two pointed positive $R$-Tannakian $(\infty,1)$-categories then there exists a positive cover $R\ra Q$ such that
\[      \omega_1\otimes_{R}Q\ra\omega_2\otimes_{R}Q   \]
is an equivalence.
\end{enumerate}
\end{thm2}

An $(\infty,1)$-category is $R$-Tannakian with respect to a topology $\tau\in\{fin,fl,\geq 0\}$ if it is a rigid $R$-tensor $(\infty,1)$-category such there exists a fiber functor with respect to $\tau$ (see Definition~\ref{nonpointeddfn}).  A $\tau$-Tannakian gerbe is a stack on the site of $R$-algebras with respect to $\tau$ which is locally equivalent to the classifying stack of a $\tau$-$R$-Tannakian group stack.  It is said to be a neutral if there exists a global point (see Definition~\ref{tannakiangerbe}).  

Here $\TGer(R,\tau)$ is the $(\infty,1)$-category of neutral $\tau$-$R$-Tannakian gerbes.  Since the positive topology is not subcanonical we rest with the weaker statement of (3).  There may also exist a reasonable notion of \textit{non-neutral} $(\infty,1)$-Tannaka duality where the duality holds over an extension of the base $\Einfty$-ring but we will not consider this more general case here.

One application of our main theorem is to the study of homotopy types of complex algebraic varieties.  Let $\bb{C}$ be the field of complex numbers.  To a complex algebraic variety $X$ we can associate its algebraic de Rham cohomology $H^*_{dR}(X)$ which are finite dimensional complex vector spaces.  An important property of these vector spaces is that they are endowed with a pure Hodge structure
\[     H^n_{dR}(X,\bb{Z})\otimes_{\bb{Z}}\bb{C}=\bigoplus_{p+q=n}H^{p,q}_{dR}(X)     \]
reflecting the arithmetic properties of the base field $\bb{C}$.  This pure Hodge structure is encoded in the Tannakian formalism by the action of a pro-algebraic group $\pi_1(X,x)^{alg}$ on the space $H^*_{dR}(X)$ which is the dual of the Tannakian category of pure Hodge structures.  This idea extends to the case of mixed Hodge structures and reflects the general principle in algebraic geometry that cohomology theories take their values in the category of linear representations of a pro-algebraic group.

If we extend our scope beyond cohomology and consider possible full homotopy types of our variety $X$ we arrive at the following question.  Can one produce a functorial geometric object $\cH(X)$ which can be thought of as the homotopy type of $X$, together with a cohomology functor 
\[  \cH\mapsto H^*(\cH(X))  \] 
in such a way that the action of the pro-algebraic group $\pi_1(X,x)^{alg}$ on $H^*(\cH(X))$ recovers the linear representation $H^*_{dR}(X)$.  In fact this has been done in \cite{KPT} using To\"en's theory of schematic homotopy types \cite{T4}.  A special example of a schematic homotopy type is the schematization of a space.   The fundamental group of this schematization is isomorphic to the pro-algebraic completion of the fundamental group of the underlying space.  Thus one way to interpret the schematization of a space is as a generalization of the pro-algebraic completion functor.

The final task in this paper is then to supply the missing Tannakian interpretation of the homotopy type of a complex variety, already conjectured in \cite{T4} and \cite{habilitation}.  More precisely we prove the following in Theorem~\ref{perfistannakian}.

\begin{thm3}
Let $k$ be a commutative ring and $X$ a finite CW complex.  Then $(\Perf(X,k),\omega_{x})$ is a pointed $k$-Tannakian $\infty$-category with respect to the positive topology.  When $k$ is a field of characteristic zero, the dual of this pointed positive $k$-Tannakian $\infty$-category is the schematization of $X$. 
\end{thm3}

This theorem, apart from providing an interpretation of the dual of the Tannakian $\infty$-category of perfect complexes of a space, extends the schematization functor over any base ring $k$.  These results have an application to non-abelian Hodge theory and in particular, to the question of derived non-abelian mixed Hodge structures.

\begin{center}
\bf{Outline}
\end{center}

We include here a brief overview of the contents in this paper.  We begin in Section~\ref{inftycategories} by recalling the basic theory of $(\infty,1)$-categories \cite{Lu}\cite{HTHC}.  All the objects in our Tannakian statements are either $(\infty,1)$-categories (with extra structure) or form $(\infty,1)$-categories in a natural way.  Due to the foundational work of Simpson and Hirschowitz in \cite{HS}, there exists a Quillen model structure on the category of $(\infty,1)$-precategories which is cartesian closed.  The existence of this model structure facilitates the study of all the standard categorical notions as applied to $(\infty,1)$-categories such as the theory of limits, colimits and Kan extensions.  A particularly important role will be played by the category of ind-objects of a $(\infty,1)$-category.

One side of our Tannaka duality statement describes certain symmetric monoidal $(\infty,1)$-categories.  In Section~\ref{monoidal} we introduce the general theory of  symmetric monoidal $(\infty,1)$-categories.  A key theorem (see \cite{L1}\cite{TV3}) is an equivalence between a symmetric monoidal $(\infty,1)$-category considered either as a functor from the category of pointed finite ordinals $\Gamma$ to the $(\infty,1)$-category of $(\infty,1)$-categories or as a cofibered $(\infty,1)$-category over $\Gamma$.  We have corresponding notions of commutative monoid objects, module objects and algebra objects in a  symmetric monoidal $(\infty,1)$-category.  An important example is the symmetric monoidal $(\infty,1)$-category of presentable $(\infty,1)$-categories.  This will allow us to define the notion of linearity using module objects and the $(\infty,1)$-category of linear symmetric monoidal $(\infty,1)$-categories with respect to an arbitrary commutative monoid object.

In Section~\ref{spectralalgebra} we recall the natural extensions of abelian categories and commutative rings to the $(\infty,1)$-categorical realm in the form of stable $(\infty,1)$-categories and $E_{\infty}$-rings.  The category of abelian groups in the familiar setting is replaced by the $(\infty,1)$-category of spectra and every stable $(\infty,1)$-category is naturally enriched over spectra.  A $E_\infty$-ring is then simply a commuative monoid object in the $(\infty,1)$-category of spectra.  We then consider t-structures on a stable $(\infty,1)$-category and recall how our t-structured spectral algebraic objects collapse to their classical counterparts by taking the heart construction.  

The other side of our Tannaka duality statement describes derived group stacks, or more generally, derived gerbes.  In Section~\ref{stacksgerbesandtopologies} we introduce and study these objects with a particular interest in those on the sites of $R$-algebras, for $R$ an $E_{\infty}$-ring.  These are the homotopical analogues of (affine) group schemes and (affine) gerbes.  They arise from the $(\infty,1)$-category of derived stacks which are simply sheaves of spaces on the $(\infty,1)$-category of ``derived rings", the $E_\infty$-rings of Section~\ref{spectralalgebra}, endowed with a (Grothendieck) topology.  We also include the theory of Hopf $R$-algebras and Hopf $R$-algebras with respect to a site.  We conclude the section with a couple of technical results for cosimplicial objects needed in the proof of our main theorem.

In Section~\ref{topologies} we include the study of the topologies of interest to us~: the positive, flat and finite topologies on the $(\infty,1)$-category of $E_\infty$-rings.  The latter two are shown to be subcanonical meaning that every prestack on their site is automatically a stack.  This means that we will be composing three seperate Tannaka statements depending on the topology chosen.  In Section~\ref{rigid} we introduce a necessary finiteness condition on a symmetric monoidal $(\infty,1)$-category.  This $(\infty,1)$-category of \textit{dualizable} objects is the $(\infty,1)$-category analogue of a rigid category in the standard theory.  Also included is the study of ind-rigidity and some of its consequences in the theory at hand.

We embark on the study of the duality theorem proper in Section~\ref{tannakadualityforinftycategories} by introducing $R$-tensor $(\infty,1)$-categories for an $E_\infty$-ring $R$.  These are $(\infty,1)$-categories which are stable, presentable and $R$-linear.  The full subcategory of $R$-tensor $(\infty,1)$-categories spanned by dualizable objects is the $(\infty,1)$-category of rigid $R$-tensor $(\infty,1)$-categories.  Those that are \textit{pointed} admit a $R$-tensor functor to the $(\infty,1)$-category of rigid $R$-modules.  The aim of the neutralized Tannaka duality theorem is to understand under what conditions the adjunction, whose left adjoint sends a pointed rigid $R$-tensor $(\infty,1)$-category to its derived group stack of tensor endomorphisms of its $R$-tensor functor and whose right adjoint sends a derived group stack to its pointed rigid $R$-tensor $(\infty,1)$-category of representations, is an equivalence of $(\infty,1)$-categories.  For this reason, the notion of a pointed $R$-Tannakian $(\infty,1)$-category and $R$-Tannakian group stack is introduced.  Both notions depend on the topology chosen, in particular the fiber functor is topology dependent, which gives rise to the three seperate cases.  The neutralized $(\infty,1)$-Tannaka duality theorem then describes how this equivalence is acheived, or how close one can get, in these cases.

The proof of the neutralized $(\infty,1)$-Tannaka duality theorem is deferred to Section~\ref{dualitytheorems} where care is taken to consider each case, one for each of the topologies in Section~\ref{topologies}, in turn.  In Section~\ref{neutral} the neutral version of the theory is stated and proven.  In this case only the existence of a fiber functor is assumed.  A comparison of our work with the classical results of Saavedra and Deligne is covered in Section~\ref{classicalcomparison} using the flat topology.  The classical theory embeds fully faithfully into our theory.  In particular, the analogue of a $k$-Tannakian category $T$ for the ffqc topology in our theory is the $Hk$-Tannakian $(\infty,1)$-category of bounded complexes in $T$ in the flat topology for the Eilenberg-Mac Lane ring spectrum $Hk$.  

Finally, Section~\ref{overfields} addresses some applications of the theory when $k$ is a field.  We prove that the $(\infty,1)$-category of representations of a pointed schematic homotopy type of finite cohomological dimension is a pointed $Hk$-Tannakian $(\infty,1)$-category for the positive topology.  Moreover, the $(\infty,1)$-category of perfect complexes over a finite CW complex has the structure of a pointed positive $Hk$-Tannakian $(\infty,1)$-category and when $k$ is of characteristic zero.  Its dual is precisely the schematization of the CW complex.

\begin{center}
\bf{Relations to other work}
\end{center}

The first article which discusses a theory of higher Tannaka duality is the paper \cite{T1} by To\"en.  Here, the theory is motivated and many of the key ingredients are introduced.  Several conjectures are then made.  These ideas are then refined and the conjectures stated clearly in the authors influential habilitation memoir \cite{habilitation}.  Tannakian $(\infty,1)$-categories over fields are also discussed in \textit{loc. cit.} and the present paper can be seen as one approach to answering the conjectures posed in that paper.  In order for our proofs to be realised, we rely much on the foundational work on $(\infty,1)$-categories developed in \cite{L1} by Lurie.  We also mention the references \cite{FI} and \cite{LVIII} for other approaches to derived Tannaka duality.

\subsection*{Acknowledgements}

This work is part of the authors doctorate thesis \cite{W1} completed in Toulouse and Montpellier (connected by the now infamous TEOZ 4765).  Therefore I would like to give special thanks to Bertrand To\"en.  I would also like to thank Mathai Varghese, Jacob Lurie and Carlos Simpson.  This research was supported by the University of Adelaide, Universit\'e Paul Sabatier, World Premier International Research Center Initiative (WPI) MEXT Japan and the Agence Nationale de la Recherche grant ANR-09-BLAN-0151 (HODAG).  I thank the Institut des Hautes \'Etudes Scientifiques for hospitality during the completion of this work.

\newpage

\section{$(\infty,1)$-categories}\label{inftycategories}

In this section we will give a brief review of the theory of $(\infty,1)$-categories.  All the results in this article are stated and proved in this setting so it will be worthwhile recalling the relevant definitions.  We will henceforth refer to a $(\infty,1)$-category as simply an $\infty$-category.

There are several different equivalent approaches to defining $\infty$-categories.  Knowing that certain models are equivalent enables us to move between one model or another depending on the given context or calculation.  The nature of the equivalence is a Quillen equivalence between certain model categories of $\infty$-categories~: we take the point of view that we are ultimately interested in the objects of the homotopy category.  

We will concentrate on two models which are known to be Quillen equivalent \cite{Three}~:
\[     \PC(\sK)\rightleftarrows\Cat(\sK).  \]
The category $\sK$ is the model category of simplicial sets with the Kan model structure.  We define an $\inftyz$-category to be a Kan complex.  Thus $\sK$ will be the model category of $\inftyz$-categories.  The category on the left-hand side is the category of $\sK$-enriched precategories (see \cite{HTHC} for a self contained account of the theory) with the injective or Reedy model structure.  This will play the principal role for our model category of $\infty$-categories.  

We recall that a $\sK$-enriched precategory is a functor
\[    A:\Dop\ra\sK     \]
such that $A_0$ is discrete.  An $\infty$-category is then a $\sK$-enriched precategory, which we will call an $\infty$-precategory, satisfying the Segal condition \cite{HS}.  The category on the right-hand side is the model category of simplicial categories, that is, the model category of categories enriched over $\sK$, see \cite{B2}, and is often useful when one would like to choose a strict model.  For this paper the reader can substitute freely any choice of model category of $\infty$-categories which is equivalent to the model category of simplicial categories (see for example those reviewed in \cite{B1}).  

Every $\sK$-enriched category is an $\infty$-precategory with the same set of objects and where composition of maps is strictly defined.  This induces a fully faithful functor
\[      \fG:\Cat(\sK)\ra\PC(\sK).      \]
We will very often consider a $\sK$-enriched category $C$ as an $\infty$-precategory by identifying $C$ with $\fG(C)$.  The advantage of empolying the model category $\PC(\sK)$ is that it is a cartesian closed model category.  Therefore, for any two objects $A$ and $B$ in $\PC(\sK)$, there exists an internal Hom object $\uHom(A,B)$ in $\PC(\sK)$ and thus an internal Hom object $\RHom(A,B)$ in the homotopy category of $\PC(\sK)$.  This is in contrast to the model category of simplicial categories which does not form an internal theory (analogous to the the theory of model categories itself).  Secondly, in contrast to some other approaches, the definition of an $\infty$-precategory is based on an inductive procedure which allows one to construct a cartesian closed model category of $(\infty,n)$-categories paving the way for further extensions of the theory presented here.  

We will denote by $\h\sK$ the homotopy category of $\sK$ obtained from $\sK$ by formally adjoining inverses to all weak equivalences.  If $C$ is an $\infty$-category, the \textit{homotopy category} of $C$ is the $\h\sK$-enriched category $\h C$ with the same set of objects and such that for any $x,y\in C$,  
\[ \Map_{\h C}(x,y)=[C(x,y)] \] 
where $[\bullet]:\sK\ra\h\sK$.  For objects $x_{0},\ldots,x_{n}\in C$, composition $\Map_{\h C}(x_{0},x_{1})\times\ldots\times\Map_{\h C}(x_{n-1},x_{n})\ra\Map_{\h C}(x_{0},x_{n})$ is given by composing the inverse of the weak equivalence $C(x_{0},\ldots,x_{n})\ra C(x_{0},x_{1})\times\ldots\times C(x_{n-1},x_{n})$ with the map $C(x_{0},\ldots,x_{n})\ra C(x_{0},x_{n})$ and applying the functor $[\bullet]$.  We obtain in this way a functor $\h:\PC(\sK)\ra\Cat(\h\sK)$.  

A map $F:C\ra D$ between $\infty$-categories is said to be an equivalence if the induced functor $\h F:\h C\ra\h D$ is an equivalence of $\h\sK$-enriched categories, ie.
\begin{itemize}
\item For every $x,y\in C$, the map $C(x,y)\ra D(F(x),F(y))$ is a weak equivalence in $\sK$.
\item Every $y\in D$ is equivalent to $F(x)$ in the homotopy category $\h D$ for some $x\in C$.
\end{itemize}
A functor between $\infty$-categories satisfying these two conditions is said to be \textit{fully faithful} and \textit{essentially surjective} respectively.  

Model categories provide a very powerful tool for proving results in the theory of $\infty$-categories.  Apart from being the natural setting to undertake comparison results as mentioned above, model categories themselves can be used to model $\infty$-categories.  The construction taking a model category to an $\infty$-category is called \textit{localisation}.  In fact any $\infty$-category which is presentable in an appropriately defined sense is equivalent to the localisation of a combinatorial simplicial model category.  Moreover, any $\infty$-category can be fully embedded into the localisation of a model category.  

\begin{dfn}\label{localisationdfn}
Let $C$ be an $\infty$-category and $S$ a set of arrows in $C$.  A \textit{localisation} of $C$ along $S$ is a pair $(L_{S}C,l)$ where $L_{S}C$ is an $\infty$-category and $l:C\ra L_{S}C$ is a functor such that the following universal property is satisfied: for any $\infty$-category $D$, the induced map
\[   \RHom(L_{S}C,D)\ra\RHom(C,D)  \]
is fully faithful and its essential image consists of those functors $F:C\ra D$ which send each arrow in $S$ to an equivalence in $D$.
\end{dfn}

We will often refer to a localisation $(L_{S}C,l)$ of $C$ along $S$ as simply $L_{S}C$.  An explicit model for the localisation $({L}_{S}C,l)$  is given by composing the homotopy pushout diagram 
\[  (S\times\widetilde{[1]}) \coprod_{S\times [1]}^h  C  \] 
in $\PC(\sK)$ with a fibrant replacement functor where $\widetilde{[1]}$ is the groupoid generated by one isomorphism $\{0\xras 1\}$ (see Section 8.2 of \cite{tdg} for this existence result).  It follows that 
\[  \h(L_{S}C)\ra S^{-1}(\h C)  \]
is an equivalence of categories where $S^{-1}(\h C)$ is the category obtained by formally inverting the elements of $S$.  

Every category $C$ can be regarded as an $\infty$-category by considering the set $C(x,y)$ for two objects $x,y\in C$ as a discrete simplicial set (a simplicial category and thus an $\infty$-category).  More generally we may consider a pair $(C,S)$ consisting of a category $C$ together with a set of morphisms $S$ of $C$ and construct the localised category $S^{-1}C$.  This procedure can be refined using the simplicial localisation construction of Dwyer and Kan \cite{DK}.  The \textit{simplicial localisation} $L^{DK}_{S}(C)$ of the pair $(C,S)$ has the property that there exists a natural isomorphism $\pi_{0}L^{DK}_{S}(C)\simeq S^{-1}C$ showing that in general, $L^{DK}_{S}(C)$ contains higher homotopical information not encoded in $S^{-1}C$.  If $C$ is a category, then $L_{S}C$ is an $\infty$-category and $L_{S}C\ra L^{DK}_{S}(C)$ is an equivalence of $\infty$-categories \cite{DK}.  When $\sM$ is a model category we will let $L(\sM):=L_{W}\sM$ be the localisation of $\sM$ along the set of weak equivalences $W$ of $\sM$.  Thus $\h L(\sM)\ra\h\sM$ is an equivalence of categories.  

Let $\sM$ be an excellent model category (Definition A.3.2.16 of \cite{Lu}) and $\sA$ a combinatorial $\sM$-enriched model category.  Let $C$ be a $\sM$-enriched category and $\sA^{C}$ be endowed with the projective model structure.  Then there exists an equivalence
\[  \RHom(C,L(\sA))\ra L(\sA^{C})  \]
of $\sM$-enriched categories.  This is called the \textit{strictification theorem}.

\begin{ex}\label{strictification}
Let $A$ be an $\infty$-precategory, $D$ a simplicial category and $\fF(A)\ra D$ an equivalence where $\fF$ is the left adjoint to the fully faithful functor $\fG$.  Then the induced map
\[   \RHom(A,L(\PC(\sK)))\ra L(\PC(\sK)^{D})  \]
is an equivalence of $\infty$-categories.
\end{ex}

\begin{dfn}\label{inftyncatdef}
We denote by $\Catinf:=L(\PC(\sK))$ the $\infty$-category of $\infty$-categories and $\cS:=L(\sK)$ the $\infty$-category of $\inftyz$-categories.
\end{dfn}

The $\infty$-category $\cS$ of spaces will play an important role in the remainder of the text fulfilling an analogous role as the category of sets does in ordinary category theory.  Let $X$ be an $\infty$-category and consider the endofunctor
\[  \Pr_X:   A    \mapsto\RHom(A^{op},X)      \]            
between the $\infty$-category of $\infty$-categories.  The $\infty$-category $\Pr_{X}(A)$ will be called the $\infty$-category of \textit{$X$-valued prestacks} on $A$.  When $A$ is an $\infty$-precategory and $X$ is the $\infty$-category $\cS$ of $(\infty,0)$-categories, we write $\Pr(A)$ for $\Pr_{\cS}(A)$ and refer to $\Pr(A)$ as the $\infty$-category of \textit{prestacks} on $A$.  This $\infty$-category will also be denoted $A^{\wedge}$.  

Let $A$ be an $\infty$-precategory.  Then we can replace $A$ by a simplicial category $C:=\fF(A)$.  Let $C^{op}\times C\ra\sK$ be the natural $\sK$-enriched bifunctor.  By adjunction this gives a map $C\ra\sK^{C^{op}}$ where the right hand side is equivalent to $A^{\wedge}$ by the strictification theorem.  We will refer to the composition
\[    A\simeq C\ra\sK^{C^{op}}\simeq A^{\wedge},  \]
which is well defined in $\h\PC(\sK)$, as the \textit{Yoneda embedding}.  The $\infty$-categorical Yoneda Lemma then states that the Yoneda embedding $A\ra\Pr(A)$ is fully faithful.  By Proposition 5.1.3.2 of \cite{Lu}, the Yoneda embedding preserves small limits.  If $C$ be an $\infty$-category, a prestack $F$ in $\Pr(C)$ is said to be \textit{representable} if it lies in the essential image of the Yoneda embedding $C\ra\Pr(C)$.  

\begin{prop}\label{inftyintomodel}
Let $C$ be an $\infty$-category.  Then there exists a simplicial model category $\sA$ and a fully faithful map
$C\ra L\sA$
of $\infty$-categories.
\end{prop}

\begin{proof}
Let $D:=\fF(C)$ be a strict model for $C$.  Then the proposition follows from the composition
\[   C\xra{y}\RHom(C^{op},L(\sK))\xras\RHom(D^{op},L(\sK))\xras L(\sK^{D^{op}})  \]
using the fully faithful $\infty$-Yoneda lemma and the strictification theorem.  We conclude by setting $\sA:=\sK^{D^{op}}$.
\end{proof}

We will use this property to characterise $\infty$-categories having special properties by placing natural conditions on the model category $\sA$ and asking that the fully faithful map $C\ra L(\sA)$ be an equivalence.  Our first example is the following.

\begin{dfn}
An $\infty$-category $C$ is said to be \textit{presentable} if it is equivalent to the localisation of a combinatorial simplicial model category.
\end{dfn}
  
Let $\Catinf^{p}$ denote the full subcategory of $\Catinf$ spanned by the presentable $\infty$-categories and colimit preserving functors.  If $C$ is a presentable $\infty$-category and $A$ is an $\infty$-precategory then the $\infty$-category $\RHom(A,C)$ is presentable.  In particular, the $\infty$-category of prestacks $\Pr(A)$ is presentable.  

Let $C$ be an $\infty$-category and $S$ a set of arrows of $S$.  In the setting of presentable $\infty$-categories, the theory of localisations has a simple characterisation~: $L_{S}C$ can be identified with a full subcategory of $C$.  An object $x$ in $C$ is said to be \textit{$S$-local} if for every arrow $f:y\ra z$ in $S$, the induced map $C(z,x)\ra C(y,z)$ is an equivalence in $\cS$.  An arrow $f:x\ra y$ in $C$ is said to be an \textit{$S$-equivalence} if for every $S$-local object $z$ in $C$, the induced map $C(y,z)\ra C(x,z)$ is an equivalence in $\cS$.  

Let $C$ be a presentable $\infty$-category and $(L_{S}C,l)$ a localisation of $C$ in the $\infty$-category $\Catinf^{p}$.  Then an object $x$ in $C$ is $S$-local if and only if it belongs to $L_{S}C$.  Furthermore, every element of $S$ is an $S$-equivalence in $C$.  A localisation in the setting of presentable $\infty$-categories will be called a \textit{Bousfield localisation} to distinguish from the more general localisation of Definition~\ref{localisationdfn}.

An $\infty$-category admits all finite limits if and only if it admits pullbacks and a final object.  A functor between $\infty$-categories preserves finite limits if and only if it preserves pullbacks and final objects.  An analogous statement holds for finite colimits by passing to the opposite $\infty$-category.  Let $F:C\ra D$ be a functor between $\infty$-categories.  If $C$ admits finite limits then $F$ is said to be \textit{left exact} if it preserves finite limits.  If $C$ admits finite colimits then $F$ is said to be \textit{right exact} if it preserves finite colimits.  It is said to be \textit{exact} if it is both left and right exact.  We denote by $\uHom^{\lex}(C,D)$ (resp. $\uHom^{\rex}(C,D)$) the full subcategory of $\uHom(C,D)$ spanned by the left exact (resp. right exact) functors.

When $\sA$ is a combinatorial simplicial model category, then the existence of homotopy limits and colimits in $\sA$ ensures the existence of limits and colimits in the $\infty$-category $L(\sA)$.  An $\infty$-category is said to be \textit{complete} if it admits all (small) limits and \textit{cocomplete} if it admits all (small) colimits.  It is said to be \textit{bicomplete} if it is both complete and cocomplete.  Let $F:C\ra D$ be a functor between $\infty$-categories.  If $C$ admits small limits then $F$ is said to be \textit{continuous} if it preserves small limits.  If $C$ admits small colimits then $F$ is said to be \textit{cocontinuous} if it preserves small colimits.  It is said to be \textit{bicontinuous} if it is both continuous and cocontinuous.  We denote by $\uHom^{\tu{ct}}(C,D)$ (resp. $\uHom^{\tu{coct}}(C,D)$) the full subcategory of $\uHom(C,D)$ spanned by the continuous (resp. cocontinuous) functors.

The following general theory of limits and colimits in an $\infty$-category is contained in \cite{Lu}.  Let $\cA$ be a collection of $\infty$-precategories.  An $\infty$-category $C$ is said to admit $\cA$-indexed colimits if it admits colimits along all diagrams indexed by elements in $\cA$.  A functor $F:C\ra D$ is said to preserve $\cA$-indexed colimits if it preserves colimits along all diagrams indexed by elements in $\cA$.  We denote by $\RHom_{\cA}(C,D)$ the full subcategory of $\RHom(C,D)$ spanned by those functors which preserve $\cA$-indexed colimits.  If $C$ is an $\infty$-category and $D$ an $\infty$-category which admits $\cA$-indexed colimits then any functor functor $F:C\ra D$ can be extended, essentially uniquely, to an $\cA$-colimit preserving functor $G:\Pr^{\cA}(C)\ra D$ where $\Pr^{\cA}(C)$ is an $\infty$-category admitting $\cA$-indexed colimits. 

More precisely, by Proposition 5.3.6.2 of \cite{Lu}, there exists an $\infty$-category $\Pr^{\cA}(C)$ admitting $\cA$-indexed colimits and a fully faithful functor $y:C\ra\Pr^{\cA}(C)$ satisfying the following universal property~: for any $\infty$-category $D$ admitting $\cA$-indexed colimits, composition with $y$ induces an equivalence 
 \[  \RHom_{\cA}(\Pr^{\cA}(C),D)\ra\RHom(C,D)  \]
of $\infty$-categories.

When $\cA$ is the collection of all $\infty$-precategories, the $\infty$-category $\Pr^{\cA}(C)$ is identified with the $\infty$-category of prestacks.  Thus composition with the Yoneda embedding $A\ra\Pr(A)$ induces an equivalence 
\[   \RHom^{\tu{coct}}(\Pr(A),C)\ra\RHom(A,C)  \] 
of $\infty$-categories.  

Given an $\infty$-category $C$, we may also define other $\infty$-categories, thought of informally as the $\infty$-categories associated to $C$ by formally adding colimits of type $\cA$, using this universal property.  If $C$ is an $\infty$-category then the $\infty$-category of prestacks on $C$ is freely generated under small colimits by the image of the Yoneda embedding.  The $\infty$-category of ind-objects of $C$ is then the smallest full subcategory of this $\infty$-category of prestacks which contains the image of the Yoneda embedding and is stable under filtered colimits.  Thus it is freely generated under filtered colimits by $C$.   

\begin{dfn}
Let $C$ be an $\infty$-category and $\cA$ the class of all small filtered $\infty$-precategories.  Then the $\infty$-category of \textit{ind-objects} of $C$ is given by $\Ind(C):=\Pr^{\cA}(C)$.
\end{dfn}

Let $\cA$ be the class of all small $\kappa$-filtered $\infty$-precategories.  The $\infty$-category of ind-objects of $C$ admits the following characterisation~: the objects of $\Ind(C)$ are functors $I\ra C$ where $I\in\cA$ and given two objects $F:I\ra C$ and $G:J\ra C$ in $\Ind(C)$, the mapping space is given by 
\[      \Map(F,G)=\underset{i\in I}{\lim}~\underset{j\in J}{\colim}~C(F(i),G(j)).  \]
By Proposition 5.3.5.14 of \cite{Lu}, the Yoneda embedding $y:C\ra\Ind(C)$ taking $x$ to the functor $x:*\ra C$ preserves all small colimits which exist in $C$.  The essential image of $y$ consists of objects $x$ in $C$ such that the corepresentable functor
\[  C(x,\bullet):C\ra\cS   \] 
preserves filtered colimits.  Such objects are said to be \textit{compact} in $C$.  Let $C^{\tu{cpt}}$ denote the full subcategory of $C$ spanned by the compact objects.  

An $\infty$-category $C$ is said to be \textit{closed} if every diagram in $C$ indexed by a small $\infty$-precategory admits a colimit in $C$.  Clearly, 
an $\infty$-category $C$ is equivalent to $\Ind(D)$ for some small $\infty$-category $D$ if and only if the $\infty$-category $C$ is closed and has a small subcategory $D$ consisting of compact objects such that every object of $C$ is a filtered colimit of objects of $D$.  This motivates the following.

\begin{dfn}
An $\infty$-category $C$ is said to be \textit{accessible} if there exists a small $\infty$-category $D$ such that
\[   C\ra\Ind(D)  \]
is an equivalence of $\infty$-categories.  
\end{dfn}

Let $C$ be a presentable $\infty$-category.  By Proposition 5.5.2.2 of \cite{Lu}, a functor $F:C^{op}\ra\cS$ is representable if and only if it preserves small limits.   Similarly, a functor $F:C\ra\cS$ is corepresentable if and only if it is accessible and preserves small limits (Proposition 5.5.2.7 of \cite{Lu}).  An important ramification of this is the adjoint functor theorem which states that if $C$ and $D$ are presentable $\infty$-categories then a functor $F:C\ra D$ admits a right adjoint if and only if it preserves small colimits and admits a left adjoint if and only if it is accessible and preserves small limits (see Corollary 5.5.2.9 of \cite{Lu}).

\section{Symmetric monoidal structures}\label{monoidal}

In this section we define symmetric monoidal structures in the setting of $\infty$-categories.  References include \cite{L1} and \cite{TV3}.  In particular we define the $\infty$-category of presentable symmetric monoidal $\infty$-categories and the symmetric monoidal $\infty$-category of modules over a commutative monoid object in a symmetric monoidal $\infty$-category.  These notions are used to define linear $\infty$-categories which are necessary for the definition of a Tannakian $\infty$-category in Section~\ref{tannakadualityforinftycategories}.

We define a symmetric monoidal $\infty$-category using the language of cofibered $\infty$-categories (see Section 3.2 of \cite{Lu} and Section 1.3 of \cite{TV3}). Let $\Gamma$ denote the category of pointed finite ordinals and point preserving maps.  This is equivalent to the category of all linearly ordered finite sets with a distinguished point $*$.  We denote the pointed ordinal $n\coprod \{*\}$ by $\<n\>$.  The category $\Gamma$ is a monoidal category with monoidal structure $(\Gamma,\vee,\<0\>)$.  

An arrow $f:\<n\>\ra\<m\>$ in $\Gamma$ is said to be \textit{inert} (resp. \textit{semi-inert}) if $f^{-1}\{j\}=\{i\}$ (resp. $f^{-1}\{j\}\in\{\emptyset,\{i\}\}$) for all $j\in\<m\>{-}*$.  It is said to be \textit{null} if $f(i)=*$ for all $i\in\<n\>$.  It is said to be \textit{active} if $f^{-1}\{*\}=\{*\}$.  Every arrow $f$ in $\Gamma$ admits a factorisation $f=f''\circ f'$ by an inert arrow $f'$ followed by an active arrow $f''$.  This factorisation is unique up to (unique) isomorphism. 

Consider an object $p:A\ra\Gamma$ in the category $\PC(\sK)_{/\Gamma}$.  An arrow $f$ in $A(a,b)$ is said to be \textit{p-cocartesian} if for all $c\in A$, the induced morphism
\[ A(b,c)\ra A(a,c)\times_{I(p(a),p(c))}I(p(b),p(c)) \]
is a weak equivalence in $\sK$.  An object $p:A\ra\Gamma$ in the category $\PC(\sK)_{/\Gamma}$ is said to be a \textit{cofibered} $\infty$-category if for every arrow $u:\<n\>\ra \<m\>$ in $\Gamma$ and every object $a$ in $A$ with $p(a)=\<n\>$, there exists a $p$-cocartesian arrow $f$ such that $p(f)$ is isomorphic to $u$ in the undercategory $\Gamma_{\<n\>/}$.  A morphism in the homotopy category of $\PC(\sK)_{/\Gamma}$ is said to be \textit{cocartesian} if it preserves cocartesian arrows.  The (non-full) subcategory of $\h(\PC(\sK)_{/I})$ consisting of cofibered objects and cocartesian morphisms will be denoted by $\h(\PC(\sK)_{/I})^{cc}$.  An important observation is that the condition to be cofibered is stable by equivalences in $\PC(\sK)$.  

Let $A$ be an $\infty$-precategory and $A_{\<n\>}$ the fiber of the map $p:A\ra\Gamma$ at $\<n\>\in\Gamma$.  For each $n\geq 1$ and $0< i\leq n$, consider the $n$ pointed maps $p_i:\<n\>\ra \<1\>$ in $\Gamma$ given by $p_i(j)=\{j\}$ if $i=j$ and $p_i(j)=*$ otherwise.  These induce natural maps $(p_i)_!:A_{\<n\>}\ra A_{\<1\>}$.

\begin{dfn}
Let $A$ be a $\infty$-precategory.  A \textit{symmetric monoidal $\infty$-category} is a cofibered object $p:A\ra\Gamma$ in $\PC(\sK)_{/\Gamma}$ such that 
\[     A_{\<n\>}\xra{\prod_{i}(p_{i})_{!}}(A_{\<1\>})^n      \] 
is an equivalence for each $n\geq 0$.
\end{dfn}

We let $A_{\<1\>}$ be the underlying $\infty$-category of $A$.  We will often abuse notation by referring to a symmetric monoidal $\infty$-category $p:A\ra\Gamma$ as simply $A$.  By Proposition 1.4 of \cite{TV3} there exists an equivalence 
\[  \h(\PC(\sK)^I)\ra \h(\PC(\sK)_{/I})^{cc}  \]
of categories.  Thus a symmetric monoidal $\infty$-category coincides with the definition as a functor into $\PC(\sK)$.  

Let $p:A\ra\Gamma$ be a symmetric monoidal $\infty$-category.  Then an arrow $f$ in $A$ is said to be \textit{$p$-inert} if $f$ is a cocartesian arrow in $A$ such that $p(f)$ is inert in $\Gamma$.

\begin{dfn}
Let $p:A\ra\Gamma$ and $q:B\ra\Gamma$ be two symmetric monoidal $\infty$-categories.  A functor $F:A\ra B$ is said to be \textit{symmetric monoidal} if the diagram
\begin{diagram}
A& &\rTo{F}& & B\\
 &\rdTo_{p}&         &\ldTo_{q}&\\
      && \Gamma &&
\end{diagram}
commutes and $F$ carries $p$-cocartesian arrows to $q$-cocartesian arrows.  It is said to be \textit{lax symmetric monoidal} if $F$ carries $p$-inert arrows to $q$-cocartesian arrows.  
\end{dfn}

Let $\RHom^{\otimes}_{\Gamma}(A,B)$ denote the full subcategory of $\RHom_{\Gamma}(A,B)$ spanned by the symmetric monoidal functors.  Let $\Catinf^{\otimes}$ denote the $\infty$-category consisting of symmetric monoidal $\infty$-categories and symmetric monoidal functors.  Likewise, let $\RHom^{\lax}_{\Gamma}(A,B)$ denote the full subcategory of $\RHom_{\Gamma}(A,B)$ spanned by the lax symmetric monoidal functors.

\begin{dfn}
Let $p:C\ra\Gamma$ be a symmetric monoidal $\infty$-category.  A \textit{commutative monoid object} in $C$ is a lax symmetric monoidal section of $p$ (where the identity map $\id_\Gamma$ endows the trivial category $\<0\>$ with a symmetric monoidal structure).  
\end{dfn}

The $\infty$-category of commutative monoid objects in $C$ will be denoted 
\[   \CMon(C):=\RHom_{\Gamma}^{\lax}(\Gamma,C). \]  
A \textit{commutative comonoid object} in $C$ is a commutative monoid object in $C^{op}$.  The $\infty$-category $\CMon(C)$ has an initial object $A$ such that the unit map $1_{C_{\<1\>}}\ra A(\<1\>)$ is an equivalence in $C_{\<1\>}$ (Corollary 3.2.1.9 of \cite{L1}).  For a commutative (co)monoid object $A$, we will sometimes use the notation $A_{n}:=A(\<n\>)$.  

\begin{ex}\label{pointwisemonoidalstructure}
Let $p:C\ra\Gamma$ be a symmetric monoidal $\infty$-category and $A$ an $\infty$-precategory.  Then the $\infty$-category $\RHom(A,C_{\<1\>})$ inherits the structure of a symmetric monoidal $\infty$-category ${\RHom}(A,C)\ra\Gamma$ called the \textit{pointwise symmetric monoidal structure} where we define
\[     {\RHom}(A,C)_{\<n\>}:=\RHom(A,C)\times_{\RHom(A,\Gamma)}\{\<n\>\}   \] 
giving ${\RHom}(A,C)_{\<n\>}\simeq\RHom(A,C_{\<n\>})$.  As a result, there exists an equivalence 
\[   \CMon({\RHom}(A,C))\ra\RHom(A,\CMon(C)) \] 
of $\infty$-categories.
\end{ex}

\begin{ex}\label{subsymmmonoidal}
Let $C$ be a symmetric monoidal $\infty$-category and $D_{\<1\>}$ a full subcategory of $C_{\<1\>}$.  Assume that for every equivalence $x\ra y$ in $C_{\<1\>}$, if $y\in D_{\<1\>}$ then $x\in D_{\<1\>}$.  Define a subcategory $D$ of $C$ by letting an object $x\in C_{\<n\>}$ belong to $D$ if and only if the image under $\prod_{i}(p_{i})_{*}:C_{\<n\>}\ra (C_{\<1\>})^n$ belongs to $(D_{\<1\>})^n$.  Then it is clear that the restriction map $D\ra\Gamma$ is a symmetric monoidal $\infty$-category if $D$ is closed under tensor products and contains the unit object of $C$.
\end{ex}

\begin{ex}\label{cmonsymmetricmonoidal}
Let $p:C\ra\Gamma$ be a symmetric monoidal $\infty$-category.  Then by Example 3.2.2.4 of \cite{L1}, there exists a symmetric monoidal structure on the $\infty$-category $\CMon(C)$ of commutative monoid objects in $C$ induced by that on $C$.  Moreover, by Proposition 3.2.4.7 of \textit{loc. cit.}, the tensor product of commutative monoid objects corresponds to the coproduct.  
\end{ex}

\begin{ex}
If $C$ admits finite colimits and the tensor product bifunctor preserves finite colimits seperately in each variable then the $\infty$-category $\Ind(C)$ of ind-objects of $C$ admits a symmetric monoidal structure which is characterised, up to symmetric monoidal equivalence, by the properties that the tensor product bifunctor $\otimes:\Ind(C)\times\Ind(C)\ra\Ind(C)$ preserves small colimits seperately in each variable and that the Yoneda embedding $C\ra\Ind(C)$ can be extended to a symmetric monoidal functor.  This statement follows from a more general statement for a symmetric monoidal $\infty$-category admitting colimits indexed by an arbitrary collection of $\infty$-precategories.  See Proposition 6.3.1.10 of \cite{L1} for a precise statement.
\end{ex}

Let $C$ be an $\infty$-category.  There exists a natural analogue of the localisation of $C$ as defined in Definition~\ref{localisationdfn} when $C$ is endowed with a symmetric monoidal structure whose underlying $\infty$-category is $L_SC$.  We will abuse notation by denoting the resulting symmetric monoidal $\infty$-category by $L_SC$.  The existence result follows from Proposition 4.1.3.4 of \cite{L1}.  When $\sM$ is a symmetric monoidal model category we will define $L(\sM):=L_W(\sM^c)$ (the monoidal product on $\sM^c$ preserves the weak equivalences $W$ in $\sM^c$).

\begin{ex}\label{cartesianstructureoncatinfty}
The model category $\PC(\sK)$ of $\infty$-categories is a symmetric monoidal simplicial model category for the cartesian product.  Thus ${\Cat}_{\infty}:=L(\PC(\sK))$ is a symmetric monoidal $\infty$-category.  Explicitly, the cofibered $\infty$-category ${\Cat}_{\infty}$ is given as follows~:
\begin{itemize}
\item The objects are pairs $(\<n\>,(C_0,\ldots,C_n))$ where $\<n\>$ is an object of $\Gamma$ and each $C_i$ is a fibrant $\infty$-precategory.
\item A map between two objects $(\<n\>,C_\bullet)$ and $(\<m\>,D_\bullet)$ is a map $u:\<n\>\ra\<m\>$ in $\Gamma$ together with a collection of functors $\prod_{u(i)=j}C_{i}\ra D_{j}$.
\end{itemize}
\end{ex}

By Proposition 4.4.4.6 and Theorem 4.4.4.7 and of \cite{L1}, if $\sM$ is a combinatorial symmetric monoidal model category with the conditions that $\sM$ is left proper, the class of cofibrations in $\sM$ is generated by cofibrations between cofibrant objects, $\sM$ satisfies the monoid axiom and every cofibration in $\sM$ is a power cofibration (see Definition 4.4.4.2 of \textit{loc. cit.}) then $\CMon(\sM)$ admits a combinatorial model structure such that the map
\[  L(\CMon(\sM))\ra\CMon(L(\sM))  \]
is an equivalence of $\infty$-categories.  When the $\infty$-category $\Catinf$ is equipped with the cartesian monoidal structure then
\[   \CMon({\Cat}_{\infty})\ra\Catinf^{\otimes} \]
is an equivalence of $\infty$-categories.  This follows from the more general equivalence of Proposition 2.4.2.5 of \cite{L1}.

Let $C$ be a symmetric monoidal $\infty$-category.  A symmetric monoidal structure on $C$ is said to be \textit{compatible with countable colimits} if for any simplicial set $A$ with only countably many simplices, the $\infty$-category $C$ admits $A$-indexed colimits and for any $x$ in $C$, the functor $\bullet\otimes x:C\ra C$ preserves these colimits.  If $C$ is compatible with countable colimits then the forgetful functor $\CMon(C)\ra C$ admits a left adjoint 
\[     \tu{Fr}:C\ra\CMon(C)  \] 
which we refer to as the \textit{free functor}.  A precise statement can be found in Corollary 3.1.3.5 of \cite{L1}.  If $C$ is equivalent to $L(\sM)$ for $\sM$ a symmetric monoidal model category then $\tu{Fr}$ is equivalent to a functor $\tu{Fr}:L(\sM)\ra L(\CMon(\sM))$ so   
\[       \tu{Fr}(x)=\bb{L}\Sym(x)      \]
where $\Sym(x):=\coprod_{n\geq 0}x^{\otimes n}/\Sigma_{n}$.

\begin{dfn}\label{otensored}
Let $p:D\ra\Gamma$ be a symmetric monoidal $\infty$-category.  An $\infty$-category $C$ is said to be \textit{tensored} over $D$ if there exists a map $F:C\ra D$ in $\PC(\sK)$ such that: 
\begin{enumerate}
\item The composition $(p\circ F):C\ra\Gamma$ is cofibered in $\PC(\sK)_{/\Gamma}$.
\item The map $F$ carries $(p\circ F)$-cocartesian arrows of $C$ to $p$-cocartesian arrows of $D$.
\item For each $n\geq 0$, the inclusion $\{n\}\subseteq \<n\>$ induces an equivalence $C_{\<n\>}\ra D_{\<n\>}\times C_{\{n\}}$ of $\infty$-categories.
\end{enumerate}
\end{dfn}

Let $C$ be an $\infty$-category tensored over $D$.  We will refer to the fiber $C_{\<0\>}$ as the underlying $\infty$-category of $C$ and by abuse, also denote it by $C$.  We obtain a natural diagram
\[  C_{\{0\}}\leftarrow C_{\<1\>}\xras D_{\<1\>}\times C_{\{1\}} \] 
which induces a symmetric monoidal bifunctor $\otimes:D_{\<1\>}\times C_{\<0\>}\ra C_{\<0\>}$, together with its higher symmetric monoidal structure, which is well defined up to homotopy.  

\begin{dfn}\label{morphexp}
Let $A$ be a symmetric monoidal $\infty$-category and $C$ a $\infty$-category tensored over $A$.  Let $a\in A$ and $x,y\in C$.  Then $C$ is said to be \textit{enriched} over $A$ if the functor $A^{op}\ra\cS$ given by  
\[ a\mapsto C(a\otimes x,y)  \]
is representable for all $x,y\in C$.  The representing object will be denoted $\Mor(x,y)$ and called the \textit{morphism object} of $x$ and $y$.  
\end{dfn}            

It follows directly from Definition~\ref{morphexp} that morphism objects are characterised by the following universal property: there exists a map ${ev}:\Mor(x,y)\otimes x\ra y$ such that composition with $ev$ yields an equivalence
\[         A(a,\Mor(x,y))\ra C(a\otimes x,y)    \]
of $\inftyz$-categories.  The composition $\Mor(y,z)\otimes\Mor(x,y)\otimes x\xra{ev}\Mor(y,z)\otimes y\xra{ev} z$ yields a composition map
\[   \Mor(y,z)\otimes\Mor(x,y)\ra\Mor(x,z)   \]
and the chain of equivalences 
\[   A(b,\Mor(a,\Mor(x,y)))\simeq A(b\otimes a,\Mor(x,y))\simeq C(b\otimes a\otimes x,y)\simeq A(b,\Mor(a\otimes x,y))  \] 
yields an equivalence
\[     \Mor(a,\Mor(x,y))\ra\Mor(a\otimes x,y)   \]
of morphism objects.

\begin{ex}
Let $C$ be a symmetric monoidal $\infty$-category.  Then the symmetric monoidal product $\otimes:C\times C\ra C$ endows $C$ with the structure of an $\infty$-category tensored over itself.  If it is furthermore enriched, then the morphism object $\Mor(c,d)$ is just the internal Hom object $\uHom(c,d)$ in $C$.
\end{ex}

Let $A$ be a monoidal $\infty$-category and $C$ a $\infty$-category tensored over $A$.  Suppose further that $C$ and $A$ are presentable $\infty$-categories.  Let $a=\colim_{i}a_{i}$ be an object in $A$.  A prestack is representable if and only if it preserves small limits.  Therefore, the $\infty$-category $C$ is enriched over $A$ if $C(\colim_{i}a_{i}\otimes x,y)\simeq\lim_{i}C(a_{i}\otimes x,y)$.   By assumption, $C(\colim_{i}a_{i}\otimes x,y)\simeq C(\colim_{i}(a_{i}\otimes x),y)$ which is naturally equivalent to $\lim_{i}C(a_{i}\otimes x,y)$.  Thus the $\infty$-category $C$ is enriched over $A$ if the functor 
\[     \bullet\otimes x:A\ra C     \] 
preserves small colimits for all $x\in C$.               
  
We will now define the $\infty$-category of module objects in a symmetric monoidal $\infty$-category.  A more detailed and general exposition can be found in Section 3.3.3 of \cite{L1}.  

\begin{notn}
Let $K^{si}$ denote the full subcategory of $\RHom(\<1\>,\Gamma)$ spanned by the semi-inert arrows.  Let $K^{null}$ denote the full subcategory of $K^{si}$ spanned by the null arrows.  There are two natural maps $ev_i:K^{\bullet}\ra\Gamma$ given by evaluation on $i\in\{0,1\}$.  A morphism in $K^{\bullet}$ is said to be \textit{inert} if its images under $e_0$ and $e_1$ are inert in $\Gamma$.  
\end{notn}

Let $p:C\ra\Gamma$ be a symmetric monoidal $\infty$-category.  We define an $\infty$-precategory $\cM(C)$ through the following universal property: for every $\infty$-precategory $A$ equipped with a map $A\ra\Gamma$ the map
\[ \RHom_{\Gamma}(A,\cM(C))\ra\RHom_{\RHom(\{1\},\Gamma)}(A\times_{\RHom(\{0\},\Gamma)}K^{si},C) \]
is an equivalence.  Thus 
\[        \cM(C)_{\<n\>}\ra\RHom_{\RHom(\{1\},\Gamma)}(K^{si}_{\<n\>>},C)  \]
is an equivalence where $K^{si}_{\<n\>>}$ denotes the homotopy fiber of $K^{si}\ra\RHom(\{0\},\Gamma)$ at $\<n\>$.  An object of $\cM(C)_{\<1\>}$ is then given by a commutative diagram
\begin{diagram}
K^{si}_{\<1\>>}& &\rTo{F}& & C\\
 &\rdTo_{ev_{1}}&         &\ldTo_{p}&\\
                             &&\Gamma. &&
\end{diagram}
Let $\overline{\cM}(C)$ denote the full subcategory of $\cM(C)$ spanned by those vertices for which the functor $F$ preserves inert morphisms.  

Similarly we define an $\infty$-precategory $\cA(C)$ through the following universal property~: for every $\infty$-precategory $A$ equipped with a map $A\ra\Gamma$ the map
\[ \RHom_{\Gamma}(A,\cA(C))\ra\RHom_{\RHom(\{1\},\Gamma)}(A\times_{\RHom(\{0\},\Gamma)}K^{null},C) \]
is an equivalence.  Thus 
\[        \cA(C)_{\<n\>}\ra\RHom_{\RHom(\{1\},\Gamma)}(K^{null}_{\<n\>>},C)  \]
is an equivalence where $K^{null}_{\<n\>>}$ denotes the homotopy fiber of $K^{null}\ra\RHom(\{0\},\Gamma)$ at $\<n\>$.  Let $\overline{\cA}(C)$ denote the full subcategory of $\cA(C)$ spanned by those vertices for which the functor $F$ preserves inert morphisms.  We define
\[     {\Mod}_{R}(C):=\overline{\cM}(C)\times_{\overline{\cA}(C)}\{R\}.  \]
By Theorem 4.4.2.1 of \cite{L1}, if $C$ admits colimits of simplicial objects such that for every object $x$ in $C$ the functor $\bullet\otimes x$ preserves these colimits, then the projection 
\[  p:{\Mod}_{R}(C)\ra\Gamma  \]
is a symmetric monoidal $\infty$-category.  The unit object of $p$ is canonically equivalent to $R$.  The symmetric monoidal product is called the \textit{relative tensor product} and for two $R$-modules $M$ and $N$ will be denoted $M\otimes_{R}N$.  See Theorem 4.4.2.1 \cite{L1} for a detailed discussion of the relative tensor product functor.  

The above construction is functorial in $p$ and hence we obtain a functor
\[ \Mod(C):\CMon(C)\ra\Catinf^\otimes  \]
sending a commutative monoid object $R$ to ${\Mod}_R(C)$ and a map $f:R\ra Q$ to the functor $\bullet\otimes_{R}Q$.  Here $\bullet\otimes_{R}Q$ is the symmetric monoidal base change functor left adjoint to the forgetful functor $\Mod_Q(C)\ra\Mod_R(C)$.  

Let $C$ be a symmetric monoidal $\infty$-category and $R$ a commutative comonoid object in $C$.  By definition, the \textit{$\infty$-category of comodules} in $C$ over $R$ is
\[    \Comod_{R}(C):=\Mod_{R}(C^{op})^{op}.      \]

\begin{ex}\label{lmodmodl}
Many examples of $\infty$-categories of modules arise from the localisation of model categories of modules.  More precisely, let $\sM$ be a combinatorial symmetric monoidal model category and $R$ a commutative monoid object of $\sM$.  Assume that $\sM$ satisfies the monoid axiom \cite{SS2}.  Then the category $\Mod_{R}(\sM)$ of $R$-modules admits a combinatorial model structure where a map is a fibration if and only if it is a fibration in $\sM$ and a weak equivalence if and only if it is a weak equivalence in $\sM$.  If $\sM$ is endowed with a simplicial model structure, then $\Mod_{R}(\sM)$ is a simplicial model category.  If we further assume that $\sM$ satisfies the conditions furnishing an equivalence $s:L\CMon(\sM)\ra\CMon(L\sM)$, then the natural map
\[    L\Mod_{R}(\sM)\ra\Mod_{s(R)}(L\sM)  \]
is an equivalence of $\infty$-categories.  This follows from Theorem 4.3.3.17 of \cite{L1}.
\end{ex}

\begin{dfn}
Let $C$ be a symmetric monoidal $\infty$-category and $R$ a commutative monoid object in $C$.  A \textit{commutative $R$-algebra object} in $C$ is a commutative monoid object in the symmetric monoidal $\infty$-category ${\Mod}_{R}(C)$ of $R$-modules.
\end{dfn}

Let $\CAlg_{R}(C)$, or simply $\CAlg_{R}$, denote the $\infty$-category $\CMon({\Mod}_{R}(C))$ of commutative $R$-algebra objects in $C$.  It follows from Corollary 3.4.1.7 of \cite{L1} that if $C$ be a symmetric monoidal $\infty$-category and  $R$ is a commutative monoid object in $C$, then there exists an equivalence
\[ \theta_R:\CAlg_R(C)\ra\CMon(C)_{R/}  \]
of $\infty$-categories.  Furthermore, by Corollary 6.3.5.16 of \cite{L1}, there exists a natural fully faithful map 
\[  \CMon(C)\ra\CMon(\Mod_C({\Cat}_\infty))   \] 
and so the map
\[  \CMon(C)\ra\CMon({\Cat}_\infty)_{C/}  \]
given by composing this fully faithful map with $\theta_C$ is fully faithful. 

The $\infty$-category $\CAlg_{R}(C)$ inherits the structure of a symmetric monoidal $\infty$-category where the tensor product is given by the tensor product in $C$ (see Example~\ref{cmonsymmetricmonoidal}).  Furthermore, this tensor product coincides with the coproduct in the $\infty$-category of commutative $R$-algebras.  

Let $C$ be a symmetric monoidal $\infty$-category such that the symmetric monoidal product preserves (small) colimits separately in each variable and the fiber $C_{\<n\>}$ is a presentable $\infty$-category for all $n>0$.  In this case we will say that $C$ is a \textit{presentable} symmetric monoidal $\infty$-category.  If $C$ is a presentable symmetric monoidal $\infty$-category then $\CMon(C)$ is a presentable $\infty$-category.  This follows from Corollary 3.2.3.5 of \cite{L1}.  Moreover, if $C$ is a presentable symmetric monoidal $\infty$-category then the $\infty$-category ${\Mod}_R(C)$ is a presentable symmetric monoidal $\infty$-category by Theorem 3.4.4.2 of \cite{L1}.  Combining these two results, the $\infty$-category $\CAlg_{R}(C)$ is a presentable $\infty$-category.

There exists a symmetric monoidal structure on the $\infty$-category $\Catinf^{p}$ of presentable $\infty$-categories by Proposition 6.3.1.14 of \cite{L1}.  This is a subcategory of the symmetric monoidal $\infty$-category ${\Cat}_\infty$ which can be explicitly described as follows~: 
\begin{itemize}
\item The objects are pairs $(\<n\>,(C_0,\ldots,C_n))$ where $\<n\>$ is an object of $\Gamma$ and each $C_i$ is a presentable fibrant $\infty$-precategory.
\item A map between two objects $(\<n\>,C_\bullet)$ and $(\<m\>,D_\bullet)$ is a map $u:\<n\>\ra\<m\>$ in $\Gamma$ together with a collection of functors $\prod_{u(i)=j}C_{i}\ra D_{j}$ which preserve colimits seperately in each variable.
\end{itemize}

One can show (see \textit{loc. cit.}) that the unit object of $\Catinf^{p}$ with this symmetric monoidal structure is the $\infty$-category $\cS$ of spaces.  Let $\Catinf^{p,\otimes}$ denote the subcategory of $\Catinf^{\otimes}$ spanned by presentable symmetric monoidal $\infty$-categories whose monoidal bifunctor preserves colimits seperately in each variable and whose morphisms are colimit preserving symmetric monoidal functors.   Then we have an equivalence
\[     \CMon({\Cat}_\infty^{p})\ra\Catinf^{p,\otimes}  \]
of $\infty$-categories.  Thus a symmetric monoidal $\infty$-category $C$ belongs to $\CMon({\Cat}_\infty^{p})$ if and only if $C$ is presentable and the tensor product bifunctor $\otimes:C_{\<1\>}\times C_{\<1\>}\ra C_{\<1\>}$ preserves (small) colimits seperately in each variable.  

Let $C$ be a presentable symmetric monoidal $\infty$-category.  Then the $\infty$-category ${\Mod}_{R}(C)$ of $R$-modules in $C$ is a presentable symmetric monoidal $\infty$-category with a bicontinuous monoidal product.  Thus ${\Mod}_{R}(D)$ belongs to $\CMon({\Cat}_\infty^{p})$.  We can then make the following definition.

\begin{dfn}\label{rlinear}
Let $D$ be a presentable symmetric monoidal $\infty$-category.  A presentable $\infty$-category is said to be \textit{$R$-linear} if it is endowed with the structure of a ${\Mod}_{R}(D)$-module object in the symmetric monoidal $\infty$-category ${\Cat}_\infty^{p}$ of presentable $\infty$-categories.
\end{dfn}

The $\infty$-category of $R$-linear $\infty$-categories is given by $\Mod_{\Mod_{R}(D)}({\Cat}_{\infty}^{p})$.  Note that an $R$-linear $\infty$-category is a presentable $\infty$-category $C$ which is tensored over the $\infty$-category $\Mod_{R}(D)$.  Then the functor $\bullet\otimes x:\Mod_{R}(D)\ra C$ preserves colimits for all $x\in C$ owing to the monoidal structure on $\Catinf^{p}$.  Thus the presentable $\infty$-category $C$ is enriched over $\Mod_{R}(D)$ as expected.  Therefore there exists an equivalence
\[   \CMon({\Mod}_{\Mod_{R}}({\Cat}_{\infty}^{p}))\simeq\CMon({\Cat}_{\infty}^{p})_{\Mod_{R}/}. \]
of $\infty$-categories.  The term on the left hand side is the $\infty$-category of $R$-linear presentable symmetric monoidal $\infty$-categories and $R$-linear symmetric monoidal functors.   

\begin{ex}
Let $D$ be a presentable symmetric monoidal $\infty$-category.  Then ${\Mod}_{R}(D)$ is an $R$-linear $\infty$-category.
\end{ex}

\begin{ex}\label{klinear}
Let $k$ be a commutative ring and $C$ a presentable $k$-linear symmetric monoidal category (ie. a presentable symmetric monoidal category with a $\Mod_{k}(\Ab)$-module structure).  Then $L(C)$ is a $k$-linear $\infty$-category, ie. if the map $\Mod_{k}\ra C$ of categories endows $C$ with a $k$-linear structure then the map of $\infty$-categories $L(\Mod_{k})\ra L(C)$ endows $L(C)$ with a $k$-linear structure.  More generally, if $\sM$ is a $k$-linear symmetric monoidal model category, ie. a $\Mod_{k}(\Ab)$-enriched symmetric monoidal model category, then $L(\sM)$ is a $k$-linear $\infty$-category.
\end{ex}

\section{Spectral algebra}\label{spectralalgebra}

In this section we address the algebraic structures appearing in our Tannaka duality theorem.  We begin by reviewing the basic theory of stable $\infty$-categories.  A more detailed account can be found in \cite{L1}.  Stable $\infty$-categories replace the role of abelian categories in the $\infty$-categorical realm.  Indeed, after defining the $\infty$-category of spectra, together with its natural symmetric monoidal structure, we will see that every stable $\infty$-category is naturally enriched over spectra.  Endowing the $\infty$-category of spectra with its natural t-structure we see that its heart is equivalent to the category of abelian groups.  Since a commutative ring is simply a commutative monoid object in the category of abelian groups, it is natural to define a ``commutative derived ring", or $\Einfty$-ring, as a commutative monoid object in the $\infty$-category of spectra.  
  
Let $C$ be an $\infty$-category.  We call an object which is both initial and terminal in $C$ a \textit{zero object} and denote it by $0\in C$.  An $\infty$-category is said to be \textit{pointed} if it contains a zero object.

\begin{dfn}\label{stableinftycategory}
An $\infty$-category $C$ is said to be \textit{stable} if it is pointed, admits finite limits and colimits and pullback and pushout squares coincide.
\end{dfn}   

Note that if a functor between stable $\infty$-categories is left or right exact it is automatically exact.  Let $\Catinf^{\perp}$ denote the full subcategory of $\Catinf$ spanned by stable $\infty$-categories and exact functors.  

Let $C$ be a pointed $\infty$-category and $f:x\ra y$ an arrow in $C$.  A \textit{kernel} of $f$ is a pullback $x\times_y0$ and a \textit{cokernel} of $f$ is a pushout $y\coprod_x0$.  They are uniquely determined up to equivalence in $C$.  A full subcategory of a stable $\infty$-category is said to be a \textit{stable subcategory} if it contains a zero object and is closed under the formation of kernels and cokernels.

The $\infty$-category $\Catinf^{\perp}$ admits all (small) limits and all (small) filtered colimits (Theorem 1.1.4.4 and Proposition 1.1.4.6 of \cite{L1}).  The structure of a stable $\infty$-category induces a heavy simplification of the nature of its limits and colimits~: if $\kappa$ is a regular cardinal, then a stable $\infty$-category has all $\kappa$-small limits (resp. colimits) if and only if it has $\kappa$-small products (resp. coproducts).  Furthermore, an exact functor between stable $\infty$-categories preserves $\kappa$-small limits (resp. colimits) if and only if it preserves $\kappa$-small products (resp. coproducts).

Let $C$ be a pointed $\infty$-category with finite limits.  The \textit{loop functor} $\Omega$ of $C$ is the endomorphism of $C$ given by 
\[  \Omega:x\mapsto 0\underset{x}{\times} 0.  \]
This functor admits a left adjoint 
\[   \Sigma:x\mapsto 0\coprod_{x} 0  \]
called the \textit{suspension functor}.  A pointed $\infty$-category is stable if and only if it admits finite limits and the loop functor is an equivalence of $\infty$-categories.  Likewise, it is stable if and only if it admits finite colimits and the suspension functor is an equivalence of $\infty$-categories.  This follows from Proposition 1.1.3.4 and Corollary 1.4.2.20 of \cite{L1}.  If $C$ is a stable $\infty$-category and $A$ is a $\infty$-precategory then the $\infty$-category $\RHom(A,C)$ is stable.  In particular, $\Ind_{\kappa}(C)$ is a stable $\infty$-category (see Proposition 1.1.3.6 of \cite{L1}).

\begin{ex}\label{stablemodelcategory}
A pointed, closed model category $\sM$ is said to be \textit{stable} if the adjunction $\Sigma\dashv\Omega$ is an equivalence in the homotopy category $\h\sM$.  Thus for any stable model category $\sM$, the $\infty$-category $L(\sM)$ is stable.  Moreover, if $\sM$ is a cofibrantly generated, proper, stable, simplicial model category with a set $P$ of compact generators, the authors in \cite{SS1} prove an equivalence between $\sM$ and the model category $\Mod_{\cE(P)}$ of modules over a certain spectral endomorphism category $\cE(P)$ (see Definition 3.7.5 of \textit{loc. cit.}).  We thus obtain an equivalence of stable $\infty$-categories $L(\sM)\ra L(\Mod_{\cE(P)})$.
\end{ex}

\begin{ex}\label{derivedinftycategory}
Let $A$ be an abelian category with enough projective objects and $C(A)$ the simplicial model category of chain complexes in $A$.  Then the \textit{derived} $\infty$-category $L(C(A))$ of $A$ is a stable $\infty$-category.  The homotopy category $\h L(C(A))$ can be identified with the derived category $D(A)$ of $A$.  Likewise, one can define the bounded (resp. bounded above, bounded below) derived $\infty$-category of $A$.  See Section 1.3.1 of \cite{L1} for more details.
\end{ex} 

Let $C$ be an $\infty$-category with finite limits and $\bb{Z}$ the linearly ordered set of integers which we consider as a filtered category.  Let $T$ be an endofunctor on $C$.  We construct the following endofunctor 
\[  \phi:\RHom(\bb{Z},C)\ra\RHom(\bb{Z},C)  \]
defined by $\phi(F)(n):=T(F(n+1))$. 

\begin{dfn}\label{spectrumobject}
Let $C$ be an $\infty$-category with finite limits and $T$ an endofunctor on $C$.  A \textit{$T$-spectrum object} of $C$ is a functor $F:\bb{Z}\ra C$ such that 
$F\ra\phi(F)$ is an equivalence in $\RHom(\bb{Z},C)$.
\end{dfn}

The $\infty$-category of $T$-spectrum objects in $C$, denoted $\Sp_{T}(C)$, is given by the homotopy pullback
\begin{diagram}
\Sp_{T}(C) &\rTo &\RHom(\bb{Z},C) \\
\dTo           &       &\dTo_{(\phi,\id)}\\
\RHom(\bb{Z},C) &\rTo^{d} &\RHom(\bb{Z},C)\times\RHom(\bb{Z},C)
\end{diagram}  
where $d$ denotes the diagonal map.  The equivalence $d(G)\simeq(\phi,\id)(F)$ induces the equivalences 
\[  F\xleftarrow{\sim} G\xras\phi(F)  \] 
whose composition gives the equivalence required in Definition~\ref{spectrumobject}.  The $\infty$-category of $T$-spectrum objects in $C$ comes naturally equipped with an evaluation functor $\Ev_{n}:\Sp_{T}(C)\ra C$ for every $n\in\bb{Z}$ which acts on a spectrum $F$ and picks out its $n$-{th} term $F(n)$.  If $C$ is a presentable $\infty$-category then this evaluation functor admits a left adjoint $\Fr_{n}:C\ra\Sp_{T}(C)$.  

We will be particularly interested in the case where the endofunctor $T$ is the loop functor.  In this case, if $C$ is a pointed $\infty$-category with finite limits then the $\infty$-category $\Sp_{\Omega}(C)$ is a stable $\infty$-category (see Proposition 1.4.2.18 of \cite{L1}).  This defines a natural functor 
$\Sp_{\Omega}$ from the $\infty$-category of pointed $\infty$-categories with finite limits and left exact functors to the $\infty$-category $\Catinf^{\perp}$ of stable $\infty$-categories whose right adjoint is the forgetful functor.  

Let $C_*$  denote the full subcategory of $\RHom([1],C)$ spanned by those morphisms $x\ra y$ for which $x$ is a terminal object of $C$.  We call $C_*$ the $\infty$-category of \textit{pointed objects} of $C$.  If $C$ is pointed, then the forgetful functor $C_{*}\ra C$ is an equivalence of $\infty$-categories.  

\begin{dfn}
A \textit{spectrum} is a $\Omega$-spectrum object of the $\infty$-category $\cS_{*}$ of pointed spaces. 
\end{dfn}

Let $\Sp:=\Sp_{\Omega}(\cS_{*})$ denote the $\infty$-category of spectra.  The $\infty$-category $\Sp$ is stable and presentable.  It follows from Definition~\ref{spectrumobject} that the $\infty$-category $\Sp$ of spectra can be identified with the homotopy limit of the tower 
\[     \{\ldots\ra\cS_{*}\xra{\Omega}\cS_{*}\xra{\Omega}\cS_{*}\}.  \]
of $\infty$-categories.

Let $\cS^{fin}$ denote the smallest full subcategory of $\cS$ which contains the final object and is stable under finite colimits.  Then $\Ind(\cS^{fin}_{*})\ra\cS_{*}$ is an equivalence of $\infty$-categories and thus $\cS_{*}$ is compactly generated.  Moreover, let 
\[   \cS_{\infty}^{fin}:=\colim~\{  \cS_{*}^{fin}\xra{\Sigma}\cS_{*}^{fin}\xra{\Sigma}\ldots  \}  \] 
in the $\infty$-category of $\infty$-categories and exact functors.  Then the $\infty$-category of spectra is compactly generated and 
\[  \Ind(\cS^{fin}_{\infty})\ra\Sp   \] 
is an equivalence of $\infty$-categories.

Let $C$ be a presentable $\infty$-category.  Then the $\infty$-category $\Sp_\Omega(C_{*})$ is presentable and the natural functor $\Ev_{n}:\Sp_\Omega(C_{*})\ra C$ admits a left adjoint $\Fr_{n}:C\ra\Sp_\Omega(C_{*})$.  Let $C=\cS$ and $*$ be the final object of $\cS$.  The object $\Fr_{0}(*)$ of $\Sp$ will be called the \textit{sphere spectrum} and will be denoted by $\bb{S}$.  Recall the homotopy group functor on spectra $\pi_{n}:\Sp\ra\Ab$ which takes a spectrum $A$ to the abelian group $\Hom_{\h\Sp}(\bb{S}[n],A)$.  A map $f:A\ra B$ of spectra is an equivalence if and only if it induces isomorphisms $\pi_{n}A\ra\pi_{n}B$ for all $n\in\bb{Z}$.

The $\infty$-category $\Sp$ admits a symmetric monoidal structure which is uniquely characterized by the property that the unit object of $\Sp$ is the sphere spectrum $\bb{S}$ and the bifunctor $\otimes:\Sp\times\Sp\ra\Sp$ preserves colimits seperately in each variable (Corollary 6.3.2.16 of \cite{L1}).  If $\sS p$ is the category of symmetric spectra endowed with the $\bb{S}$-model structure \cite{Sh} and smash product symmetric monoidal structure then $\sS p$ can be lifted to a simplicial symmetric monoidal model category and there exists an equivalence $L(\sS p)\ra\Sp$ of symmetric monoidal $\infty$-categories.

\begin{notn}
Let $C$ be a stable $\infty$-category and $x$ be any object of $C$.  Let
\[   x[n]:= 
\begin{cases}
\Sigma^{n}x       &\text{if $n\geq 0$}, \\ 
\Omega^{-n}x &\text{if $n\leq 0$}, 
\end{cases} \]
given by taking the $n$th power of the suspension and loop functors.  We use the same notation for the corresponding object in $\h C$.  
\end{notn}

The homotopy category of a stable $\infty$-category $C$ is a triangulated category where the suspension functor $\Sigma:x\mapsto x[1]$ denotes the translation functor.  See Theorem 1.1.2.13 of \cite{L1} for the proof.  Furthermore, there exists a spectrum of maps between any two objects in $C$.  For all $x,y\in C$, since $x\simeq 0\times_{x[1]}0$, the space $\Map_{C}(x,y)$ (pointed by the zero map) is the zeroth space of the spectrum
\[  \ldots\xra{\Omega}\Map_{C}(x,y[2])\xra{\Omega}\Map_{C}(x,y[1])\xra{\Omega}\Map_{C}(x,y)\xra{\Omega}\Map_{C}(x[1],y)\xra{\Omega}\Map_{C}(x[2],y)\xra{\Omega}\ldots.  \]
More precisely, let $\Catinf^{\perp,p}$ denote the full subcategory of $\Catinf^p$ spanned by stable, presentable objects.  

\begin{prop}\label{stableisenrichedoverspectra}
Let $C$ be a stable, presentable $\infty$-category.  Then $C$ is tensored and enriched over the $\infty$-category $\Sp$ of spectra. 
\end{prop}

\begin{proof}
By Proposition 6.3.2.15 of \cite{L1} there exists an equivalence
\[     \Mod_{\Sp}({\Cat}_\infty^{p})\ra\Catinf^{\perp,p}  \]
of $\infty$-categories.  A presentable $\infty$-category which is tensored over $\Sp$ where the tensored structure preserves (small) colimits is automatically stable and since the functor $\bullet\otimes x:\Sp\ra C$ preserves (small) colimits for all $x\in C$, the result follows.
\end{proof}

\begin{dfn}
A spectrum $A$ is said to be \textit{connective} if $\pi_{n}A\simeq 0$ for all $n<0$.  It is said to be \textit{discrete} if it is connective and $0$-truncated.
\end{dfn}

We denote by $\Sp^{c}$ (resp. $\Sp^{d}$) the full subcategory of $\Sp$ spanned by the connective (resp. discrete) spectra.  The $\infty$-category of connective spectra is the smallest full subcategory of $\Sp$ closed under colimits and extensions which contains the sphere spectrum $\bb{S}$.   It is projectively generated with the sphere spectrum being a compact projective generator.  

\begin{dfn}
A \textit{commutative ring spectrum} is a commutative monoid object in the $\infty$-category $\Sp$ of spectra with respect to the smash product monoidal structure.
\end{dfn}

A commutative ring spectrum will be referred to as an $\Einfty$-ring.  The $\infty$-category of $\Einfty$-rings will be denoted $\fE:=\CMon(\Sp)$.  The $\infty$-category $\fE$ will play the role of our generalised theory of rings.  If $R$ is an $\Einfty$-ring then the $\infty$-category of \textit{commutative $R$-algebras} 
\[  \CAlg_{R}:=\CMon(\Mod_{R})  \] 
in $\Sp$ is equivalent to $\fE_{R/}$.  

\begin{ex}
Let $S^{n}$ be the simplicial $n$-sphere.  For any commutative ring $k$, one can associate an $\Einfty$-ring spectrum $Hk$ called the \textit{Eilenberg Mac Lane} ring spectrum which is the sequence of simplicial abelian groups $k\otimes S^{n}$ where $(k\otimes S^{n})_{m}$ is the free abelian group on the non-basepoint $m$-simplices of $S^{n}$.  This construction defines a fully faithful functor $H:\Rng\ra\fE$ from the $\infty$-category of commutative rings to the $\infty$-category of $\Einfty$-rings with $\pi_0(Hk)=0$ and $\pi_i(Hk)\neq 0$ for all $i>0$.
\end{ex}

For $R\in\fE$ and $n\in\bb{Z}$, let $\pi_n R$ denote the $n$th homotopy group of the underlying spectrum of $R$.  We can identify $\pi_n R$ with the set $\pi_0\Map_{\Sp}(\bb{S}[n],R)$.  In particular, $\pi_0 R$ is a discrete commutative ring and $\pi_{n}R$ has the natural structure of a $\pi_{0}(R)$-module.  An $\Einfty$-ring $R$ is said to be \textit{connective} if $\pi_{n}R\simeq 0$ for all $n<0$.  The full subcategory of commutative ring spectra spanned by the connective objects, denoted $\fE^{c}$, is equivalent to the $\infty$-category $\CMon(\Sp^{c})$ of commutative monoid objects in the $\infty$-category of connective spectra.  

We can think of connective $\Einfty$-rings as simply spaces endowed with an addition and multiplication satisfying the axioms for a commutative ring up to coherent homotopy.  More precisely, let $\Ev_0:\fE^{c}\ra\cS$ denote the composition $\fE^{c}\ra\Sp^{c}\ra\cS$.  Then
\[      \theta:\fE^{c}\ra\Mod_{T}(\cS)       \]
is an equivalence of $\infty$-categories for the monad $T=\Ev_0\circ\Fr_0$ on $\cS$ (see Remark 7.1.1.8 of \cite{L1}).

An $\Einfty$-ring $R$ is said to be \textit{bounded} if $\pi_i(R)=0$ for $i>n$ for some $n$.  It is said to be \textit{discrete} if it is connective and $0$-truncated.  We let $\fE^{d}$ denote the full subcategory of $\fE$ spanned by the discrete objects.  A connective $\Einfty$-ring $R$ is discrete if and only if for all $n>0$ the homotopy group $\pi_{n}R$ is trivial.  The $\infty$-category $\fE^{d}$ is equivalent to the $\infty$-category $\CMon(\Sp^{d})$ of commutative monoid objects in the $\infty$-category of discrete spectra.    

A convenient way to manipulate algebra in the context of ring spectra is to utilise its model categorical interpretation.  Let $\sS p$ be the model category of symmetric spectra.  Then the category $\CMon(\sS p)$ inherits the structure of a simplicial model category again by \textit{loc. cit.} and by Section~\ref{monoidal} there exists an equivalence
\[   s:L(\CMon(\sS p))\ra\fE   \]
of $\infty$-categories.  Furthermore, by Example~\ref{lmodmodl}, there exists an equivalence
\[   L(\Mod_R(\sS p))\ra\Mod_{s(R)}(\Sp) \]
of $\infty$-categories.

\begin{lem}
Let $R$ be an $\Einfty$-ring.  The $\infty$-category of modules $\Mod_{R}$ is a stable $\infty$-category.
\end{lem}

\begin{proof}
By the above discussion, the $\infty$-category $\Mod_R$ is equivalent to the localisation of a stable model category.  The localisation of a stable model category is a stable $\infty$-category (Example~\ref{stablemodelcategory}).
\end{proof}

\begin{ex}
When $R\in\fE$ is discrete, ie. an ordinary commutative ring, the triangulated category $\h\tu{Mod}_{R}$ corresponds to the classical derived category of $\Mod_{R}$.
\end{ex}

\begin{ex}\label{cdgaca}
Let $k$ be a field of characteristic zero and  $\tu{cdga}_{k}$ the model category of commutative differential graded $k$-algebras where the weak equivalences are given by the quasi-isomorphisms (see Section 5 of \cite{SS2}), ie. the model category of commutative monoid objects in the symmetric monoidal model category $C(k):=C(\Mod_{k}(\Ab))$ of chain complexes of $k$-modules.  By Theorem 5.1.6 of \cite{SS1} there exists a Quillen equivalence $\Mod_{Hk}(\sS p)\ra C(k)$ where the model category $\sS p$ of symmetric spectra is endowed with the $\bb{S}$-model structure.  Thus we have a diagram
\begin{diagram}
\CMon(L(\Mod_{Hk}(\sS p))) &\rTo^{\sim}  &      \CMon(L(C(k)))   \\  
\dTo      &              &     \dTo \\
\CAlg_{Hk}(\Sp)  &\rTo       &  L(\tu{cdga}_{k})  
\end{diagram}
of $\infty$-categories.  The left vertical arrow is an equivalence by Example~\ref{lmodmodl} and the right vertical arrow is an equivalence by the discussion following Example~\ref{cartesianstructureoncatinfty}.  Therefore the $\infty$-category of $Hk$-algebras in $\Sp$ can be identified with the localisation of the model category of commutative differential $k$-algebras.
\end{ex}

Let $R$ be a connective $\Einfty$-ring and $\CAlg_{R}^{c}:=\CMon(\Mod_{R}^{c})$ the $\infty$-category of \textit{connective commutative $R$-algebras}.  The $\infty$-category $\CAlg_{R}^{c}$ is projectively generated with the compact projective objects being identified with those connective commutative $R$-algebras that are retracts of a finitely generated free commutative $R$-algebra.  If  $(\CAlg_{R}^{c})^{fp}$ denotes the smallest full subcategory of $\CAlg_{R}^{c}$ which contains all finitely generated free $R$-algebras and is stable under finite colimits then
\[   \Ind((\CAlg_{R}^{c})^{fp})\ra\CAlg_{R}^{c}  \]
is an equivalence of $\infty$-categories.  Moreover, the $\infty$-category $\CAlg_{R}^{c}$ is compactly generated with the compact objects being the finitely presented $R$-algebras.

Consider the $\infty$-category $\Mod_{R}$ of $R$-modules for $R$ an $\Einfty$-ring.  If $M$ is an $R$-module we will denote by $\pi_nM$ the homotopy group of its underlying spectrum.  The $\infty$-category $(\Mod_{R})_{\geq 0}$ is the smallest full subcategory of $\Mod_{R}$ which contains $R$ and is stable under small colimits.  A module $M$ in $\Mod_{R}$ is said to be \textit{connective} if $\pi_nM=0$ for all $n<0$ and we call $(\Mod_{R})_{\geq 0}$ the $\infty$-category of \textit{connective} $R$-modules.  Likewise, a module $M$ in $\Mod_R$ is said to be \textit{anti-connective} if $\pi_nM=0$ for all $n>0$.  An $R$-module is said to be \textit{free} if it is equivalent to a coproduct of copies of $R$ and \textit{finitely generated} if it can be written as a finite coproduct of copies of $R$.

Let $R$ be a connective $\Einfty$-ring.  An $R$-module $M$ is said to be \textit{projective} if it is a projective object of the $\infty$-category $(\Mod_{R})_{\geq 0}$ of connective $R$-modules (note that the $\infty$-category $\Mod_{R}$ has no nonzero projective objects).  The $R$-module $M$ is projective if and only if there exists a free $R$-module $N$ such that $M$ is a retract of $N$.  If $N$ is moreover finitely generated, then $M$ is a compact projective object of $(\Mod_{R})_{\geq 0}$.  This is equivalent to $M$ being projective and $\pi_{0}M$ being finitely generated as a $\pi_{0}R$-module.  The $\infty$-category of of connective modules over a connective $\Einfty$-ring is projectively generated.

Let $R$ be a connective $\Einfty$-ring.  Then the inclusion $\CAlg_{R}^{c}\hookrightarrow\CAlg_{R}$ commutes with colimits so there exists an adjunction
\[      i:\CAlg_{R}^{c}\rightleftarrows\CAlg_{R}:(\bullet)^{c}  \]
where the right adjoint is called the \textit{connective cover}.  Explicitly, a connective cover of an $R$-algebra $A$ is a map $f:A\ra A^{c}$ such that $A^{c}$ is connective and for any connective $R$-algebra $B$, there exists an equivalence
\[   \Map_{\CAlg_{R}}(B,A)\ra\Map_{\CAlg_{R}^{c}}(B,A^{c})   \]
of $\inftyz$-categories.

Let $C$ be an $\infty$-category and $x,y\in C$.  Then $y$ is said to be a retract of $x$ if it is a retract of $x$ in $\h C$, ie. there exists a diagram
$y\xra{i}x\xra{p}y$
which coincides with the identity $\id_y$ in $\h C$.

\begin{dfn}
Let $C$ be a stable $\infty$-category.  Let $C^{\tu{perf}}$ denote the smallest stable subcategory of $C$ which contains the unit object and is closed under retracts.  An object $x\in C$ is said to be \textit{perfect} if it is an object of $C^{\tu{perf}}$.
\end{dfn}

Let $R$ be an $\Einfty$-ring.  Then there exists an equivalence $\ModRperf\ra\Mod_{R}^{\tu{cpt}}$ between the $\infty$-category of perfect $R$-modules and the full subcategory of $\Mod_R$ spanned by the compact objects.  Furthermore if $R$ is a connective $\Einfty$-ring and $\Mod_{R}^{\tu{fgp}}$ denotes the smallest stable subcategory of $\Mod_{R}$ which contains all finitely generated projective modules then there exists an equivalence $\Mod_{R}^{\tu{fgp}}\ra\ModRperf$ of $\infty$-categories.

\begin{dfn}
Let $C$ be a stable $\infty$-category.  Then $C$ is said to admit a \textit{t-structure} if there exists a t-structure on the homotopy category $\h C$.
\end{dfn}

Let $C$ be a stable $\infty$-category.  The full subcategory of $C$ spanned by the objects of $(\h C)_{\leq n}$ and $(\h C)_{\geq n}$ will be denoted by $C_{\leq n}$ and $C_{\geq n}$ respectively.  

\begin{ex}
When $C$ is a presentable stable $\infty$-category, any small collection of objects $\{x_{\alpha}\}$ determines a t-structure on $C$.  The construction is as follows (see Proposition 1.4.5.11 of \cite{L1}).  One builds a subcategory $C'$ of $C$ as the smallest full subcategory of $C$ containing $\{x_{\alpha}\}$ which is closed under small colimits and such that for every distinguished triangle
\[     x\ra y\ra z\ra x[1]  \]
for which $x$ and $z$ are in $C'$, then $y$ is in $C'$.  In this case, there exists a t-structure on $C$ such that $C'=C_{\geq 0}$ and $C_{\geq 0}$ is presentable.  
\end{ex}

We call a t-structure on a presentable stable $\infty$-category $C$ \textit{accessible} if the subcategory $C_{\geq 0}$ is presentable.  It follows that if $C$ admits an accessible t-structure then $C_{\leq 0}$ is also presentable.  The t-structures that we will be concerned with in this paper are the following accessible t-structures on the $\infty$-category of spectra and the $\infty$-category of $R$-modules for $R$ a connective $\Einfty$-ring.

\begin{ex}\label{modrtstructure}
Let $R$ be a connective $\Einfty$-ring.  Then there exists an accessible t-structure on the $\infty$-category $\Mod_{R}$ of $R$-modules where 
\begin{itemize}
\item $(\Mod_{R})_{\leq 0}$ is the full subcategory of $\Mod_{R}$ spanned by the objects $\{M\in\Mod_R|\forall n>0,\pi_{n}M\simeq0\}$.
\item $(\Mod_{R})_{\geq 0}$ is the full subcategory of $\Mod_{R}$ spanned by the objects $\{M\in\Mod_R|\forall n<0,\pi_{n}M\simeq0\}$.  
\end{itemize}
This t-structure is left and right complete.  With this t-structure, the heart of $\Mod_{R}$ is equivalent to the abelian category of discrete modules over the ring $\pi_0(R)$.
\end{ex}

Let $C$ be a stable $\infty$-category.  A t-structure $(C_{\leq 0},C_{\geq 0})$ on $C$ is non-degenerate if and only if $\cup _{i}C_{\leq i}=\cup_{i}C_{\geq i}=C$ and $\cap _{i}C_{\leq i}=\cap_{i}C_{\geq i}=0$.  The t-structure of Example~\ref{modrtstructure} are non-degenerate.  The usefulness of this non-degeneracy is that it enables us to check equivalences in these $\infty$-categories on their corresponding truncations in the following sense.  The $\infty$-category $C_{\leq n}$ is stable under limits in $C$ and the $\infty$-category $C_{\geq n}$ is stable under colimits in $C$.  Hence there exists a left adjoint 
\[  \tau_{\leq n}:C\ra C_{\leq n} \]  
to the inclusion map $C_{\leq n}\hookrightarrow C$ and a right adjoint 
\[ \tau_{\geq n}:C\ra C_{\geq n}  \]  
to the inclusion map $C_{\geq n}\hookrightarrow C$.  By Proposition 1.2.1.10 of \cite{L1} there exists an equivalence $\tau_{\leq m}\circ\tau_{\geq n}\simeq\tau_{\geq n}\circ\tau_{\leq m}$ of functors from $C$ to $C_{\leq m}\cap C_{\geq n}$ which we will denote by $\tau_{n,m}:C\ra C_{[n,m]}$.

\begin{dfn}
Let $C$ be a stable $\infty$-category.  A t-structure on $C$ is said to be \textit{non-degenerate} if for all objects $x$ in $C$, if $\tau_{n,n}x=0$ for all $n$ then $x=0$.
\end{dfn}

\begin{dfn}
Let $C$ and $D$ be stable $\infty$-categories admitting t-structures.  A functor $F:C\ra D$ is said to be \textit{left} (resp. \textit{right}) \textit{t-exact} if it is exact and sends $C_{\leq 0}$ into $D_{\leq 0}$ (resp. sends $C_{\geq 0}$ into $D_{\geq 0}$).  It is said to be \textit{t-exact} if it is both left and right t-exact.
\end{dfn}

\begin{lem}\label{texactlemma}
Let $C$ and $D$ be stable $\infty$-categories admitting t-structures and let $F\dashv G$ be an adjunction between them.  Then $F$ is right t-exact if and only if $G$ is left t-exact.
\end{lem}

\begin{proof}
Let $F:C\ra D$ be right t-exact.  Then for any $x\in C_{\geq 0}$ we have $F(x)\in D_{\geq 0}$.  Thus for any $y\in D_{\leq -1}$, $\Hom_{\h D}(F(x),y)\simeq\Hom_{\h C}(x,G(y))=0$.  Therefore $G(y)\in C_{\leq -1}$.
\end{proof}

Another consequence of having a non-degenerate t-structure on a stable $\infty$-category is the following.  We say that a stable $\infty$-category $C$ endowed with a t-structure is \textit{left t-complete} if the natural map 
\[  C\ra\holim_nC_{\leq n}:=\holim_n\{\ldots\ra C_{\leq 2}\xra{\tau_{\leq 1}}C_{\leq 1}\xra{\tau_{\leq 0}}C_{\leq 0}C\xra{\tau_{\leq -1}}\ldots\}   \]
is an equivalence of $\infty$-categories.  By Proposition 1.2.1.19 of \cite{L1}, if a stable $\infty$-category with a t-structure admits countable products such that $C_{\geq 0}$ is stable under countable products, then $C$ is left t-complete if and only if $\cap_iC_{\geq i}=0$.

\begin{dfn}
Let $C$ be a stable $\infty$-category and $D$ a stable $\infty$-category admitting a t-structure $(D_{\leq 0},D_{\geq 0})$.  A functor $F:C\ra D$ is said to \textit{create a t-structure} on $C$ if $C_{\leq 0}:=\{x\in C| F(x)\in D_{\leq 0}\}$ and $C_{\geq 0}:=\{x\in C| F(x)\in D_{\geq 0}\}$ define a t-structure on $C$.
\end{dfn}

\begin{dfn}
Let $C$ be a stable $\infty$-category with a t-structure $(C_{\leq 0},C_{\geq 0})$.  The \textit{heart} of $C$ is the full subcategory $C^\heartsuit:=C_{\leq 0}\cap C_{\geq 0}$.
\end{dfn}

Let $C$ be a stable $\infty$-category admitting a t-structure.  Then for any object $x\in C$ and $n\geq -1$, the object $x$ belongs to $C_{\leq n}$ if and only if the space $C(y,x)$ is $n$-truncated for all $y\in C_{\geq 0}$.  Thus for $x$ and $y$ in $C^{\heartsuit}$, the group $\pi_{n}C(x,y)$ vanishes for all $n>0$ and so there exists an equivalence  
$C^{\heartsuit}\ra\h C^{\heartsuit}$
of $\infty$-categories.

Let $C$ be a stable symmetric monoidal $\infty$-category.  A t-structure on $C$ is said to be \textit{compatible} with the symmetric monoidal structure if for all $x\in C$ the functor $x\otimes\bullet$ is exact and $C_{\geq 0}$ is closed under tensor products and contains the unit object.  The t-structure on the $\infty$-category $\Sp$ of spectra is compatible with the symmetric monoidal structure.  As a result, by Proposition~\ref{subsymmmonoidal}, the $\infty$-category $\Sp^{c}$ of connective spectra inherits the structure of a symmetric monoidal $\infty$-category.  

Let $R$ be a connective $\Einfty$-ring.  By extension, the $\infty$-category $\Mod_{R}^{c}$ of connective $R$-modules and connective $R$-algebras $\CAlg_{R}^{c}:=\CMon(\Mod_{R}^{c})$ inherit symmetric monoidal structures.  Let $C$ be a stable symmetric monoidal $\infty$-category which admits a t-structure that is compatible with the symmetric monoidal structure.  Then by Example 2.2.1.10 of \cite{L1}, the heart $C^{\heartsuit}$ of $C$ inherits the structure of a symmetric monoidal category.

\begin{dfn}
Let $C$ be a stable $\infty$-category admitting a t-structure and let $n\in\bb{Z}$.  We define a functor 
\[  \pi_n^t:C \ra C^{\heartsuit}   \] 
sending $x$ to $\tau_{0,0}(x[-n])$. 
\end{dfn}

Let $C$ be a stable $\infty$-category.  It follows from Theorem 1.3.6 of \cite{BBD} that the category $C^{\heartsuit}$ is abelian.  One can show that the heart of a presentable stable $\infty$-category equipped with an admissible $t$-struture is a presentable abelian category.  Let $\Ab$ denote the category of abelian groups.  When $C$ is the $\infty$-category $\Sp$ of spectra endowed with its natural t-structure then the functor 
\[     \Sp^{\heartsuit}\ra\Ab     \]
is an equivalence of $\infty$-categories.  This follows from Proposition 1.4.3.5 of \cite{L1}.  Furthermore, the functor 
\[    \pi_0^t:\fE^{d}\ra\Rng   \]
is an equivalence of $\infty$-categories (Proposition 7.1.3.18 of \cite{L1}) showing that the $\infty$-category of commutative ring spectra contains, as a fully subcategory, the ordinary theory of commutative rings $\Rng$.

To end this section, we prove an important lemma which provides conditions under which a functor between stable $\infty$-categories admitting non-degenerate t-structures commutes with limits of cosimplicial objects needed in the proof of our main duality theorem. 

\begin{lem}\label{tstructurelimitlemma}
Let $C$ and $D$ be stable $\infty$-categories admitting non-degenerate t-structures.  Let $f:C\ra D$ be a t-exact functor.  If $X$ is a cosimplicial object in $C$ such that there exists $k\geq 0$ with $\pi_{i}^t(X_n)=0$ for all $i>k$ and all $n$ then
\[    f(\underset{n\in\Delta}{\lim}~X_n)\ra\underset{n\in\Delta}{\lim}~f(X_n)  \]
is an equivalence in $D$.
\end{lem}

\begin{proof}
We can choose $k=0$.  Since the t-structure on $D$ is non-degenerate, we can check the equivalence on the truncation
\[   \tau_{\geq-N}(f(\underset{n}{\lim}~X_n))\ra\tau_{\geq-N}(\underset{n}{\lim}~f(X_n)). \]
The functor $f$ is t-exact and the truncation commutes with limits so we are reduced to proving
\[     f(\underset{n}{\lim}(\tau_{\geq-N}X_n))\ra\underset{n}{\lim}~f(\tau_{\geq-N}X_n)). \]
The limits in $C_{[-N,0]}$ are considered as limits in $C_{\leq 0}$.  Note that $C_{[-N,0]}$ is a subcategory of $C$ where the mapping spaces are $N$-truncated.  Any limit along $\Delta$ in an $\infty$-category whose mapping spaces are $N$-truncated is a finite limit.  Since $f$ is t-exact it commutes with finite limits and truncations so the induced functor $f:C_{[-N,0]}\ra D_{[-N,0]}$ preserves finite limits and the result follows.
\end{proof}

\section{Stacks, gerbes and Hopf algebras}\label{stacksgerbesandtopologies}

This section is devoted to the group side of the correspondence in our Tannaka duality theorem.  The first main result gives conditions under which the $\infty$-category of sheaves valued in some presentable symmetric monoidal $\infty$-category is tensored and enriched, both over itself and the $\infty$-category which it is valued in.  If these sheaves are merely presheaves then we give conditions for the presheaf of maps between two of these objects to be a sheaf.  The main example of interest is when the presentable symmetric monoidal $\infty$-category is the $\infty$-category of $\infty$-categories.  In this case, we provide sufficient conditions for a presheaf valued in $\infty$-categories to be a sheaf of $\infty$-categories.  

After defining group stacks and gerbes we show that gerbes are exactly those stacks which are locally equivalent to the classifying stack of a group stack.  Furthermore, affine group stacks correspond to Hopf algebras in a symmetric monoidal $\infty$-category.  We finish the section with some technical results concerning Hopf algebras needed in the proof of our duality theorem. 

Let $(C,\tau)$ be a site, ie. an $\infty$-category endowed with a topology.  A topology on an $\infty$-category $C$ is equivalent to a Grothendieck topology on the homotopy category $\h(C)$ (see Remark 6.2.2.3 of \cite{Lu}).  If $C$ is an $\infty$-category with pullbacks, to prove the existence of a topology $\tau$ on $C$ it suffices to prove the existence of a pretopology on $C$.  That is, a function $cov_{\tau}$ which assigns to each object $x$ in $C$ a collection $cov_{\tau}(x)$ of subsets of objects in $C_{/x}$ called \textit{covering families} of $x$ satisfying the three axioms~:
\begin{itemize}
\item Stability~: If $f:y\ra x$ is an equivalence in $C$ then the singleton $\{f:y\ra x\}$ is in $cov_{\tau}(x)$.
\item Composition~: If $\{f_{i}:y_{i}\ra x\}_{i\in I}$ is in $cov_{\tau}(x)$ and if for each $i\in I$ one has a family $\{g_{ij}:z_{ij}\ra y_{i}\}_{j\in J_{i}}$ in $cov_{\tau}(y_{i})$ then the family $\{f_{i}\circ g_{ij}:z_{ij}\ra x\}_{i\in I,j\in J_{i}}$ is in $cov_{\tau}(x)$.
\item Base change~: If $\{f_{i}:y_{i}\ra x\}_{i\in I}$ is in $cov_{\tau}(x)$ then for any morphism $g:z\ra x$, the pullbacks $z\times_{x}y_{i}$ exist and the family $\{z\times_{x}y_{i}\ra z\}_{i\in I}$ is in $cov_{\tau}(z)$.
\end{itemize}
A pretopology on $C$ determines a topology on $C$~: a sieve $R$ on an object $x$ in $C$ is in $\tau(x)$ if and only if there exists a covering family $J$ in $cov_\tau(x)$ such that $J$ is a subset of $R$.

Let $(C,\tau)$ be a site and $\{u_{i}\ra x\}_{i\in I}$ a covering family of $x\in C$.  Let $u=\coprod_{i}u_{i}$.  We will say that the map $u\ra x$ is a \textit{covering} of $x$.  A \textit{cover} of $x$ (associated to $u$) is the simplicial prestack $u_{*}\in s\Pr(C)_{/x}$ given by
\[   u_{*}: [n]  \mapsto u\underset{x}{\times}\hdots\underset{x}{\times}u.   \]
                
\begin{dfn}
Let $C$ be a site and $X$ an $\infty$-category with limits.  An $X$-valued prestack $F:C^{op}\ra X$ on $C$ is said to be an \textit{$X$-valued stack} if for all $x\in C$ and all coverings $u_{*}$ in $s\Pr(C)_{/x}$ the map
\[   F(x)\ra\underset{\Delta}{\lim}~F(u_{*})  \]
is an equivalence in $X$.
\end{dfn}

The full subcategory of $\Pr_X(C)$ spanned by the $X$-valued stacks on the site $(C,\tau)$ will be denoted $\St_{X}(C,\tau)$.  A $\cS$-valued stack will simply be called a \textit{stack} and we will denote the $\infty$-category of stacks $\St_{\cS}(C,\tau)$ by $\St(C,\tau)$.  A topology $\tau$ on $C$ is said to be \textit{subcanonical} if every representable functor on $C$ is a stack with respect to $\tau$.

\begin{rmk}\label{inclusionofstacks}
Note that the adjunction $\prod:\sK\rightleftarrows\PC(\sK):\fK$ between model categories (see \cite{HS}) induces an adjunction $\cS\dashv\Catinf$ between $\infty$-categories and thus an adjunction 
\[       \St_{\cS}(C,\tau)\rightleftarrows\St_{\Catinf}(C,\tau)   \]
where $(\fK^{*}F)(x)=\fK F(x)$ for a stack $F$ in $\St_{\Catinf}(C,\tau)$.
\end{rmk}

\begin{ex}\label{commastack}
Let $(C,\tau)$ be a site for a subcanonical topology $\tau$ and $h_x$ the representable prestack on an object $x$ in $C$.  Then there exists an equivalence 
$\St(C,\tau)_{/h_x}\ra\St(C_{/x},\tau)$ 
of $\infty$-categories.
\end{ex}

Let $C$ be a symmetric monoidal $\infty$-category.  We give the opposite $\infty$-category of commutative monoid objects in $C$ the following special notation~:
\[ \Aff_{C}:=\CMon(C)^{op}.  \]
When $C$ is the symmetric monoidal $\infty$-category ${\Mod}_{R}(D)$ of $R$-modules in a symmetric monoidal $\infty$-category $D$, we will write $\Aff_{R}:=\Aff_{{\Mod}_{R}(D)}$.  In other words, the $\infty$-category $\Aff_{R}$ is the opposite of the $\infty$-category of commutative $R$-algebras in $D$.  The Yoneda embedding $\Aff_{C}\ra(\Aff_{C})^{\wedge}$ will be denoted by $\Spec$.

\begin{ex}
Let $C$ be a symmetric monoidal $\infty$-category and $R$ a commutative monoid object of $C$.  Then we will denote by $\St(R,\tau):=\St(\Aff_{R},\tau)$ the $\infty$-category of stacks with respect to the site $(\Aff_{R},\tau)$ of $R$-algebras in $C$.  By Example~\ref{commastack} we have an equivalence $\St(R,\tau)\simeq\St(\Aff_{C},\tau)_{/\Spec R}$ of $\infty$-categories.
\end{ex}

The $\infty$-category of stacks on a site can be obtained by the localisation (in the sense of Definition~\ref{localisationdfn}) of the $\infty$-category of prestacks.  The following proposition gives two possible choices for the set of maps from which to localise.  

\begin{prop}\label{stacklocalisation}
Let $(C,\tau)$ be a site.  The following classes of maps give the same localisation of $\Pr(C)$~:
\begin{enumerate}
\item The set of all covering sieves on $x$.
\item The set of maps $u_{*}\ra x$.
\end{enumerate}
\end{prop}

\begin{proof}
We note that $\Pr(C)$ is a presentable $\infty$-category.  Thus the localisation is a Bousfield localisation in the sense of Section~\ref{inftycategories}.  The result now follows from Proposition A1 of \cite{DHI}.
\end{proof}

Let $S$ denote the equivalent set of maps of Proposition~\ref{stacklocalisation}.  Then $\St(C,\tau)\simeq L_{S}(\Pr(C))$.  

\begin{prop}\label{xstacksareenrichedinx}
Let $(C,\tau)$ be a site and $X$ a presentable symmetric monoidal $\infty$-category whose tensor product preserves colimits seperately in each variable.  Then the $\infty$-category $\xstacks$ is tensored and enriched over itself.  Moreover, it is tensored and enriched over $X$.
\end{prop}

\begin{proof}
We first observe that the $\infty$-category $\St_X(C,\tau)$ of $X$-valued stacks on $C$ is naturally a symmetric monoidal $\infty$-category with respect to the pointwise symmetric monoidal structure of Example~\ref{pointwisemonoidalstructure}.  Explicitly, ${\St}_X(C,\tau)_{\<n\>}\simeq\St_{X_{\<n\>}}(C,\tau)$.  To show that it is enriched over itself it remains to show that $\bullet\otimes F:\St_X(C,\tau)\ra\St_X(C,\tau)$ preserves colimits (see Section~\ref{monoidal}).  This follows by the assumption that the tensor product on $X$ preserves colimits seperately in each variable (colimits are calculated pointwise in functor categories).

To show that $\St_X(C,\tau)$ is tensored and enriched over $X$ it suffices to show that there exists a colimit preserving symmetric monoidal functor $X\ra{\St}_X(C,\tau)$.  First consider the $\infty$-category $\Pr_X(C)$ of $X$-valued prestacks endowed with its pointwise symmetric monoidal structure.  The constant prestack functor induces a symmetric monoidal functor $X\simeq{\Pr}_X(*)\ra{\Pr}_X(C)$.  This functor preserves colimits seperately in each variable owing to the assumption that they do in $X$.  Finally, the stackification functor $L_S$ preserves colimits (it is a left adjoint) so the composition $X\ra{\Pr}_X(C)\ra L_S{\Pr}_X(C)\simeq{\St}_X(C,\tau)$ preserves colimits seperately in each variable.
\end{proof}              

The internal Hom provided by Proposition~\ref{xstacksareenrichedinx}, and more generally for $X$-valued prestacks, will be denoted by $\uHom(F,G)$.  We will now demonstrate that $\uHom(F,G)$ is an $X$-valued stack under the weaker condition that $F$ is only an $X$-valued prestack when the conditions of Proposition~\ref{xstacksareenrichedinx} are satisfied.

\begin{prop}\label{homstack}
Let $(C,\tau)$ be a site, $X$ a presentable symmetric monoidal $\infty$-category whose tensor product preserves colimits separately in each variable and $F$ and $G$ two $X$-valued prestacks on $C$.  If $G$ is an $X$-valued stack then $\uHom(F,G)$ is an $X$-valued stack on $C$.
\end{prop}

\begin{proof}
Let $F$ be an $X$-valued prestack and $G$ an $X$-valued stack.  Then $\uHom(F,G)$ is an $X$-valued stack if and only if the map
\[  \Mor(x,\uHom(F,G))\ra\Mor(\underset{n}{\colim}~u_{*},\uHom(F,G))   \]
is an equivalence in $X$.  This is equivalent to the condition that $\Mor(x\otimes F,G)\ra\Mor(\colim_{n}(u_{*}\otimes F),G)$ is an equivalence in $X$ and subsequently to $\colim_{n}(u_{*}\otimes F)\ra x\otimes F$ being an equivalence of $X$-valued stacks.  Let $B$ be an object of $X$.  Since $G$ is an $X$-valued stack, the map $\Map(B,G(x))\ra\lim_n\Map(B,G(u_*))$ is an equivalence and so $\Map(x\otimes B,G)\ra\Map(\colim_{n}(u_*\otimes B),G)$ is an equivalence for any $B\in X$.  

Now recall that any $X$-valued prestack can be written as a colimit given by $\colim_{\alpha}(v_{\alpha}\otimes B_{\alpha})$ for $v_{\alpha}$ a set of prestacks and $B_{\alpha}$ a set of generators for the presentable $\infty$-category $X$.  Therefore, $x\otimes F\simeq x\otimes\colim_\alpha(v_{\alpha}\otimes B_{\alpha})\simeq\colim_{\alpha}(x\otimes v_{\alpha})\otimes B_{\alpha}\simeq\colim_{n}(\colim_{\alpha}(u_*\otimes v_{\alpha})\otimes B_\alpha)\simeq\colim_{n}(u_*\otimes\colim_{\alpha}(v_\alpha\otimes B_\alpha))\simeq\colim_{n}(u_*\otimes F)$ and the result follows.
\end{proof}

Let $n\geq 0$ be an integer.  Recall that a $\inftyz$-category $A$ is said to be \textit{$n$-truncated} (resp. \textit{$n$-connective}) if for every $i>n$ (resp. $i<n$)
\[  \pi_i(A,a)\simeq *  \]
for all objects $a\in A$.  An $\inftyz$-category which is $1$-connective will be called \textit{connected}.  A map of $\inftyz$-categories $f:A\ra B$ is said to be $n$-truncated (resp. $n$-connective) if the homotopy fibers of $f$ taken over any base point of $B$ is $n$-truncated (resp. $n$-connective).  A prestack $F:C^{op}\ra\cS$ is said to be $n$-truncated (resp. $n$-connective) if $F(x)$ is $n$-truncated (resp. $n$-connective) for all $x\in C$.  A map of prestacks $F\ra G$ is said to be $n$-truncated (resp. $n$-connective) if $F(x)\ra G(x)$ is $n$-truncated (resp. $n$-connective) for all $x\in C$.

Let $C$ be an $\infty$-category and $n\geq 0$ an integer.  An object $x$ in $C$ is said to be \textit{$n$-truncated} if the representable prestack $C(\bullet,x)$ is $n$-truncated.  An arrow $f:x\ra y$ in $C$ is said to be \textit{$n$-truncated} if the map of prestacks $C(\bullet,f)$ is $n$-truncated.  Let $\tau_{\leq n}C$ denote the full subcategory of $C$ spanned by the $n$-truncated objects.  Then there exists an equivalence
$\tau_{\leq 0}C\ra\h(\tau_{\leq 0}C)$
of $\infty$-categories.  Note that objects in the $\infty$-category $\tau_{\leq 0}C$ are objects $x\in C$ such that for all $y\in C$, the $\inftyz$-category $C(y,x)$ is homotopy equivalent to a discrete space, ie. $C(y,x)\ra\pi_0C(y,x)$ is a homotopy equivalence for all $y\in C$.  We can construct a left adjoint to the inclusion functor $i:\tau_{\leq n}C\ra C$ given by
\[  t_{n}:C\ra\tau_{\leq n}C.   \]

\begin{prop}\label{stackcharacterisation}
Let $(C,\tau)$ be a site and 
\[     F:C^{op}\ra\Catinf    \]
a prestack of $\infty$-categories satisfying the following conditions~:
\begin{enumerate}
\item For each object $x$ in $C$, the $\infty$-category $F(x)$ admits limits.
\item For each object $x$ in $C$ and for any covering $\{u_i\ra x\}_{i\in I}$ in $cov^{\tau}(x)$, the functor $F(x)\ra F(u_{i})$ preserves limits. 
\item For each object $x$ in $C$ and for any covering $\{u_{i}\ra x\}_{i\in I}$ in $cov^{\tau}(x)$, the functor $F(x)\ra \prod_{i}F(u_{i})$ is conservative.
\item For any map $f:y\ra x$ in $C$, the functor $f^{*}:=F(f):F(x)\ra F(y)$ admits a right adjoint $f_{*}:F(y)\ra F(x)$.
\item For each pullback square
\begin{diagram}
y  &\rTo^{f}   &  x  \\
\uTo^{p}  &      &\uTo_{g}\\
z    &\rTo^{q}    & y'
\end{diagram}
in $C$, the natural morphism  $g^{*}f_{*}\Rightarrow q_{*}p^{*}$ is an equivalence in the $\infty$-category $\RHom(F(y),F(y'))$.
\end{enumerate}
Then $F$ is a stack of $\infty$-categories.
\end{prop}

\begin{proof}
We need to show that for any covering $u\ra x$ in $cov^{\tau}(x)$ the map
\[  F(x)\ra\underset{\Delta}{\lim}~F(u_{n})  \] 
is an equivalence of $\infty$-categories.  Consider the pullback diagram
\begin{diagram}
u  &\rTo^{f}   & x \\
\uTo^{p}  &      &\uTo_{f}\\
v    &\rTo^{q}    & u
\end{diagram}
in $C$ where $f\in cov^{\tau}(x)$.  Construct the section $\delta:u\ra u\times_{x}u$ of the map $q$ where $q\circ\delta=\id_{u}$.  Taking the nerve of the maps $f$ and $q$ and using (1) we obtain a homotopy commutative diagram
\begin{diagram}
F(x)  &\rTo^{A} &\underset{\Delta}{\lim}~F(u_{n}) \\
\dTo  &     &\dTo \\
F(u)  &\rTo^{A'}    &\underset{\Delta}{\lim}~F(v_{n}) 
\end{diagram}  
in $\Catinf$.  Using (2), (4) and (5) we obtain an adjoint homotopy commutative diagram
\begin{diagram}
F(x)  &\lTo^{B} &\underset{\Delta}{\lim}~F(u_{n}) \\
\dTo  &     &\dTo \\
F(u)  &\lTo^{B'}    &\underset{\Delta}{\lim}~F(v_{n})
\end{diagram}  
in $\Catinf$.  To complete the proof it will suffice to show that the unit and counit of the adjunction $A\dashv B$ are equivalences.  By (3) and the fact that both squares commute, we are able to check the corresponding statement for the adjunction $A'\dashv B'$.  This adjunction is an equivalence owing to the fact that for any covering admitting a section, the prestack $F$ satisfies descent.
\end{proof}

We now define the notion of a group object in an $\infty$-category.  We start with the more general notion of a groupoid object.  Let $C$ be an $\infty$-category with pullbacks.  A \textit{category object} in $C$ is a functor $F:\Dop\ra C$ such that for all $n\geq 0$, the canonical map
\[        F([n])\ra F([1])\times_{F([0])}\times\ldots\times_{F([0])} F([1])    \]
is an equivalence in $C$.  Let $\tu{Ct}(C)$ denote the full subcategory of $\RHom(\Dop,C)$ spanned by the category objects of $C$.  A category object $F$ in $C$ is said to be a \textit{groupoid object} in $C$ if it takes every partition $[2]=\{S\cup S'|S\cap S'=\{s\}, s\in S\}$ to a pullback square
\begin{diagram}
F([2])                  &\rTo             &F(S')\\
\dTo       &                         &\dTo\\
F(S)               &\rTo          &F(\{s\})
\end{diagram}
in $C$.  Let $\Gpd(C)$ denote the full subcategory of $\tu{Ct}(C)$ spanned by the groupoid objects of $C$.  We have an adjoint pair
\[     i:\Gpd(C)\rightleftarrows\tu{Ct}(C):j   \]
where $j(F)$ is the groupoid object of isomorphisms of a category object $F$ in $C$.  

\begin{dfn}
Let $C$ be an $\infty$-category and $G$ a groupoid object in $C$.  Then $G$ is said to be a \textit{group object} in $C$ if $G([0])$ is a terminal object in $C$.
\end{dfn}

Let $\Gp(C)$ denote the full subcategory of $\Gpd(C)$ spanned by the group objects of $C$.  The $\infty$-category $\Gp(\cS)$ will play the analogue of the category of groups in the $\infty$-categorical context.  If $C$ is an $\infty$-category then there exists an equivalence
\[  \Gp(\Pr(C))\ra\RHom(C^{op},\Gp(\cS))  \]
of $\infty$-categories (since limits in functor categories are computed pointwise).  If $(C,\tau)$ is a site then a $\Gp(\cS)$-valued stack will be called a \textit{group stack} on $C$.  The $\infty$-category $\Gp(\St(C,\tau))$ of group stacks on $C$ will be denoted $\Gp(C,\tau)$.  

\begin{ex}
If $C$ is the site $(\Aff_{R},\tau)$ of commutative $R$-algebras, we will denote the $\infty$-category $\Gp(\Aff_{R},\tau)$ by $\Gp(R,\tau)$.
\end{ex}

Let $(C,\tau)$ be a site.  A stack $F$ in $\St(C,\tau)$ is said to be \textit{locally non-empty} if for all $x\in C$ there exists a $\tau$-covering $u\ra x$ such that $F(u)$ is non-empty.  It is said to be \textit{locally connected} if $t_0(F)\ra *$ is an isomorphism (of sheaves of sets).  A morphism of prestacks $\phi:F\ra G$ is said to be a \textit{local equivalence} if it is \textit{fully faithful}, ie. $\phi_{x}:F(x)\ra G(x)$ is fully faithful for all $x\in C$, and \textit{locally essentially surjective}, ie. for all $x\in C$ and $a\in G(x)$ there exists a covering $\alpha:u\ra x$ such that $\alpha^{*}(a)$ is equivalent to $\alpha^{*}(\phi_{x}(b))$ (ie. an isomorphism in $\h G(u)$) for some $b\in F(x)$.  If $F$ and $G$ are stacks then $\phi$ is a local equivalence if and only if it is an equivalence of stacks.  

\begin{dfn}
Let $(C,\tau)$ be a site and $F$ a stack in $\St(C,\tau)$.  Then $F$ is said to be a \textit{gerbe} in $\St(C,\tau)$ if it is locally non-empty and locally connected.
\end{dfn}

The full subcategory of $\St(C,\tau)$ spanned by gerbes will be denoted by $\Ger(C,\tau)$.  A gerbe $G$ in $\Ger(C,\tau)$ is said to be \textit{neutral} if there exists a morphism $*\ra G$ in $\Ger(C,\tau)$.  

\begin{ex}
Let $C$ be a symmetric monoidal $\infty$-category and $R$ a commutative monoid object of $C$.  Then we will denote by $\Ger(R,\tau):=\Ger(\Aff_{R},\tau)$ the $\infty$-category of stacks with respect to the site $(\Aff_{R},\tau)$ of $R$-algebras in $C$.
\end{ex}

We now provide a requisite characterisation of a gerbe.  First we will need a small lemma.

\begin{lem}\label{mapofkancomplexes}
Let $f:C\ra D$ be a map of $\inftyz$-categories.  Then the following are equivalent.
\begin{enumerate}
\item The map $f$ is fully faithful, ie. for all $x,y\in C$, 
\[  x\times^{h}_{C}y=\Map_{C}(x,y)\xras\Map_{D}(fx,fy)=fx\times^{h}_{D}fy.  \]
\item The map $\delta:C\ra C\times_{D}^{h}C$ is an equivalence.
\item The map $\pi_{0}(C)\ra\pi_{0}(D)$ is a monomorphism and for all $x\in C$ and $i>0$, the map 
\[ \pi_{i}(C,x)\ra\pi_{i}(D,f(x))  \] 
is an equivalence.
\end{enumerate}
\end{lem}

\begin{proof}
$(1)\Rightarrow(2)$.  Since $f$ is fully faithful, the map
\[    C(x,y)\xras C(x,y)\times^{h}_{D(fx,fy)}C(x,y)    \]
is an equivalence so $\delta:C\ra C\times_{D}^{h}C$ is fully faithful.  For essential surjectivity, we need to show that any object $(x,y,\alpha:fx\xras fy)$ in $C\ra C\times^{h}_{D}C$ is equivalent to an object $\delta(z)=(z,z,\id_{z})$ for some $z\in C$.  Let $z=x$.  Since $f$ is fully faithful, we set $\beta:x\xras y$ and $\alpha=f(\beta)$.  $(2)\Rightarrow(3)$.  Let $d\in D$ and $(x,y)\in C\times_D^{h}C$ such that $d\simeq f(x)\simeq f(y)$.  Then $\delta^{-1}(x,y)$ is equivalent to the path space between $x$ and $y$ in $f^{-1}(d)$.  Thus $f^{-1}(d)$ is empty or contractible.  Thus $\pi_{0}(C)\ra\pi_{0}(D)$ is a monomorphism.  $(3)\Rightarrow(1)$.  This statement is clear.
\end{proof}

Recall that the \textit{classifying space} functor $\B:\Gp(\cS)\ra\cS$ is given by
\[  G \mapsto \B G:[n]\mapsto G_{n,n}.  \]
It admits a right adjoint $\Omega$ which sends an $\inftyz$-category $A$ to $\Omega(A):[n]\mapsto A^{\Delta_{*}^{n}}$.  Let $(C,\tau)$ be a site.  We construct the following \textit{classifying prestack} functor $\overline{\B}:\Gp(\Pr(C))\ra\Pr(C)$ by sending
\[   G \mapsto\overline{\B}G:x\mapsto \B(G(x))  \]
together with its right adjoint $\overline{\Omega}$, where $\overline{\Omega}(F):x\mapsto\Omega(F(x))$.  Finally, the \textit{classifying stack} functor
\[      \widetilde{\B}:\Gp(C,\tau)\ra\St(C,\tau)      \]   
is the stackification of $\overline{\B}$ and admits the right adjoint $\widetilde{\Omega}:\St(C,\tau)\ra\Gp(C,\tau)$. 

\begin{prop}\label{gerbecharacterisation}
Let $(C,\tau)$ be a site and $F$ a stack on $C$.  The following are equivalent~:
\begin{enumerate}
\item The stack $F$ is a gerbe.
\item The stack $F$ is locally equivalent to $\widetilde{\B}G$ for $G$ a group stack in $\Gp(C,\tau)$.
\end{enumerate}
\end{prop}

\begin{proof}
Let $\Kan_{0}$ be the category of Kan complexes with a single $0$-simplex.  We have a natural string of equivalences $\Gp(\cS)\simeq L(\Gp(\Kan))\simeq L(\Kan_{0})\simeq L(s\tu{Gp})$ of $\infty$-categories (see Corollary 6.4 of \cite{GJ}).  By Section 4 and 5 of Chapter 5 of \textit{loc. cit} there exists an adjunction  
\[      {W}:s\Gp\rightleftarrows\Kan_{0}:{\Omega}   \] 
where the construction ${W}G$ is a model for $\B G$ where $G\in s\tu{Gp}$.  Thus for a pointed stack $F$ and a $\Gp(\cS)$-valued prestack $G$, we have an equivalence 
\[   \Map_{\Pr(C)_{*}}(\overline{\B}G,F)\ra\Map_{\Gp(\Pr(C))}(G,\overline{\Omega}F)  \]
of $\inftyz$-categories where $\overline{\Omega}F:=\Aut(s)$ for $s\in F(*)$.  

We claim that if $F$ is a stack which is locally non-empty and locally connected with $F(*)\neq\emptyset$ then $F$ is locally equivalent to $\widetilde{\B} G$ for $G=\Aut(s)$, $s\in F(*)$.  By the equivalence above, the identity map $G\ra\Aut(s)$ corresponds to a map of prestacks
\[  \phi:\overline{\B}G\ra F      \]
sending $*$ to $s$.  But since $F$ is a stack, the universal property of stackification implies that $\phi$ is actually a map of stacks $\phi:\widetilde{\B}G\ra F$.  

It remains to show that $\phi$ is fully faithful and locally essentially surjective.  By Lemma~\ref{mapofkancomplexes}, fully faithfulness is equivalent to the condition that $\widetilde{\B}G\ra\widetilde{\B}G\times_{F}\widetilde{\B}G$ is an equivalence of stacks.  By the universal property if suffices to check it for a map of prestacks.  By Lemma~\ref{mapofkancomplexes} again, it suffices to check the two conditions of Lemma~\ref{mapofkancomplexes} part (3).  The first condition of (3) is clear.  The second condition follows from the fact that for all $x\in C$ we have $\pi_{i}(\overline{\B}G(x),*)\simeq\pi_{i-1}(G(x),*):=\pi_{i-1}(\overline{\Omega}F(x), s)\simeq\pi_{i}(F(x),s)$.  

Finally, since $F$ is locally non-empty and locally connected there always exists a $\tau$-covering $\alpha:u\ra x$ such that for $a\in F(x)$ the map $\alpha^{*}(a)\ra\alpha^{*}(\phi_{x}(*))$ is an equivalence. 
\end{proof}

\begin{dfn}\label{hopfdfn}
Let $C$ be a symmetric monoidal $\infty$-category and $R$ a commutative monoid object in $C$.  A \textit{Hopf $R$-algebra} in $C$ is a cogroup object $B$ in the symmetric monoidal $\infty$-category ${\CAlg}_{R}(C)$ of commutative $R$-algebras in $C$.
\end{dfn}

Let $\Hopf_R(C)$ denote the full subcategory of $\Comon({\CAlg}_{R}(C))$ spanned by the Hopf $R$-algebras in $C$.  We will call $B_{1}:=B([1])$ the \textit{underlying} $R$-algebra of $B$.  We have a well defined functor
\[  \Hopf:\CMon(C)\ra\Catinf    \]  
sending $R$ to $\Hopf_{R}(C)$.  When the $\infty$-category $C$ is clear from the context we will simply write $\Hopf(R)$ in place of $\Hopf_{R}(C)$.  By Example~\ref{cmonsymmetricmonoidal} the relative tensor product monoidal structure on $\CAlg_{R}(C)$ coincides with the coproduct.  

\begin{dfn}
Let $C$ be a symmetric monoidal $\infty$-category, $R$ a commutative monoid object in $C$ and $(\Aff_{R},\tau)$ a site.  A Hopf $R$-algebra in $C$ is said to be a \textit{$\tau$-Hopf $R$-algebra} if its underlying $R$-algebra is an element of $\tau(R)$.
\end{dfn}

We denote by $\Hopf(R,\tau)$ the $\infty$-category of $\tau$-Hopf $R$-algebras in $C$.  When $R$ is an $\Einfty$-ring, we will be primarily interested in the $\infty$-category of positive (resp. flat, finite) Hopf $R$-algebras in the $\infty$-category $\Sp$ of spectra.  See Section~\ref{topologies} for a definition of these topologies. 

\begin{prop}\label{specisfullyfaithful}
Let $C$ be a symmetric monoidal $\infty$-category, $R$ a commutative monoid object of $C$ and $(\Aff_{R},\tau)$ the site of $R$-algebras with respect to a subcanonical topology $\tau$.  Then the functor
\[    \Spec:\Hopf(R)\ra\Gp(R,\tau)   \]       
is fully faithful.
\end{prop}  

\begin{proof}
Since $\tau$ is subcanonical, the Yoneda embedding $\Aff_{R}\ra\Pr(\Aff_{R})$ factors through the subcategory of stacks and hence $\Spec:\Aff_{R}\ra\St(R,\tau)$ is fully faithful.  The tensor product in $\CAlg_{R}$ corresponds to the coproduct by Example~\ref{cmonsymmetricmonoidal} and the Yoneda lemma preserves limits so we have an induced fully faithful functor $\Spec:\Comon(\CAlg_{R})\ra\Mon(\St(R,\tau))$ on monoid objects.  Restricting to group-like objects we find that $\Spec:\Hopf(R)\ra\Gp(R,\tau)$ is fully faithful.
\end{proof}

Informally, we have a diagram
\begin{diagram}
\Aff_R   & \rTo          &\Hopf(R) \\          
   \dTo_{\Spec} &                     &\dTo_\Spec           \\  
\St(R,\tau)     & \rTo      &  \Gp(R,\tau)
\end{diagram}
of $\infty$-categories where the vertical arrows take an algebraic object to its corresponding affine geometric object and the horizontal arrows pass to the corresponding group objects.    

Let $C$ be an $\infty$-category with finite colimits and $X$ a cosimplicial object in $C$.  For any cosimplicial set $A$, we define a cosimplicial object in $C$ given by
\[     X\otimes A: [n]  \mapsto \underset{A_{n}}{\prod}X_{n}.          \]             
Let $h_{0},h_{1}:X\rightrightarrows Y$ be two arrows in $cC:=\RHom(\Delta,C)$.  A \textit{homotopy} between $h_{0}$ and $h_{1}$ is a map $h:X\otimes\Delta^{1}\ra Y$ such that 
\[  h\circ(\id_{X}\times i_{0})=h\circ(i_{0}\times\id_{X})=h_{0}  \]
\[  h\circ(\id_{X}\times i_{1})=h\circ(i_{1}\times\id_{X})=h_{1} \]
where $i_{0},i_{1}:\Delta^{0}\ra\Delta^{1}$ denote the inclusion maps.  A diagram 
\[     f:X\rightleftarrows Y:g  \] 
in $cC$ is said to be a \textit{homotopy equivalence} if there exists a map $k:X\otimes\Delta^{1}\ra X$ satisfying
\[  k\circ(\id_{X}\times i_{0})=k\circ(i_{0}\times\id_{X})=k_{0}:=\id \] 
\[   k\circ(\id_{X}\times i_{1})=k\circ(i_{1}\times\id_{X})=k_{1}:=g\circ f  \]
together with a map $l:Y\otimes\Delta^{1}\ra Y$ satisfying
\[  l\circ(\id_{Y}\times i_{0})=l\circ(i_{0}\times\id_{Y})=l_{0}:=\id \]     
\[  l\circ(\id_{Y}\times i_{1})=l\circ(i_{1}\times\id_{Y})=l_{1}:=f\circ g.  \]

\begin{lem}\label{hetoe}
Let $C$ be an $\infty$-category with finite limits.  Then the functor $\holim_{n}:cC\ra C$ takes homotopy equivalences in $cC$ to equivalences in $C$.
\end{lem}

\begin{proof}
Let $X$ be a cosimplicial object in $C$ and $A$ a simplicial set.  It will suffice to show that $\holim_n(X\otimes A)\simeq\holim_n(X)\otimes A$.  Let $\oX$ be a constant cosimplicial object in $C$.  We have that
\[   \holim_n(A\otimes\oX)\simeq\holim_n(A)\otimes\oX\simeq A\otimes\oX\simeq A\otimes\holim_n(\oX).  \]
Now let $X=\oX\otimes B$ for $B$ a simplicial set.  We have that
\[  \holim_n((\oX\otimes B)\otimes A)\simeq\holim_n(\oX\otimes(B\otimes A))\simeq\holim_n(\oX)\otimes(B\otimes A)\simeq\holim_n(\oX\otimes B)\otimes A.  \]
Finally, $\oX\otimes\Delta^{n}$ generates the $\infty$-category $cC$ by homotopy limits.  Therefore, setting $X\simeq\holim_\alpha(\oX_{\alpha}\otimes\Delta^{n_\alpha})$ we have
\begin{align*}
  \holim_n(X\otimes A)    &\simeq \holim_n(\holim_\alpha(\oX_{\alpha}\otimes\Delta^{n_\alpha})\otimes A) \\
                         &\simeq\holim_n(\holim_\alpha((\oX_{\alpha}\otimes\Delta^{n_\alpha})\otimes A)) \\
                          &\simeq\holim_\alpha(\holim_n(\oX_{\alpha}\otimes\Delta^{n_\alpha})\otimes A) \\
                          &\simeq\holim_\alpha(\holim_n(\oX_{\alpha}\otimes\Delta^{n_\alpha}))\otimes A\simeq\holim_n(X)\otimes A.  
\end{align*}
\end{proof}

\begin{notn}
Let $\Dplus$ be the category of augmented simplicial sets.  We define a categrory $\Delta_{-\infty}$ given as follows~:
\begin{itemize}
\item The set of objects $\Ob(\Delta_{-\infty})$ is given by $\Ob(\Dplus)$.
\item The set of maps $\Hom_{\Delta_{-\infty}}([n],[m])$ is the set of order preserving maps $f:\{-\infty\}\cup[n]\ra\{-\infty\}\cup[m]$ which preserve the base point $\{-\infty\}$ thought of as a least element of $\{-\infty\}\cup[p]$ for any $[p]\in\Delta_{-\infty}$.
\end{itemize}
\end{notn}

We have the natural sequence of inclusions $\Delta\subseteq\Dplus\subseteq\Delta_{-\infty}$ where $\Dplus$ is identified with the full subcategory of $\Delta_{-\infty}$ with the same set of objects and a where a map $f$ in $\Delta_{-\infty}$ belongs to $\Dplus$ if and only if $f^{-1}(-\infty)=\{-\infty\}$.  The forgetful functor $\Delta_{-\infty}\ra\Delta$ admits a left adjoint.  

Let $C$ be an $\infty$-category and let $c_{-\infty}C$ denote the $\infty$-category $\RHom(\Delta_{-\infty},C)$.  We have an induced adjunction
\[    c_{-\infty}C\rightleftarrows cC    \]
between $\infty$-categories.  Let $X$ be a cosimplicial object in $C$.  We let $\Dec_+$ denote the comonad $+\circ\Dec(X)$ on the $\infty$-category $cC$ of cosimplicial objects in $C$.

\begin{prop}\label{illusie}
Let $C$ be an $\infty$-category and $X$ a cosimplicial object in $C$.  Then there exists a homotopy equivalence $X_{0}\ra\Dec_{+}(X)$ between cosimplicial objects in $C$.
\end{prop}

\begin{proof}
See Proposition 1.4 of \cite{Il}.
\end{proof}

\begin{prop}\label{hopflimit}
Let $C$ be a symmetric monoidal $\infty$-category and $B$ a Hopf $R$-algebra in $C$.  Then there exists an equivalence
\[    R\ra\underset{n\in\Delta}{\holim}~B^{\otimes_{R}(n+1)}  \]
in the $\infty$-category $\CAlg_{R}$.
\end{prop}

\begin{proof}
This follows from Lemma~\ref{hetoe}, Proposition~\ref{illusie} and the equivalence $\Dec_{+}(B)_{n}=B_{n+1}\simeq B^{\otimes_{R}(n+1)}$.
\end{proof}

\begin{dfn}
Let $C$ be an $\infty$-category.  An augmented cosimplicial object $X:\Dplus\ra C$ is said to be \textit{split} if there exists a map $X\ra\Dec(X)$.  A cosimplicial object is said to be \textit{split} if it extends to a split augmented cosimplicial object.  Let $F:C\ra D$ be a functor.  A (augmented) cosimplicial object $X$ of $C$ is said to be \textit{$F$-split} if $F\circ X$ is split as a (augmented) cosimplicial object of $D$.
\end{dfn}

\begin{prop}\label{splitislimit}
Let $C$ be an $\infty$-category and $X:\Dplus\ra C$ a split augmented cosimplicial object in $C$.  Then $X$ is a limit diagram.
\end{prop} 

\begin{proof}
This is essentially (the dual of) Lemma 6.1.3.16 of \cite{Lu}.
\end{proof}

\section{The positive, flat and finite topologies}\label{topologies}

In the introduction we stated three duality theorems in both the neutralized and neutral contexts depending on the topology chosen on the underlying site of our group stack or gerbe.  We introduce and study these topologies in this section which go under the names positive, flat and finite.  The two latter topologies are shown to be subcanonical, the proof of which uses the result that the presheaf of $\infty$-categories sending an algebra in spectra to its $\infty$-category of modules is a sheaf of $\infty$-categories.

Recall that a module $M$ over an ordinary ring $R$ is said to be \textit{flat} if the functor $\bullet\otimes_{R}M:\Mod_{R}(\Ab)\ra\Mod_{R}(\Ab)$ is exact (ie. preserves finite limits and colimits).

\begin{dfn}\label{flattopologies}
Let $R$ be an $\Einfty$-ring and $A$ an $R$-algebra.  An $A$-module $M$ is said to be 
\begin{enumerate}
\item \textit{Positive} if the functor $\bullet\otimes_{A}M:\Mod_{A}\ra\Mod_{A}$ preserves anti-connective objects.
\item \textit{Flat} if the abelian group $\pi_{0}M$ is flat as a module over the ordinary commutative algebra $\pi_{0}A$ and for each $n\in\bb{Z}$, the map 
$      \pi_{n}A\otimes_{\pi_{0}A}\pi_{0}M\ra\pi_{n}M      $
is an isomorphism of abelian groups.
\item \textit{Finite} if the functor $\bullet\otimes_{A}M:\Mod_{A}\ra\Mod_{A}$ preserves all (small) limits.
\end{enumerate}
\end{dfn}

A map $A\ra B$ of $R$-algebras is said to be \textit{positive} (resp. \textit{flat}, \textit{finite}) if $B$ is positive (resp. flat, finite) when considered as an $A$-module.  If $R$ is a connective $\Einfty$-ring then every flat $R$-module is also connective.  If $R$ is a discrete $\Einfty$-ring then every $R$-module $M$ is flat if and only if $M$ is discrete and $\pi_{0}(M)$ is flat over $\pi_{0}(R)$ in the classical sense.  A module $M$ is finite over an $\Einfty$-ring if and only if it is perfect (see Proposition~\ref{limitsthendualizable}).  

Let $k$ be a commutative ring and $M$ and $N$ be $k$-modules.  Recall the construction of the abelian groups $\tu{Tor}_{n}^{k}(M,N)$ (see for example \cite{We}).  Recall also that a $k$-module $M$ is flat if and only if for any $k$-module $N$, the group $\tu{Tor}^{k}_{1}(M,N)=0$.  Let $R$ be a discrete $\Einfty$-ring and $M$ and $N$ be two discrete $R$-modules.  Then the canonical map
\[      \pi_{n}(M\otimes_{R}N)\ra\tu{Tor}_{n}^{\pi_{0}R}(\pi_{0}M,\pi_{0}N)    \]
is an isomorphism.  

Let $\Mod_{R}^{\geq 0}$ (resp. $\Mod_{R}^{fl}$, $\Mod_{R}^{fin}$) denote the full subcategory of $\Mod_{R}$ spanned by the positive (resp. flat, finite) $R$-modules.  These full subcategories are closed under taking tensor products and contain the unit object $R$ of $\Mod_{R}$.  Hence by Example~\ref{subsymmmonoidal}, these $\infty$-categories inherit a symmetric monoidal structure.  We deduce that the functor $\CMon(\Mod_{R}^{\geq 0})\ra\fE_{R/}$ is fully faithful and its essential image consists of positive maps $R\ra R'$.  Similarly statements hold for the flat and finite examples.  

\begin{lem}\label{topologylemma}
Let $R$ be an $\Einfty$-ring.
\begin{enumerate}
\item Maps of positive, flat and finite $R$-algebras are stable under composition. 
\item Let
\begin{diagram}
A &\rTo^{f} &B\\
\dTo &  &\dTo\\
C  &\rTo^{g}  &D
\end{diagram}
be a pushout in $\fE_{R/}$ of $R$-algebras.  If $f$ is positive (resp. flat, finite) then $g$ is positive (resp. flat, finite).
\end{enumerate}
\end{lem}

\begin{proof}
For part (1), let $A\ra B\ra C$ be two maps of $R$-algebras.  For any $A$-module $M$, there exists a natural equivalence 
\[        C\otimes_{B}(B\otimes_{A}M)\simeq C\otimes_{A}M         \]
showing that the functor $C\otimes_{A}\bullet$ is equivalent to the composition $C\otimes_{B}(B\otimes_{A}\bullet)$ of functors.  Since the composition of two functors preserving connective objects is connective this proves the positive part.   Since the composition of two exact functors is exact and the above equivalence is an isomorphism on $\pi_{0}$ objects, the flat case is satisfied.  Finally, the composition of two functors preserving limits preserves limits which proves the finite case.  To prove (2), observe that there exists a natural equivalence $D\simeq B\otimes_{A}C$.  Thus for any $C$-module $M$, there exists an equivalence 
\[  D\otimes_{C}M\simeq(B\otimes_{A}C)\otimes_{C}M\simeq B\otimes_{A}M   \]
of $B$-modules.  Following the argument above, this shows that if $f$ is positive (flat, finite) then $g$ is also.
\end{proof}

Let $R$ be an $\Einfty$-ring and $A\ra B$ a map of $R$-algebras.  Consider the base change functor 
\begin{align*}
B\otimes_{A}\bullet:\Mod_{A}&\ra\Mod_{B}\\
M  &\mapsto B\otimes_{A}M.
\end{align*}
A map of $R$-algebras $A\ra B$ is said to be \textit{conservative} if the base change functor $B\otimes_{A}\bullet$ is conservative, ie. $B\otimes_{A}M\simeq 0$ if and only if $M\simeq 0$.

\begin{dfn}
Let $R$ be an $\Einfty$-ring.  A finite family of maps $\{A\ra B_{i}\}_{i\in I}$ of $R$-algebras is said to be a \textit{positive} (resp. \textit{flat}, \textit{finite}) \textit{covering} if $A\ra B_{i}$ is positive (resp. flat, finite) and conservative for each $i\in I$.
\end{dfn}

\begin{prop}
Let $R$ be an $\Einfty$-ring.  The positive, flat and finite coverings define a topology on the $\infty$-category $\Aff_{R}$.
\end{prop}

\begin{proof}
The conservative property is clearly stable under composition and pushouts.  Thus the three cases can be deduced from Lemma~\ref{topologylemma}.  
\end{proof}

The positive, flat and finite topologies will be denoted by ``${\geq 0}$", ``$\tu{fl}$" and ``$\tu{fin}$" respectively.  The most important example of a sheaf with respect to these topologies in our context is the sheaf of $\infty$-categories sending an $R$-algebra to the $\infty$-category of modules over $R$.  We begin with the preasheaf
\[  \Mod:\Aff_{R}^{op}\ra\Catinf   \]
of $\infty$-categories sending $A$ to $\Mod_{A}$ and a map $f:A\ra B$ to $B\otimes_{A}\bullet$.                       

\begin{prop}\label{modisastack}
Let $R$ be an $\Einfty$-ring.  The functor $\Mod$ is a stack of $\infty$-categories over the site $\Aff_{R}$ with respect to the flat and finite topologies.
\end{prop}

\begin{proof}
We begin with the finite topology.  We will show that $\Mod:\Aff_{R}^{op}\ra\Catinf$ satisfies each of the conditions of Proposition~\ref{stackcharacterisation}.  For any $A\in\CAlg_R$, the $\infty$-category $\Mod_{A}$ has limits since $\Mod_{A}$ is presentable (and presentable $\infty$-categories admit all limits).  Given any $u:B\ra A$ in $\Aff_{R}$ the base change functor $u^{*}:\Mod_{A}\ra\Mod_{B}$ commutes with limits along $\Delta$ by virtue of the flat and finite topologies.  Its right adjoint $u_{*}$ is given by the conservative forgetful functor.  For any pushout square
\begin{diagram}
A  &\rTo^{u}  &B\\
\dTo^{v}  &   &\dTo_{u'}\\
C  &\rTo^{v'}   &D
\end{diagram}
in $\CAlg_R$ we have, for $M\in\Mod_{B}$,
\[  (v')^{*}u_*'(M)=(v')^{*}(M\otimes_{B}D)\simeq (v')^{*}(M\otimes_{A}C)\simeq v_{*}(M\otimes_{B}A)=v_{*}u^{*}(M) \]
where the first equivalence follows from the natural equivalence $B\coprod C\simeq B\otimes_{A}C$ in $\CAlg_R$ of Example~\ref{cmonsymmetricmonoidal}.  The proof of the flat case can be extracted from Lemma 2.2.2.13 of \cite{TVII} (the same arguments hold here).
\end{proof}

\begin{prop}\label{subcanonical}
Let $R$ be an $\Einfty$-ring.  The flat and finite topologies on $\Aff_{R}$ are subcanonical.  
\end{prop}

\begin{proof}
This follows from the fact that if $\Mod$ is a stack on $\Aff_{R}$ with respect to a topology $\tau$, then $\tau$ is subcanonical.  This can be seen as follows.  Assume $\Mod$ is a stack.  Then by definition we have an equivalence $\Mod_{A}\ra\lim_{\Delta}\Mod_{B_{*}}$ for any covering $B\ra A$ of $A$.  Thus for all $M\in\Mod_{A}$, the unit map $M\ra\lim_{\Delta}(M\otimes_{A}B_{*})$ is an equivalence.  Take $M=A$.  Then $A\ra\lim_{\Delta}B_{*}$ is an equivalence and for all $C\in\CAlg_R$ the composition
\[     \Map(C,A)\ra\Map(C,\lim_{\Delta}B_{*})\ra\lim_{\Delta}\Map(C,B_{*})     \]
is an equivalence.  Thus the representable prestack $h_{C}$ is a stack.  The result now follows from Proposition~\ref{modisastack}.
\end{proof}

Let $\tau\in\{\tu{fl},\tu{fin}\}$.  Since $\tau$ is subcanonical, we have a fully faithful morphism 
\[     \Aff_{R}\ra\St(R,\tau)  \]
of $\infty$-categories given by the Yoneda embedding.  We denote a stack in the essential image of this functor by $\Spec A$ for an $R$-algebra $A$.  A stack $F$ in $\St(R,\tau)$ is said to be \textit{affine} if $F\ra\Spec A$ is an equivalence of stacks for some $R$-algebra $A$.  An affine stack is called an \textit{affine group stack} if the affine stack is a group stack.

\section{Rigid $(\infty,1)$-categories}\label{rigid}

Here we describe what it means for an $\infty$-category to be rigid.  This amounts to every object being dualizable and is a strong condition which gives much of the Tannakian theory its flavour.  The notion of a dual object in an ordinary category has its origins in the example of the category of vector spaces~:  a vector space admits a dual if and only if the vector space is finite dimensional.  Thus the rigidification of an $\infty$-category, that is, discarding all objects that do not admit duals, can be thought of as the implementation of a ``finiteness condition" on its objects.  

Many of the $\infty$-categories we work with are ind-rigid.  That is, they are equivalent to the $\infty$-category of ind-objects of its full subcategory of dualizable objects.  For example the $\infty$-category of modules over an $\Einfty$-ring is ind-rigid~: any module over an $\Einfty$-ring is a given by a filtered colimit of rigid objects.  We prove the very useful projection formula for $\infty$-categories and later in this section show that under certain conditions, the $\infty$-category of endomorphisms of a functor valued in $R$-modules is an affine group stack. 

Recall that an object $y$ in a symmetric monoidal category $C$ is said to be a \textit{dual} of an object $x$ in $C$ if there exists maps $ev_{x}:x\otimes y\ra 1$ and $coev_{x}:1\ra y\otimes x$ such that the compositions 
\[  x\xra{id_x\otimes coev_x}x\otimes y\otimes x\xra{ev_x\otimes id_x}x  \]
\[  y\xra{coev_x\otimes id_y} y\otimes x\otimes y\xra{id_y\otimes ev_x}y  \]
coincide with the identity maps of $x$ and $y$.  Let $C$ now be a symmetric monoidal $\infty$-category.  An object $x$ in $C$ is said to be $\textit{dualizable}$ if it admits a dual when considered as an object of the symmetric monoidal category $\h C$.  We will denote the dual of an object $x$ by $\xv$.

\begin{dfn}\label{rigiddfn}
A symmetric monoidal $\infty$-category is said to be \textit{rigid} if all objects are dualizable.
\end{dfn}

See also Proposition 2.6 of \cite{TV3} for more equivalent characterisations of rigidity.  If $C$ is a symmetric monoidal $\infty$-category then the unit object $1_{C}$ is dualizable.  Moreover, the dualizable objects are stable by isomorphism in $\h C$ and stable by the tensor product.  We denote the full subcategory of $C$ consisting of dualizable objects by $C^{\text{\rig}}$.  Let $\Catinf^{\rig}$ denote the full subcategory of $\Catinf^{\otimes}$ spanned by the rigid $\infty$-categories.  By Theorem 2.10 and Lemma 2.11 of \cite{TV3} we deduce that there exist adjunctions
\begin{diagram}
\Catinf^\otimes  &\pile{ \rTo^{\tu{Fr}^{\rig}}  \\  \lTo_{i} }  & \Catinf^{\rig}  &\pile{ \rTo^{i}  \\  \lTo_{(\bullet)^{\rig}} }  & \Catinf^\otimes
\end{diagram}
of $\infty$-categories where $\tu{Fr}^{\rig}(C)$ is the free rigid $\infty$-category generated by the symmetric monoidal $\infty$-category $C$.  The right adjoint $(\bullet)^{\rig}$ will be called the \textit{rigidification} functor and $C^\rig$ the rigidification of $C$.

\begin{prop}\label{rigidiff}
Let $C$ be a symmetric monoidal $\infty$-category.  Then $C$ is rigid if and only if it satisfies the following conditions~:
\begin{enumerate}
\item The $\infty$-category $C$ is enriched over itself.
\item The map $\uHom(x,1)\otimes y\ra\uHom(x,y)$ is an equivalence in $C$ for all $x,y\in C$.
\end{enumerate}
\end{prop}

\begin{proof}
It is enough to prove this statement in $\h C$.  The classical statement can then be found for example in Section 2 of \cite{D2}.
\end{proof}

Let $C$ be a rigid $\infty$-category.  To any map $f:x\ra y$ in $C$ there corresponds a \textit{transpose} map given by the composition
\[   ^{t}f:x\xra{1\otimes coev_{x}}\yv\otimes x\otimes\xv\xra{1\otimes f\otimes 1}\yv\otimes y\otimes\xv\xra{ev_{y}\otimes 1}\xv.\]
Similarly, to any $f:\yv\ra\xv$ we associate a composition map
\[   f^{t}:x\xra{coev_{y}\otimes 1}y\otimes\yv\otimes x\xra{1\otimes f\otimes 1}y\otimes\xv\otimes x\xra{y\otimes ev_{x}}y.  \]
This induces an equivalence
\[    C(x,y)\ra C(\yv,\xv).   \]
of $\inftyz$-categories.  The following Lemma follows straightforwardly from Proposition~\ref{rigidiff}.

\begin{lem}
Let $C$ be a rigid $\infty$-category and $x,y,x',y'$ be objects in $C$.  Then the following hold. 
\begin{enumerate}
\item The map $\uHom(x,y)\otimes\uHom(x',y')\ra\uHom(x\otimes x',y\otimes y')$ is an equivalence in $C$.
\item The map $x\otimes y\ra\uHom(\xv,y)$ is an equivalence in $C$.  
\item The map $\uHom(x,y)^{\vee}\ra\uHom(y,x)$ is an equivalence in $C$.
\end{enumerate}  
\end{lem}

\begin{proof}
Set $\xv=\uHom(x,1)$.  For (1), we have a chain of equivalences 
\[  \uHom(x,y)\otimes\uHom(x',y')\simeq\xv\otimes(x')^\vee\otimes y\otimes y'\simeq\uHom(x\otimes x',y\otimes y')  \] 
since $\xv\otimes\yv$ is dual to $x\otimes y$.  For (2), we have a chain of equivalences 
\[    C(z,x\otimes y)\simeq C(\xv\otimes\yv,\zv)\simeq C(\xv,\uHom(\yv,\zv))\simeq C(\xv,\uHom(z,y))\simeq C(z,\uHom(\xv,y))   \]
which is functorial in $z$.  Finally, for (3), we have a chain of equivalences 
\[  \uHom(x,y)^{\vee}\simeq\uHom(\uHom(x,y),1)\simeq\uHom(\xv\otimes y,1)\simeq\yv\otimes x  \] 
due to (2).  Thus $\uHom(x,y)^\vee\simeq\uHom(y,x)$.
\end{proof}

\begin{prop}\label{rigidnaturaltransformation}
Let $C$ and $D$ be rigid $\infty$-categories and $F,G:C\ra D$ be two symmetric monoidal functors.  Then any map $\alpha:F\ra G$ is an equivalence.
\end{prop}

\begin{proof}
As shown in \cite{Sa}, an explicit inverse to $\alpha$ is given by the map $\beta:G\ra F$ making the following diagram 
\begin{diagram}
F(\xv)   & \rTo^{\alpha_{x^{\vee}}}          & G(\xv) \\          
   \dTo &                     &\dTo           \\  
F(x)^{\vee}     & \rTo^{^{t}\beta_{x}}         & G(x)^{\vee}
\end{diagram}
commute for all $x\in C$.
\end{proof}

\begin{dfn}
Let $C$ be a presentable symmetric monoidal $\infty$-category.  Then $C$ is said to be \textit{ind-rigid} if $\Ind(C^{\rig})\ra C$ is an equivalence of $\infty$-categories.
\end{dfn}

Let $R$ be an $\Einfty$-ring.  It follows from the proof of Proposition 7.2.5.2 of \cite{L1} that 
\[  \Ind(\ModRperf)\ra\Mod_{R}  \] 
is an equivalence of $\infty$-categories.  Combining this with the equivalence $\ModRperf\ra\Mod_R^\rig$ we find that the $\infty$-category $\Mod_{R}$ is ind-rigid.  Let $G=F\circ y:\ModRrig\ra\Mod_{R}$ denote the composition of $F:\Ind(\ModRrig)\ra\Mod_{R}$ with the Yoneda embedding.  It follows from Proposition 5.3.5.11 of \cite{Lu} that the set of objects $\{G(M)\}_{M\in\ModRrig}$ generate $\Mod_{R}$ under filtered colimits.

\begin{prop}\label{projectionformula} 
Let $C$ and $A$ be symmetric monoidal $\infty$-categories, $f:C\ra A$ a symmetric monoidal functor and $g$ its right adjoint.  Assume $C$ is ind-rigid.  Then for any object $x$ in $C$ and $a$ in $A$, the map 
\[      g(a)\otimes x\ra g(a\otimes f(x))     \] 
is an equivalence.
\end{prop}

\begin{proof}
Let $x$ be a dualizable object of $C$ and $y$ be an arbitrary object of $C$.  Then $C(y,ga\otimes x)\cong C(y\otimes\xv,ga)\cong A(f(y\otimes\xv),a)\cong A(fy\otimes f\xv,a)\cong A(fy\otimes(fx)^\vee,a)\cong A(fy,a\otimes fx)\cong C(y,g(a\otimes fx))$.  Since any object in $C$ is given by a colimit of dualizable objects by assumption and the above demonstration is functorial in $x$, the result follows.
\end{proof}

The equivalence in Proposition~\ref{projectionformula} is often called the \textit{projection formula}.  Setting $a=1$, $f(x)=b$ and applying $f$ to the projection formula gives the equivalence $fg(1)\otimes b\simeq fg(b)$.

\begin{prop}\label{limitsthendualizable}
Let $R$ be an $\Einfty$-ring.  If the functor $\bullet\otimes_{R}M:\Mod_{R}\ra\Mod_{R}$ commutes with limits then $M$ is dualizable.
\end{prop}

\begin{proof}
Let $X$ and $Y$ be $R$-modules.  We can write any $R$-module as a colimit of perfect $R$-modules so we set $X=\colim_{\alpha}X_{\alpha}$.  Assume that the functor $\bullet\otimes_{R}M$ commutes with limits and recall that the $\infty$-category $\ModRperf$ is equivalent to $\ModRrig$.  We have that 
\begin{align*}
 \uHom(X,Y)\underset{R}{\otimes}M  &\simeq\underset{\alpha}{\lim}~\uHom(X_{\alpha},Y)\otimes M \\
                                            &\simeq\underset{\alpha}{\lim}~(\uHom(X_{\alpha},Y)\otimes M)  \\    
                                            &\simeq\underset{\alpha}{\lim}~(X_{\alpha}^{\vee}\otimes Y\otimes M) \\
                                            &\simeq\underset{\alpha}{\lim}~\uHom(X_{\alpha},Y\otimes M)\simeq\uHom(X,Y\otimes M).  
\end{align*}
Setting $X=M$ and $Y=R$ we see that $M^{\vee}\otimes_{R}M\simeq\uHom(M,M)$ so $M$ is dualizable by Proposition~\ref{rigidiff}.                                             
\end{proof}

\begin{lem}\label{indrigidlemma}
Let $C$ and $D$ be presentable symmetric monoidal $\infty$-categories.  Assume that $C$ is ind-rigid.  Then there exists an equivalence
\[    \uHom^{\otimes}(C,D)\ra\uHom^{\otimes}(C^{\rig},D^{\rig})     \]
of $\infty$-categories.
\end{lem}

\begin{proof}
By the universal property of ind-objects, the map $\uHom^{\otimes}(C,D)\ra\uHom^{\otimes}(C^{\rig},D)$ is an equivalence.  The result now follows from the fact that symmetric monoidal functors preserve rigid objects.
\end{proof}

Let $F:C\ra D$ be a functor between $\infty$-categories.  Then we denote by $\End(F)$ the mapping space $\Map(F,F)$ taken in the $\infty$-category $\RHom(C,D)$.  If $C$ and $D$ are symmetric monoidal $\infty$-categories we let $\End^{\otimes}(F)$ denote the mapping space $\Map(F,F)$ in $\RHom^{\otimes}_{\Gamma}(C,D)$.  

Let $R$ be an $\Einfty$-ring.  First recall that given two rigid $R$-modules $M$ and $N$, the mapping space $\Map_{\Mod_R}(M,N)$ as a functor on the $\infty$-category $\CAlg_R$ of $R$-algebras is given by
\[   \Map(M,N)(A):=\Map(M\otimes_{R}A,N\otimes_{R}A).   \]
This functor is representable by the chain of equivalences     
\[  \Map_{\Mod_{R}}(M\otimes_{R}A,N\otimes_{R}A)\simeq\Map_{\Mod_{R}}(M\otimes_{R}N^{\vee},A)\simeq\Map_{\CAlg_R}(\tu{Fr}(M\otimes_{R}N^{\vee}),A)  \]
where the second equivalence follows from Proposition~\ref{rigidiff} and the third follows from the equivalence 
\[  \CAlg_R(\tu{Fr}(M),\bullet)\simeq\Mod_{R}(M,\bullet)  \] 
arising from the adjunction $\CAlg_R\dashv\Mod_R$ (see Section~\ref{monoidal}).

\begin{lem}\label{graphlem}
Let $R$ be an $\Einfty$-ring, $C$ be a symmetric monoidal $\infty$-category and $F:C\ra\Mod_R^{\rig}$ a symmetric monoidal functor.  Then $\End^{\otimes}(F)$ is representable.
\end{lem}

\begin{proof}
Any symmetric monoidal $\infty$-category is of the form $\hocolim_\alpha C_{\alpha}$ where $C_\alpha$ is the free symmetric monoidal $\infty$-category over an $\infty$-graph $G_\alpha$ defined through the following universal property~: for any symmetric monoidal $\infty$-category $D$, there exists an equivalence
\[  \RHom^{\otimes}(C_\alpha,D)\simeq\RHom^{\tu{grph}}(G_\alpha,D)  \]
where $\RHom^{\tu{grph}}$ denotes the $\infty$-category of functors between $\infty$-graphs.  Thus we have an equivalence
\begin{align*}
\RHom^{\otimes}(C,\Mod_R^{\rig}) &\ra\underset{\alpha}{\holim}~\RHom^{\tu{grph}}(G_\alpha,\Mod_R^\rig)\\
                                F &\mapsto F_\alpha    
\end{align*}                                
where $F_\alpha$ sends an object $x$ in $G_\alpha$ to a rigid module $M_x$ and the mapping space $G_\alpha(x,y)$ to the mapping space $\Mod_R^\rig(M_x,M_y)$.  

When $G_\alpha$ consists of a single object $x$, then $\End(F_\alpha)\simeq\End(M_x)=\Spec\tu{Fr}(M_x\otimes_R(M_x)^\vee)$.  When $G_\alpha$ consists of two objects $x$ and $y$ and the simplicial set $A$ of arrows between $x$ and $y$ then 
\[     \End(F_\alpha)=\End(M_x)\times^h_{\Map(M_x,M_y)^{A}}\End(M_y).  \]
Since representable objects are stable under homotopy limits and any $\infty$-graph is generated under homotopy colimits by the above two simple graphs, the functor $\holim_\alpha\End^{\otimes}(F_\alpha)\simeq\End^\otimes(F)$ is representable.
\end{proof}

\begin{prop}\label{endisagroupstack}
Let $R$ be an $\Einfty$-ring, $C$ be a presentable ind-rigid symmetric monoidal $\infty$-category and $F:C\ra\Mod_{R}$ a symmetric monoidal functor.  Then $\End^{\otimes}(F)$ is a representable $\Gp(\cS)$-valued prestack.  Hence it is an affine group stack with respect to any subcanonical topology.
\end{prop}

\begin{proof}
Since the $\infty$-category $C$ is ind-rigid, the map $\RHom^{\otimes}(C,\Mod_{R})\ra\RHom^{\otimes}(C^{\rig},{\Mod}_{R}^{\rig})$ is an equivalence by Lemma~\ref{indrigidlemma}.  Thus we have an equivalence $\End^{\otimes}(F)\simeq\End^{\otimes}(F^{\rig})$. By Lemma~\ref{graphlem}, $\End(F^\rig)$ is representable.  Finally, $\End^{\otimes}(F)\simeq\RAut^{\otimes}(F)$ by Proposition~\ref{rigidnaturaltransformation} so $\End^{\otimes}(F)$ is in fact a representable $\Gp(\cS)$-valued prestack and hence an affine group stack for any subcanonical topology.
\end{proof}

\section{Neutralized Tannaka duality for $(\infty,1)$-categories}\label{tannakadualityforinftycategories}

We begin by constructing the $\infty$-category of $R$-linear, stable, presentable, symmetric monoidal $\infty$-categories which we will call $R$-tensor $\infty$-categories.  In Example~\ref{cartesianstructureoncatinfty} we saw an explicit construction of the cartesian monoidal structure on the $\infty$-category $\Catinf$ of $\infty$-categories.  Let ${\Cat}_{\infty}^{p,\perp}$ denote the full subcategory of ${\Cat}_{\infty}^{p}$ spanned by presentable $\infty$-categories which are moreover stable $\infty$-categories.  The projection ${\Cat}_{\infty}^{\perp,p}\ra \Gamma$ determines a symmetric monoidal structure on the $\infty$-category $\Catinf^{p,\perp}$ of stable, presentable $\infty$-categories (see Proposition 6.3.2.15 of \cite{L1}).  The unit object of $\Catinf^{p,\perp}$ with this symmetric monoidal structure is the $\infty$-category $\Sp$ of spectra.  

Let $\Catinf^{p,\perp,\otimes}$ denote the subcategory of $\Catinf^{\otimes}$ spanned by stable, presentable symmetric monoidal $\infty$-categories whose symmetric monoidal bifunctor preserves colimits seperately in each variable and whose morphisms are colimit preserving symmetric monoidal functors.  Then we have an equivalence
\[      \CMon({\Cat}_{\infty}^{p,\perp})\ra\Catinf^{p,\perp,\otimes}    \]
of $\infty$-categories.  Thus a symmetric monoidal $\infty$-category $C$ belongs to $\CMon({\Cat}_{\infty}^{p,\perp})$ if and only if $C$ is stable, presentable and the bifunctor $\otimes:C_{\<1\>}\times C_{\<1\>}\ra C_{\<1\>}$ preserves colimits separately in each variable.    

Let $R$ be an $\Einfty$-ring.  Then the stable, presentable symmetric monoidal $\infty$-category ${\Mod}_{R}$ of $R$-modules belongs to $\CMon({\Cat}_{\infty}^{p,\perp})$.  Applying this observation to the results of Section~\ref{monoidal} gives the equivalence
\[   \CMon({\Mod}_{\Mod_{R}}({\Cat}_{\infty}^{p,\perp}))\simeq\CMon({\Cat}_{\infty}^{p,\perp})_{\Mod_{R}/}. \]
of $\infty$-categories.  The term on the left hand side is the $\infty$-category of $R$-linear, stable, presentable symmetric monoidal $\infty$-categories and $R$-linear symmetric monoidal functors.   

\begin{dfn}
Let $R$ be an $\Einfty$-ring.  A symmetric monoidal $\infty$-category is said to be a \textit{tensor} $\infty$-category if it is stable and presentable.  It is said to be an \textit{$R$-tensor} $\infty$-category if it is $R$-linear, stable and presentable. 
\end{dfn}

We will denote the $\infty$-category of tensor $\infty$-categories by $\Tens^{\otimes}:=\Catinf^{p,\perp,\otimes}$ and the $\infty$-category of $R$-tensor $\infty$-categories and $R$-linear symmetric monoidal functors by 
\[  \Tens_R^{\otimes}:=(\Catinf^{p,\perp,\otimes})_{\Mod_{R}/}.  \]  

We now introduce the stack of fiber functors and state our duality theorems for neutralized higher Tannaka duality.  In the next section, we will describe the proofs.  Let $R$ be an $\Einfty$-ring and $\tau\in\{fl,fin\}$.  The stack $\Mod$ of modules of Proposition~\ref{modisastack} can naturally be extended to act on the $\infty$-category of stacks $\St(R,\tau)$ as follows.  The objects of $\St(R,\tau)$ can be considered as stacks associated to the functors
\[   F:\CAlg_{R}\ra\Catinf   \]
taking values in the $\infty$-category of $\infty$-categories using the inclusion of Remark~\ref{inclusionofstacks}.  The action of $\Mod$ on the $\infty$-category of stacks $\St(R,\tau)$ is then given by
\begin{align*}
\Mod:\St(R,\tau)^{op}   &\ra   \Catinf\\
   F &\ra \Mor(F,\Mod)
\end{align*}
which by abuse we also denote $\Mod$ where $\Mor$ is the morphism object of $\St_{\Catinf}(R,\tau)$ given by Proposition~\ref{xstacksareenrichedinx}.  

The $\infty$-category $\Mor(F,\Mod)$ is naturally endowed with the structure of an $R$-tensor $\infty$-category.  The tensor $\infty$-structure on $\Mor(F,\Mod)$ is induced from that on $\Mod$~: it is presentable, stable and is given the pointwise symmetric monoidal structure 
\[  {\Mor}(F,\Mod)_{\<n\>}:=\Mor(F,{\Mod})\times_{\Mor(F,\underline\Gamma)}\{\<n\>\}   \]
of Example~\ref{pointwisemonoidalstructure} where ${\Mod}_{\<n\>}(A):=({\Mod}_{A})_{\<n\>}$ for an $R$-algebra $A$ and $\underline\Gamma$ is the constant prestack.  The $R$-linear structure on $\Mor(F,\Mod)$ is also induced from that on $\Mod$ through the composition
\[   \Mod_{R}\xra{\psi}\Mor(F,\underline{\Mod}_{R})\ra\Mor(F,\Mod) \]
where $\underline{\Mod}_{R}$ is the constant prestack and $\psi$ is the natural constant map.  Thus we obtain a functor
\[     \Mod:\St(R,\tau)^{op}\ra\Tens_{R}^\otimes.  \]

\begin{dfn}
Let $R$ be an $\Einfty$-ring.  A dualizable object of $\Tens^\otimes_R$ will be called a \textit{rigid $R$-tensor} $\infty$-category.
\end{dfn}

Let $\Tens_{R}^{\rig}:=(\Tens^\otimes_R)^\rig$ denote the $\infty$-category of rigid $R$-tensor $\infty$-categories.  Restricting the functor $\Mod$ to rigid objects we obtain a natural functor
\begin{align*}
\Perf:\St(R,\tau)^{op}  &\ra \Tens_{R}^{\rig}\\
                        F            &\mapsto\Mor(F,\Perf)               
\end{align*}
where the stack $\Perf:\CAlg_{R}\ra\Tens_{R}^{\rig}$ on the right hand side sends a commutative $R$-algebra $A$ to the $\infty$-category $\Mod_{A}^{\rig}$ of rigid $A$-modules.

\begin{lem}\label{leftadjointtoperf}
The functor $\Perf$ admits a left adjoint.
\end{lem}

\begin{proof}
Let $C$ be a rigid $R$-tensor $\infty$-category.  We have the following chain of equivalences 
\[   \Map_{\Tens_{R}^{\rig}}(C,\Perf(F))\simeq\Map_{\St(R,\tau)}(C\times F,\Perf)\simeq\Map_{\St(R,\tau)}(F,\uHom(C,\Perf)). \] 
Here $\uHom(C,\Perf)$ is a stack by Proposition~\ref{homstack}, since $\Perf$ is a stack, where we regard $C$ as a constant prestack.  
\end{proof}

The left adjoint to $\Perf$ of Lemma~\ref{leftadjointtoperf} will be denoted
\begin{align*}
   \Fib: \Tens_{R}^{\rig} &\ra\St(R,\tau)^{op}  \\    
                                C     &\mapsto\uHom(C,\Perf)   
\end{align*}   
where $\Fib(C)(A):=\Map_{\Tens_{R}^{\rig}}(C,\Mod_{A}^{\rig})$ for a commutative $R$-algebra $A$.  

We would now like to consider conditions on rigid $R$-tensor $\infty$-categories and stacks on certain sites of $R$-algebras for which the adjunction $\Fib\dashv\Perf$ is an equivalence.  We begin with some preliminary definitions and results.

\begin{dfn}\label{secomod}
Let $R$ be an $\Einfty$-ring.  The $\infty$-category of \textit{Segal comodules} over a Hopf $R$-algebra $B$ is given by the following limit
\[  \SeComod_{B}:=\underset{n\in\Delta}{\lim}~\Mod_{B_n} \]            
of $\infty$-categories.
\end{dfn}

We will often abuse terminology by calling a Segal comodule over a Hopf $R$-algebra simply a comodule over a Hopf $R$-algebra.  This is justified since by Theorem 6.2.4.2 of \cite{L1}, there exists an equivalence
\[     \SeComod_B\ra\Comod_B  \]
of $\infty$-categories.  Restricting to rigid objects we have an identification $\SeComod^{\rig}_{B}:=\lim_n\Mod^{\rig}_{B_n}$.  The forgetful functor $\SeComod\ra\Mod_{R}$ is given by the evaluation map $ev_{0}:\lim_{n}\Mod_{B_n}\ra\Mod_{B_{0}}=\Mod_{R}$.

\begin{ex}
The object $\Dec_+(B)$ is a $B$-comodule which is just $B$ thought of as a comodule over itself.  More precisely, a $B$-comodule, by definition, consists of objects $M_n\in\Mod_{B_n}$ for all $[n]\in\Delta$ and for every arrow $[n]\ra[m]$ in $\Delta$, an equivalence $M_n\otimes_{B_n}B_m\xras M_m$ in $\Mod_{B_m}$.  We have that $\Dec_+(B)_n=B_{n+1}$ and $B_{n+1}\otimes_{B_n}B_m\simeq B^{\otimes_{R}n+1}\otimes_{B^{\otimes n}}B^{\otimes m}\simeq B^{\otimes m+1}$ where the first equivalence follows from Segal maps of the Hopf $R$-algebra structure on $B$.  Thus $\Dec_+(B)$ is a $B$-comodule.
\end{ex}

\begin{prop}\label{secomodmodbg}
Let $R$ be an $\Einfty$-ring and $B$ a Hopf $R$-algebra.  Then there exists an equivalence
\[  \SeComod_{B}\simeq\Mod(\widetilde{\B}G) \]
of $\infty$-categories for an affine group stack $G=\Spec B$.  
\end{prop}

\begin{proof}
Firstly, we have that $\wtB G=|\wtB G_\bullet|$ where $\wtB G_{\bullet}:\Dop\ra\St(C,\tau)$ is the functor that sends $[n]$ to $G_n$.  Therefore 
\[   \widetilde{\B}\Spec B=\colim_{n\in\Delta}\Spec B_n.  \]
We thus have an equivalence
\[ \Mod(\wt{\B}G):=\Mor(\wt{\B}\Spec B,\Mod)=\Mor(\colim_{n\in\Delta}\Spec B_{n},\Mod)\simeq\lim_{n\in\Delta}\Mod_{B_n}=:\SeComod_{B} \]
given by the Yoneda Lemma.
\end{proof}

Let $R$ be an $\Einfty$-ring and $B$ a rigid Hopf $R$-algebra.  A corollary of Proposition~\ref{secomodmodbg} is that there exists an equivalence
\[  \SeComod^{\tu{rig}}_{B}\simeq\Perf(\widetilde{\B}G) \]
of $\infty$-categories for an affine group stack $G=\Spec B$.  We call $\Perf(\widetilde{\B}G)$ the $\infty$-category of \textit{representations} of $G$ and denote it by
\[ \Rep(G):=\Perf(\widetilde{\B}G).   \]     
for $G=\Spec B$ where $B$ is a Hopf $R$-algebra in the $\infty$-category $\Sp$ of spectra. 

\begin{dfn}\label{pointednotation}
The $\infty$-category of rigid $R$-tensor $\infty$-categories over $\Mod_R^\rig$ will be denoted 
\[   (\Tens_R^{\rig})_{*}:=(\Tens_{R}^{\rig})_{/\Mod_{R}^{\rig}}.  \]
The objects of $(\Tens_{R}^{\rig})_{*}$ will be described as pairs $(T,\omega)$ where $T$ is a rigid $R$-tensor $\infty$-category and $\omega:T\ra\Mod_{R}^{\rig}$ is an $R$-tensor functor.  They will be called \textit{pointed rigid $R$-tensor $\infty$-categories}.  Let  
\[  \Fib_{*}:(\Tens_{R}^{\rig})_{*}\ra\Gp(R,\tau)^{op}  \] 
be the functor defined by $\Fib_{*}(T,\omega):=\End^{\otimes}(\omega)$.  Let 
\[  \Perf_{*}:\Gp(R,\tau)^{op}\ra(\Tens_{R}^{\rig})_{*}  \] 
be the functor defined by $\Perf_{*}(G):=(\Rep(G),\nu)$ where $\nu:=f^{*}:\Rep(G)\ra\Mod_{R}^{\rig}$ is the functor induced by the natural map $f:*\ra\wt{\B}G$.
\end{dfn}

\begin{lem}\label{pointedadjunction}
The maps of Definition~\ref{pointednotation} induce an adjunction
\[       \Fib_{*}:(\Tens_{R}^{\rig})_{*}\rightleftarrows\Gp(R,\tau)^{op}:\Perf_{*}   \]
of $\infty$-categories.
\end{lem}   

\begin{proof}
Consider the homotopy pullback diagram
\begin{diagram}
\Map_{*}(\wt{\B}G,\Fib(T))   &\rTo  &\Map_{\St(R,\tau)}(\wt{\B}G,\Fib(T))\\
\dTo      &        &\dTo\\
*            &\rTo^{\omega}       &\Map_{\St(R,\tau)}(*,\Fib(T))
\end{diagram}
and its corresponding adjoint diagram
\begin{diagram}
\Map_{*}(T,\Perf(\wt{\B}G))   &\rTo  &\Map_{\Tens_{R}^{\rig}}(T,\Perf(\wt{\B}G))\\
\dTo     &        &\dTo\\
*           &\rTo^{\omega}       &\Map_{\Tens_{R}^{\rig}}(T,\Mod_R^\rig).
\end{diagram}
using Lemma~\ref{leftadjointtoperf}.  Since the two diagrams are equivalent, the homotopy pullbacks are equivalent.  Thus we have a chain of equivalences
\[  \Map(G,\Fib_{*}(T))\simeq\Map_{*}(\wt{\B}G,\Fib(T))\simeq\Map_*(T,\Perf(\wt{\B}G))  \]
where the first map arises from the adjunction $\wt{\B}\dashv\wt{\Omega}_{*}$.
\end{proof}

We now state the main results of the paper.  They will be proven in the next section.  We will begin with the pointed case, otherwise known as \textit{neutralized} Tannaka duality for $\infty$-categories.  We would like to study conditions on pointed rigid $R$-tensor $\infty$-categories and group stacks for which the adjunction of Lemma~\ref{pointedadjunction} is an equivalence of $\infty$-categories.  We make use of the positive, flat and finite topologies introduced in Definition~\ref{flattopologies}.  

We begin by defining the appropriate subcategory of pointed rigid $R$-tensor $\infty$-categories which we call \textit{Tannakian}.

\begin{dfn}\label{fiberfunctors}
Let $R$ be an $\Einfty$-ring, $C$ a rigid $R$-linear symmetric monoidal $\infty$-category and $\omega:C\ra\Mod_{R}^{\rig}$ an $R$-linear symmetric monoidal functor.  Denote $\Ind(\omega)$ by $\who$.  Then $\omega$ is said to be~:
\begin{enumerate}
\item A \textit{finite fiber functor} if $\who$ is conservative and preserves (small) limits.

Let $R$ be a connective $\Einfty$-ring.  Then $\omega$ is said to be~:

\item A \textit{flat fiber functor} if $\who$ is conservative, creates a t-structure on $C$, is exact and whose right adjoint is t-exact.

Let $R$ be a connective bounded $\Einfty$-ring.  Then $\omega$ is said to be~:

\item A \textit{positive fiber functor} if $\who$ is conservative, creates a t-structure on $C$ and is exact.
\end{enumerate}
\end{dfn}

Note that since $\who$ is a presentable symmetric monoidal functor it commutes with colimits and hence by the adjoint functor theorem admits a right adjoint $\whp$ which is a lax symmetric monoidal functor.  Also, we remark that since positive and flat fiber functors are conservative, t-exact and defined over a connective bounded base $\Einfty$-ring, the t-structures created are non-degenerate.

\begin{dfn}\label{pointeddfn}
Let $R$ be an $\Einfty$-ring.  A \textit{pointed $R$-Tannakian $\infty$-category} with respect to $\tau$ is a pair $(T,\omega)$ where $T$ is a rigid $R$-tensor $\infty$-category and $\omega:T\ra\Mod_{R}^{\rig}$ is a $\tau$-fiber functor.
\end{dfn}

Let $(\Tan_{R}^{\tau})_{*}$ denote the full subcategory of $(\Tens_{R}^{\rig})_{*}$ spanned by pointed $\tau$-$R$-Tannakian $\infty$-categories.  We will often abuse terminology by referring to a pointed $R$-Tannakian $\infty$-category $(T,\omega)$ as simply $T$.  

A notion of rigidity manifests itself on the opposite side of the duality in the following sense.  Let 
\[  \who_G:\Mod(\wtB G)\ra\Mod_R  \] 
and $\omega_G:\Perf(\wtB G)\ra\Mod_R^\rig$ be the forgetful functors.
 
\begin{dfn}
Let $R$ be an $\Einfty$-ring.  An affine group stack $G=\Spec B$ in $\Gp(R,\tau)$ is said to be 
\begin{enumerate}
\item \textit{Weakly rigid} if $\End^{\otimes}(\who_G)\ra\End^{\otimes}(\omega_G)$ is an equivalence.
\item \textit{Rigid} if the map $\SeComod_{B}\ra\Ind(\SeComod_{B}^{\rig})$ is an equivalence of $\infty$-categories.
\end{enumerate}
\end{dfn}

\begin{lem}\label{rigidthenweaklyrigid}
Let $R$ be an $\Einfty$-ring and $G=\Spec B$ be a rigid affine group stack.  Then $G$ is weakly rigid.
\end{lem}

\begin{proof}
Since $G$ is rigid, the $\infty$-category $\SeComod_B$ is ind-rigid and by Lemma~\ref{indrigidlemma} there exists an equivalence 
\[   \uHom^\otimes(\SeComod_B,\Mod_R)\simeq\uHom^\otimes(\SeComod_B^\rig,\Mod_R^\rig)   \] 
of $\infty$-categories.  The result now follows from Proposition~\ref{secomodmodbg}.
\end{proof}  

We remark that these rigidity conditions on an affine group stack (or on a Hopf $R$-algebra) are a new feature of the higher categorical approach~: in the classical case, an affine group scheme over a field is automatically rigid.  We will now define the objects on the group side of the correspondence.  

\begin{dfn}\label{tannakiangroupstack}
Let $R$ be an $\Einfty$-ring.  A group stack $G$ in $\Gp(\Pr(\Aff_R))$ is said to be \textit{$R$-Tannakian} if it is of the form $\Spec B$ for $B$ a weakly rigid Hopf $R$-algebra.  It is said to be \textit{$\tau$-$R$-Tannakian} for a topology $\tau$ if it is $R$-Tannakian where $B$ is a $\tau$-Hopf $R$-algebra.
\end{dfn}

Let $\TGp(R,\tau)$ denote the full subcategory of $\Gp(R,\tau)$ spanned by $\tau$-$R$-Tannakian group prestacks.  

Our first main theorem is the following generalization of neutralized Tannaka duality to the $\infty$-categorical setting.

\begin{thm}[Neutralized $\infty$-Tannaka duality]\label{pointedinftytannakaduality}
Let $\tau$ be a subcanonical topology.  Then the map
\[    \Perf_{*}:\TGp(R,\tau)^{op}\ra(\Tens^{\rig}_R)_*   \]
is fully faithful.  Moreover, the adjunction $\Fib_*\dashv\Perf_*$ induces the following~:
\begin{enumerate}
\item Let $R$ be an $\Einfty$-ring.  Then $(T,\omega)$ is a pointed finite $R$-Tannakian $\infty$-category if and only if it is of the form $\Perf_*(G)$ for $G$ a finite $R$-Tannakian group stack.
\item Let $R$ be a connective $\Einfty$-ring.  Then $(T,\omega)$ is a pointed flat $R$-Tannakian $\infty$-category if it is of the form $\Perf_*(G)$ for $G$ a flat $R$-Tannakian group stack.
\item Let $R$ be a bounded connective $\Einfty$-ring.  Then $(T,\omega)$ is a pointed positive $R$-Tannakian $\infty$-category if it is of the form $\Perf_*(G)$ for $G$ a positive $R$-Tannakian group stack.
\end{enumerate}
\end{thm}

\section{Proof of the neutralized theorem}\label{dualitytheorems}

We will now embark on the proof of the higher Tannaka duality statement described at the end of the last section.  For an $R$-linear tensor functor $f:C\ra\Mod_R$, we will denote by $f_A$ the composition
\[      C\xra{f}\Mod_R\xra{\bullet\otimes_RA}\Mod_A    \]
given by composing $f$ with the base change functor.

\begin{prop}\label{specbendomegah}
Let $R$ be an $\Einfty$-ring, $B$ a $R$-bialgebra and $\wh{\omega}:\SeComod_{B}\ra\Mod_{R}$ the forgetful functor.  Then there exists an equivalence
\[      \Spec B\ra\End^{\otimes}(\who).   \]
\end{prop}

\begin{proof}
Let $G=\Spec B$ and consider the homotopy pullback diagram
\begin{diagram}
*=\Spec R   &\rTo^{u}  & \wtB G\\
\uTo^{v}      &        &\uTo_{u}\\
G           &\rTo^{v}  & *=\Spec R.
\end{diagram}
The induced functors between stable, presentable, $R$-linear symmetric monoidal $\infty$-categories are given by
\begin{align*}
\whp=u_{*} &:\Mod_{R}\ra\Mod(\wt{\B}G) \\  
\who=u^{*} &:\Mod(\wt{\B}G)\ra\Mod_{R} \\  
v^{*} &:\Mod_{R}\ra\Mod_{B}  \\   
v_{*} &:\Mod_{B}\ra\Mod_{R}.   
\end{align*}
By Proposition 2.5.14 of \cite{LVIII}, we have the base change formula $u^{*}u_{*}\simeq v_{*}v^{*}$.  Therefore the composition $\who\circ\whp$ is given by $B\otimes_{R}\bullet$.  

Observe now that the object $u_*(R)$ in $\Mod(\wtB G)$ is endowed with the structure of a commutative monoid object inherited from $R$.  From Corollary 6.3.5.18 of \cite{L1} (see also Section~\ref{monoidal}) we deduce the existence of a fully faithful functor
\[   \rho:\CMon(\Mod(\wtB G))\ra(\Tens^\otimes)_{\Mod(\wtB G)/}   \]
which takes $u_*R$ to $\rho(u_*R)=u^*:\Mod(\wtB G)\ra\Mod_R$.  Thus we have the following chain of equivalences
\[  \End^\otimes(\who)\simeq\Map_{\CMon(\Mod(\wtB G))}(u_*R,u_*R)\simeq\Map_{\CAlg_R}(u^*u_*R,R)\simeq\Map_{\CAlg_R}(B,R)  \]
which completes the proof.
\end{proof}

\begin{cor}\label{counit}
Let $R$ be an $\Einfty$-ring.  Then the counit map 
\[     G\ra\Fib_{*}(\Perf_{*}(G)) \] 
is an equivalence in $\Gp(R,\tau)$ when $G$ is of the form $\Spec B$ for $B$ a Hopf $R$-algebra and is weakly rigid.
\end{cor}

\begin{proof}
We need to show that $\Spec B\ra\End^{\otimes}(\omega)$ is an equivalence where $\omega:\Perf(\wtB G)\ra\Mod_{R}^{\rig}$ for $G$ weakly rigid.  We have a diagram 
\[    G=\Spec B\xras\End^\otimes(\who)\xras\End^\otimes(\omega)   \]
where the first equivalence follows from Proposition~\ref{specbendomegah} and the second equivalence follows from the assumption that $G$ is weakly rigid.  Thus $\Spec B$ is equivalent to $\End^{\otimes}(\omega)$.
\end{proof}

Corollary~\ref{counit} states that we have a full embedding of the $\infty$-category of $R$-Tannakian group stacks into the $\infty$-category of pointed rigid $R$-tensor $\infty$-categories given by the rule $\Perf_*:G\mapsto(\Perf(\wtB G),\omega)$.  This is the first statement of Theorem~\ref{pointedinftytannakaduality}.  We now prove the remainder of the neutralized $\infty$-Tannaka duality theorem. 

\begin{proof}[Proof of Theorem~\ref{pointedinftytannakaduality}]
Let $\tau\in\{\geq 0,fl,fin\}$ be a topology.  We will first show that if $\omega$ is a $\tau$-fiber functor then the $R$-algebra $B:=\who\whp(R)$ is a $\tau$-$R$-algebra.   For all three cases we consider the projection formula $\who\whp(M)\simeq B\otimes_{R}M$ of Proposition~\ref{projectionformula} for an $R$-module $M$.  For the positive case, the functor $\who$ is t-exact by definition so by Lemma~\ref{texactlemma}, the right adjoint $\whp$ is left t-exact and so $B$ is positive.  For the flat case, $R$ is connective and the functors $\who$ and $\whp$ are t-exact so $B=\who\whp(R)$ is a connective $R$-module.  Also, the functor $\who\whp\simeq B\otimes_{R}\bullet$ is left t-exact by Lemma~\ref{texactlemma}.   It then follows from Theorem 7.2.2.15 of \cite{L1} that $B$ is flat over the connective $\Einfty$-ring $R$.  The finite case follows simply since $\who$ preserves limits by definition and $\whp$ preserves limits since it is a right adjoint.  Thus the composition $\who\whp$ preserves limits and by the projection formula we are done.

Let $(T,\omega)$ be a $\tau$-$R$-Tannakian $\infty$-category.  We will now show that the unit $(T,\omega)\ra\Perf_{*}(\Fib_*(T,\omega))$ of the adjunction $\Fib_*\dashv\Perf_*$ is an equivalence when restricted to the subcategory $(\Tan_R^{\tau})_*$, ie. $(T,\omega)$ is equivalent to $(\Perf(\wtB G),\omega_G)$ for $G=\End^{\otimes}(\omega)$ and $\omega_G:\Perf(\wtB G)\ra\Mod_{R}^{\rig}$ the forgetful functor.  We consider the map $\whphi:\wh{T}\ra\Mod(\wtB G)$.  We have the commutativity $\who\simeq\who_{G}\circ\whphi$ where $\who_{G}:\Mod(\wtB G)\ra\Mod_{R}$ is the forgetful functor and we consider the following diagram
\begin{diagram}
\wh{T}& &\pile{\rTo^{\whphi}  \\  \lTo_{\whpsi}} & & \Mod(\wtB G)\\
   &\rdTo_{\who\dashv\whp}&         &\ldTo_{\who_{G}\dashv\whp_G}&\\
                  && \Mod_{R} &&
\end{diagram}
where $\whpsi$ is the right adjoint to $\whphi$ owing to the fact that $\whphi$ commutes with colimits (it is a map between presentable $\infty$-categories in $\Tan_{R}^\otimes$).  Now observe that $\whphi$ is conservative since $\who$ is conservative (by definition of a $\tau$-fiber functor), $\who_{G}$ is the conservative forgetful functor and $\who\simeq\who_{G}\circ\whphi$.  Therefore, we have the following~:
\begin{itemize}
\item[(*)] The map $\whphi$ is an equivalence if and only if $\whphi\circ\whpsi\ra\id$ is an equivalence.  
\end{itemize}
We treat the three different topology cases separately.\\

1. \textit{Finite case}~: In other words, we assume that $\who$ is conservative and preserves limits.  We begin by proving that $\whphi\circ\whpsi(E)\ra E$ is an equivalence when $E$ is of the form $\whp_G(M)$ for $M\in\Mod_{R}$.  We have that $\whp_G(M)\simeq B\otimes_{R}M$.  Furthermore, by the commutativity of the diagram we have 
\[  \who_{G}\circ\whphi\circ\whpsi(E)\simeq\who\circ\whpsi\circ\whp_G(M)\simeq\who\circ\whp(M). \] 
By the projection formula of Proposition~\ref{projectionformula}, we have the equivalence $\who\circ\whp(M)\simeq M\otimes_{R}\who\whp(R)$.  Therefore $\who_{G}\circ\whphi\circ\whpsi(E)\simeq\who_G(B\otimes_R M)\simeq\who_{G}(E)$ and the result follows from the conservativity of $\who_{G}$.  

We now prove that $\whphi\circ\whpsi(E)\ra E$ is an equivalence for a general $B$-comodule $E$.  The functor $\who_G$ is conservative and preserves $\who_G$-split limits by Proposition 6.2.4.1 of \cite{L1}.  Thus by the dual of Proposition 6.2.2.11 of \cite{L1} we have that for all $E$ in $\SeComod_{B}$, there exists an augmented cosimplicial object $E_\bullet:\Dplus\ra\Comod_{B}$ given by $E_n\simeq(\whp_G\circ\who_G)^{n+1}E$ which is $\who_G$-split.  It follows from Proposition~\ref{splitislimit} that there exists an equivalence
\[      E\ra\underset{n\in\Delta}{\holim}~((E_0\otimes_R B^{\otimes_R n})\otimes_R B)       \]
where $E_0=\who_G(E)$.  Hence we can consider $E$ to be of the form $\holim_n\whp_G(M_n)$ for $M_n=(E_0\otimes_R B^{\otimes_R n})\in\Mod_{R}$.  

We have a diagram
\begin{diagram}
\whphi\circ\whpsi(E)   &\rTo  & E\\
\dTo                         &        &\dTo\\
\underset{n}{\holim}~\whphi\circ\whpsi(\whp_G(M_n))    &\rTo  &\underset{n}{\holim}~\whp_G(M_n).
\end{diagram}
The lower horizontal arrow is an equivalence from the above case of $E=\whp_G(M)$.  Also, the left vertical arrow is an equivalence from the fact that $\whphi\circ\whpsi$ commutes with limits~: since $\whpsi$ is a right adjoint functor it preserves limits and $\whphi$ preserves limits since $\who$ preserves limits (being a finite fiber functor), $\who_{G}$ preserves limits (because its the forgetful functor from $\Mod(\wtB G)$ where $G$ is finite) and $\who\simeq\who_G\circ\whphi$.  Thus the upper horizontal arrow is an equivalence.  Thus $\wh{T}\ra\Mod(\wtB G)$ is an equivalence.

The equivalence $\wh{T}\ra\Mod(\wtB G)$ induces an equivalence $T\ra\Perf(\wtB G)$ on rigid objects since symmetric monoidal functors preserve rigid objects.  Thus $\SeComod_B$ is ind-rigid (using Proposition~\ref{secomodmodbg}) and hence by Lemma~\ref{rigidthenweaklyrigid} the affine group stack $G$ is weakly rigid.\\

2. \textit{Positive case}~: In other words, we assume that $\who$ is conservative, creates a non-degenerate t-structure and is t-exact.  We also assume that the base $\Einfty$-ring $R$ is bounded and connective.  Let $E$ be an object in $\Mod(\wtB G)_{\leq n}$, where $\Mod(\wtB G)_{\leq n}$ denotes the full subcategory of $\Mod(\wtB G)$ spanned by objects which get mapped to $(\Mod_{R})_{\leq n}$ (with the t-structure of Example~\ref{modrtstructure}) under $\who_G$.  We have that  $E\xras\holim_n((E_0\otimes_R B^{\otimes_R n})\otimes_R B)$ is an equivalence in $\Mod(\wtB G)_{\leq n}$ as above with $E_0=\who_GE$ in $(\Mod_R)_{\leq n}$.  

Let $M_n:=(E_0\otimes_R B^{\otimes_R n})\in(\Mod_{R})_{\leq n}$ and consider the diagram
\begin{diagram}
\who\circ\whpsi(E)   &\rTo  & \who_G(E)\\
\dTo                         &        &\dTo\\
\underset{n}{\holim}~\who\circ\whpsi(M_n\otimes_R B)    &\rTo  &\underset{n}{\holim}~\who_G(M_n\otimes_R B).
\end{diagram}

The bottom horizontal arrow is an equivalence since 
\[   \who\circ\whpsi(M_n\otimes_R B)\simeq\who\circ\whpsi\circ\whp_G(M_n)\simeq\who\circ\whp(M_n)\simeq\omega_{G}(M_n\otimes_R B).  \]
The right vertical arrow is an equivalence since $\omega_{G}$ preserves $\omega_{G}$-split limits by Proposition 6.2.4.1 of \cite{L1}.  The left vertical arrow is also an equivalence by the following.  Note firstly that $\whpsi(M_n\otimes_R B)\simeq\whpsi\circ\whp_G(M_n)\simeq\whp(M_n)$ and so $\whpsi(M_n\otimes_RB)$ is in $T_{\leq n}$ since $\whp$ is left t-exact.  Therefore, the cosimplicial object 
\[     [k]\mapsto\whpsi(M_k\otimes B)  \]
in $T$ satisfies the property that $\pi_i^t(\whpsi(M_k\otimes_R B))=0$ for all $i>n$ and all $k$.  

We can then apply Lemma~\ref{tstructurelimitlemma} to the t-exact fiber functor $\who:\wh{T}\ra\Mod_{R}$ to deduce that 
\[    \underset{n}{\holim}~\who\circ\whpsi(M_n\otimes_R B)\xras\who\circ\underset{n}{\holim}~\whpsi(M_n\otimes_R B)   \]
is an equivalence.  Secondly, the functor $\whpsi$ preserves limits (it is a right adjoint) so the left vertical arrow is an equivalence.  We then have that the map $\who_G\circ\whphi\circ\whpsi(E)\simeq\who\circ\whpsi(E)\ra\omega_{G}(E)$ is an equivalence.  Thus by the conservativity of $\who_G$, we obtain the equivalence 
\[   \wh{T}_{\leq n}\ra\Mod(\wtB G)_{\leq n}.  \]

We deduce that the map $\cup_n\wh{T}_{\leq n}\ra\cup_n\Mod(\wtB G)_{\leq n}$ between the unions over all $n$ is an equivalence.  Rigid modules over a bounded $\Einfty$-ring $R$ are bounded and so $T$ is contained in $\cup_n\wh{T}_{\leq n}$ (in fact $\wh{T}=\cup_n\wh{T}_{\leq n}$ since $\wh{T}$ is left t-complete owing to its non-degenerate t-structure) and $\Perf(\wtB G)$ is contained in $\cup_n\Mod(\wtB G)_{\leq n}$.  Since $\whphi$ is a symmetric monoidal functor it preserves rigid objects and so we obtain an induced fully faithful functor $\whphi:T\ra\Perf(\wtB G)$.  

However, for any $E,F\in\Perf(\wtB G)$, we have a diagram
\begin{diagram}
\who(\whpsi(E)\otimes\whpsi(F))   &\rTo^{\sim}  & \who_G(E)\otimes\who_G(F)\\
\dTo                         &        &\dTo_{\sim}\\
\who\circ\whpsi(E\otimes F)    &\rTo^{\sim}  & \who_G(E\otimes F)
\end{diagram}
using the symmetric monoidal structures on $\who$ and $\who_G$.  Therefore the left vertical arrow of the diagram is an equivalence.  Conseqently, the map $\whpsi:\Perf(\wtB G)\ra T$ is symmetric monoidal and thus preserves rigid objects leading to the equivalence
\[ T\xras\Perf(\wtB G)  \]
of $\infty$-categories.  Finally, $G$ is weakly rigid by combining this equivalence with Proposition~\ref{specbendomegah}.\\
 
3. \textit{Flat case}: In other words, we assume that $\who$ is conservative, creates a non-degenerate t-structure, is t-exact and whose right adjoint is t-exact.  The first part of the proof can be deduced directly from the positive case~: it follows that $\wh{T}_{\leq n}\ra\Mod(\wtB G)_{\leq n}$ is an equivalence.  However, here the base $\Einfty$-ring $R$ is merely connective.  In the flat case though, we actually have an equivalence
\[     \wh{T}\simeq\underset{n}{\holim}~\wh{T}_{\leq n}\ra\underset{n}{\holim}~\Mod(\wtB G)_{\leq n}\simeq\Mod(\wtB G)  \]
of $\infty$-categories where the identification on the left hand side follows from the left t-exactness of $\wh{T}$ and the identification on the right hand side follows from the fact that a non-degenerate t-structure is created on $\Mod(\wtB G)$.  Since monoidal functors preserve rigid objects we have an equivalence $T\ra\Perf(\wtB G)$ of $\infty$-categories.  Combining this equivalence with Proposition~\ref{specbendomegah} we find that $G$ is rigid and hence weakly rigid.
\end{proof}

\section{Neutral Tannaka duality for $(\infty,1)$-categories}\label{neutral}

We now consider the case where there simply exists a $\tau$-fiber functor.  This is the non-pointed case, otherwise known as \textit{neutral} Tannaka duality for $\infty$-categories.

\begin{dfn}\label{nonpointeddfn}
Let $R$ be an $\Einfty$-ring and $\tau\in\{fin,fl,\geq 0\}$ a topology.  A rigid $R$-tensor $\infty$-category $T$ is said to be a \textit{$R$-Tannakian $\infty$-category} with respect to $\tau$ if there exists a $\tau$-fiber functor $w:T\ra\Mod_{R}^{\rig}$.
\end{dfn}

We denote the $\infty$-category of $R$-Tannakian $\infty$-categories with respect to $\tau$ by $\Tan_{R}^{\tau}$.  The objects on the other side of the correspondence are described as follows.

\begin{dfn}\label{tannakiangerbe}
Let $R$ be an $\Einfty$-ring.  A stack $F$ in $\St(R,\tau)$ is said to be a \textit{$\tau$-$R$-Tannakian gerbe} if it is locally equivalent to $\widetilde{\B}G$ for $G$ a $\tau$-$R$-Tannakian group stack.  It is said to be a \textit{neutral} $\tau$-$R$-Tannakian gerbe if there moreover exists a morphism of stacks $*\ra F$.   
\end{dfn}  

Let $\TGer(R,\tau)$ denote the $\infty$-category of neutral $\tau$-$R$-Tannakian gerbe's.  We have a natural inclusion $\TGer(R,\tau)\subseteq\Ger(R,\tau)$.  

We now state the Tannaka duality theorem for $\infty$-categories in the neutral setting.  Note that we have a weaker statement in positive case owing to the fact that the positive topology is not subcanonical.  

\begin{thm}[Neutral $\infty$-Tannaka duality]\label{inftytannakaduality}
Let $\tau$ be a subcanonical topology.  Then the map
\[    \Perf:\TGer(R,\tau)^{op}\ra\Tens^{\rig}_R  \]
is fully faithful.  Moreover, the adjunction $\Fib\dashv\Perf$ induces the following~:
\begin{enumerate}
\item Let $R$ be an $\Einfty$-ring.  Then $T$ is a finite $R$-Tannakian $\infty$-category if and only if it is of the form $\Perf(G)$ for $G$ a neutral finite $R$-Tannakian gerbe.
\item Let $R$ be a connective $\Einfty$-ring.  Then $T$ is a flat $R$-Tannakian $\infty$-category if it is of the form $\Perf(G)$ for $G$ a neutral flat $R$-Tannakian gerbe.
\item Let $R$ be a bounded connective $\Einfty$-ring.  If $(T,\omega_1)$ and $(T,\omega_2)$ are two pointed positive $R$-Tannakian $\infty$-categories then there exists a positive cover $R\ra Q$ such that
\[      \omega_1\otimes_{R}Q\ra\omega_2\otimes_{R}Q   \]
is an equivalence.
\end{enumerate}
\end{thm}

To prove the neutral $\infty$-Tannaka duality statement of Theorem~\ref{inftytannakaduality} it suffices to combine the neutralized statement of Theorem~\ref{pointedinftytannakaduality} with the demonstration that two fiber functors are equivalent after base change.

\begin{prop}\label{locallyequivalent}
Let $R$ be an $\Einfty$-ring and $\tau\in\{\geq 0,fl,fin\}$.  Given two $\tau$-fiber functors $\omega$ and $\nu$ over $R$, there exists a $\tau$-cover $R\ra Q$ in $\fE$ such that $\omega$ and $\nu$ are equivalent over $Q$. 
\end{prop}

\begin{proof}
Let $(T,\omega)$ and $(T,\nu)$ be two pointed $R$-Tannakian $\infty$-categories with respect to $\tau$.  By Proposition~\ref{rigidnaturaltransformation} this amounts to showing that there exists a $\tau$-cover $R\ra Q$ such that $\uHom(\omega,\nu)(Q)=\Map(\omega_Q,\nu_Q)\neq\emptyset$.  It suffices to prove that $\uHom(\who,\wh{\nu})(Q)\neq\emptyset$.  Let $\whp$ and $\whq$ denote the right adjoints to $\who$ and $\wh{\nu}$ respectively.  We have equivalences 
\[  \uHom(\who,\wh{\nu})\simeq\uHom(\who_G,\wh{\nu}_{G'})\simeq\Map(\whp_GR,\whq_{G'}R)\simeq\Map(\wh{\nu}_{G'}\whp_GR,R)  \]
where the first equivalence follows from Theorem~\ref{pointedinftytannakaduality} and the second equivalence from Corollary 6.3.5.18 of \cite{L1} (see the proof of Proposition~\ref{specbendomegah}).  Therefore, $\uHom(\who,\wh{\nu})(Q)\simeq\Map_{\CAlg_R}(\wh{\nu}(B),Q)$ and we set $Q:=\wh{\nu}(B)$ to consider the identity map.  It remains to show that there exists a $\tau$-cover $R\ra\wh{\nu}(B)$.  Since $\wh{\nu}$ is $R$-linear, $\wh{\nu}(B)\otimes_{R}\bullet\simeq\wh{\nu}(B\otimes_{R}\bullet)$ so $R\ra\wh{\nu}(B)$ is a $\tau$-$R$-algebra since $B$ is a $\tau$-Hopf $R$-algebra and $\nu$ is a $\tau$-fiber functor.

We will now show that $R\ra\wh{\nu}(B)$ is conservative, ie. given $M\in\Mod_{R}$ such that $\wh{\nu}(B)\otimes_{R}M\simeq 0$ then $M\simeq 0$.   Let $B'=B\otimes_RR'$.  By the projection formula, we have the following statement~:
\begin{itemize}
\item[(*)] For all $R\ra R'$, the map $\wh{\nu}(B')\ra\wh{\nu}(B)\otimes_RR'$ is an equivalence.
\end{itemize}
By Proposition~\ref{hopflimit}, the map $R'\ra\holim_{n\in\Delta}((B')^{\otimes_{R}(n+1)})$ is an equivalence and since $\nu$ is a $\tau$-fiber functor we have the following statement~:
\begin{itemize}
\item[(**)] The map $R'\ra\holim_{n\in\Delta}(\wh{\nu}(B')^{\otimes_{R}(n+1)})$ is an equivalence.
\end{itemize}
Set $R'=\Sym_{R}(M):=\coprod_{p\geq 0}M^{\otimes_{R}p}/\Sigma_{p}$ and assume that $\wh{\nu}(B)\otimes_{R}M\simeq 0$.  Therefore
\begin{align*}
R' &\simeq\underset{n\in\Delta}{\holim}(\wh{\nu}(B)^{\otimes_{R}n+1}\underset{R}{\otimes}\Sym_{R}(M))\\
            &\simeq\underset{n\in\Delta}{\holim}\coprod_{p\geq 0}(\wh{\nu}(B)^{\otimes n+1}\underset{R}{\otimes}M^{\otimes_{R}p}/\Sigma_{p})\\
            &\simeq R.
\end{align*}                        
The first line is an equivalence in $\CAlg_{R}$ and follows from $(**)$ and $(*)$.  The second line is an equivalence in $\Mod_{R}$ and follows from the fact that the tensor product commutes with coproducts and the forgetful functor $\CAlg_{R}\ra\Mod_{R}$ is conservative and commutes with limits (it is a right adjoint).   The third line is a result of the only non-zero term being $p=0$ by assumption followed by $(**)$ applied to $R$.  Thus $M\simeq 0$.
\end{proof}

\begin{proof}[Proof of Theorem~\ref{inftytannakaduality}]
The first part of the proof follows from Corollary~\ref{counit}.  The remainder follows directly from Theorem~\ref{pointedinftytannakaduality} and Proposition~\ref{locallyequivalent}.
\end{proof}

\section{Comparison with the classical theory}\label{classicalcomparison}

The following series of comparison results shows that the Tannaka duality theorem for $\infty$-categories of Section~\ref{neutral} naturally generalises the classical theory.  Let $k$ be a field.  In \cite{Sa}, Saavedra defined the notion of a neutral Tannakian category over $k$.  That is, a rigid abelian $k$-linear symmetric monoidal category $T$ for which there exists an exact $k$-linear symmetric monoidal functor $T\ra\Vect_{k}$ (the fiber functor) taking values in the category of finite rank projective $k$-modules.  

The collection of ordinary fiber functors form a stack $\Fib(T)$ over $k$ in the faithfully flat quasi-compact topology, denoted $ffqc$.  Recall that a finite family $\{A_{i}\ra A\}_{i\in I}$ of arrows in $\Aff_{k}$ is an $ffqc$ cover if the morphism $\prod_{i\in I}A_{i}\ra A$ is faithfully flat (ie. exact and conservative).  Let $\St(k,ffqc)$ denote the $\infty$-category of stacks on the classical site $(\Aff_{k},ffqc)$.  

\begin{prop}\label{fullyfaithfuli}
Let $k$ be a commutative ring and $Hk$ its corresponding Eilenberg-Mac Lane ring spectrum.  The inclusion 
\[       i:\St(k,ffqc)\ra\St(Hk,fl)        \]
of $\infty$-categories is fully faithful.
\end{prop}

\begin{proof}
We will show that there exists a composition of fully faithful maps
\[  \St(k,ffqc)\hookrightarrow\St_{c}(Hk,fl)\hookrightarrow\St(Hk,fl)  \]
where $\St_{c}(Hk,fl)$ is the $\infty$-category of stacks on the site of connective $Hk$-algebras.  Firstly, there exist adjunctions
\begin{diagram}
\CAlg_{k}  &\pile{ \rTo^{i}  \\  \lTo_{\tau_{\leq 0}} }  &\CAlg_{Hk}^{c}  &\pile{ \rTo^{j}  \\  \lTo_{(\bullet)^{c}} }  &\CAlg_{Hk}
\end{diagram}
of $\infty$-categories (see Section~\ref{spectralalgebra}).  We then left Kan extend these adjunctions to obtain the adjunctions
\begin{diagram}
\RHom(\CAlg_{k},\cS)  &\pile{ \rTo^{i_{!}}  \\  \lTo_{\tau_{\leq 0}^{*}} }  &\RHom(\CAlg_{Hk}^{c},\cS)  &\pile{ \rTo^{j_{!}}  \\  \lTo_{((\bullet)^{c})^{*}} }  &\RHom(\CAlg_{Hk},\cS)
\end{diagram}
This chain of adjunctions induces adjunctions on the subcategories of stacks by a property of Bousfield localisations since $i$ preserves $ffqc$-covers.   It follows from Lemma 2.2.4.1 of \cite{TVII} that $i_!:\St(k,ffqc)\rightarrow\St_{c}(Hk,fl)$ is fully faithful.  A similar argument holds for $j_!:\St_{c}(Hk,fl)\hookrightarrow\St(Hk,fl)$.
\end{proof}

The right adjoint to the functor $i$ of Proposition~\ref{fullyfaithfuli}   
\[       \tau_{\leq 0}:\St(Hk,fl)\ra\St(k,ffqc)       \]
is explicitly given by $\tau_{\leq 0}(F)(k'):=F(Hk')$ for a stack $F$ in $\St(Hk,fl)$.  If $F$ is a neutral affine gerbe in the sense of \cite{D2} then $i(F)$ in $\St(Hk,fl)$ is a neutral flat $Hk$-Tannakian gerbe.  Thus neutral affine gerbes as defined in the classical Tannakian theory form a full subcategory of neutral flat $Hk$-Tannakian gerbes in our sense.  

For the following comparison results, we will need to introduce a finiteness condition on some of our categories in order for the statements to make sense.   The notion is that of finite cohomological dimension of an arbitrary abelian category.  

\begin{dfn}\label{finitecd}
Let $C$ be an abelian category.  Then $C$ is said to be of \textit{finite cohomological dimension} if there exists $n$ such that for every object $x$ in $C$, the group $\Ext^{i}(y,x)=0$ for all $i>n$ and $y\in C$. 
\end{dfn}

Let $T$ be an abelian category and $C(T)$ (resp. $C^{b}(T)$) be the category of unbounded (resp. bounded) complexes in $T$.  We denote by $LC^{b}(T)$ the $\infty$-category given by localising $C^{b}(T)$ at the set of quasi-isomorphisms.  If $T$ is a $R$-Tannakian $\infty$-category (resp. $k$-Tannakian category) we will denote $\Fib^{\tau}(T)$ the neutral $R$-Tannakian gerbe (resp. neutral affine gerbe) of fiber functors on $C$ with respect to the topology $\tau$.

\begin{prop}\label{fiberfunctorequivalence}
Let $k$ be a field and $T$ a $k$-Tannakian category.  Assume that $T$ is of finite cohomological dimension.  Then the $\infty$-category $LC^{b}(T)$ is a flat $Hk$-Tannakian $\infty$-category and there exists an equivalence
\[       i(\Fib^{ffqc}(T))\ra\Fib^{fl}(LC^{b}(T))          \]
of stacks in $\St(Hk,fl)$.
\end{prop}

\begin{proof}
Since $T$ is a $k$-Tannakian category there exists a natural $k$-tensor functor $\Mod_{k}^{\rig}\ra T$ between categories defined by $M\mapsto M\otimes 1_{T}$.  This exists since for any $x,y\in T$, the $k$-linear structure on $T$ induces an equivalence 
\[    \Hom_{\Mod_{k}^{\rig}}(M,\Hom_{T}(x,y))\simeq\Hom_{T}(M\otimes x,y).  \]  
Using the equivalence $LC^{b}(\Mod_{k}^{\rig}(\Ab))\simeq\Mod_{Hk}^{\rig}(\Sp)$ (see Example~\ref{cdgaca}) we obtain a map
\[    \Mod_{Hk}^{\rig}\ra LC^{b}(T)  \]
of symmetric monoidal $\infty$-categories.  This map commutes with colimits and induces the $Hk$-linear structure on $LC^b(T)$.  Furthermore, the $\infty$-category $LC^{b}(T)$ is stable by Example~\ref{derivedinftycategory} which makes $LC^{b}(T)$ a rigid $Hk$-tensor $\infty$-category.  The fiber functor $\omega':LC^{b}(T)\ra\Mod_{Hk}^{\rig}$ induced from the fiber functor $\omega:T\ra\Mod^\rig_k$ on $T$ is clearly flat~: the forgetful functor $\who':LC(T)\ra LC(k)$ is conservative and creates a t-structure for which $\who'$ and its right adjoint are t-exact. 

For the second part of the proposition, consider the classical Tannaka duality theorem which states that $\Fib^{ffqc}(T)$ is naturally equivalent to a neutral gerbe $H$ in the \textit{ffqc} topology which is locally equivalent to $\wtB G$ for an affine group scheme $G$.  We also have that $\Fib^{fl}(LC^{b}(T))$ is equivalent to a neutral flat $Hk$-Tannakian gerbe $H'$ locally equivalent to $\wtB G'$ for $G'$ a flat $Hk$-Tannakian group stack.   From Proposition~\ref{specbendomegah}, the gerbe $H'$ is locally equivalent to $\wtB(iG)$.  From the equivalence $i(\wtB G)\simeq\wtB (iG)$, where $i$ is the fully faithful functor of Proposition~\ref{fullyfaithfuli}, we obtain the desired result.  
\end{proof}

Let $\Tan^{ffqc,!}_{k}$ denote the category of $k$-Tannakian categories of finite cohomological dimension with respect to the $ffqc$ topology.

\begin{cor}\label{tanffqctanfl}
Let $k$ be a field.  Then the functor
\[     LC^{b}:\Tan^{ffqc,!}_{k}\ra\Tan^{fl}_{Hk}   \]
is an equivalence of $\infty$-categories.
\end{cor}

\begin{proof}
This follows directly from Proposition~\ref{fiberfunctorequivalence}, Proposition~\ref{fullyfaithfuli} and the flat $\infty$-Tannaka duality Theorem~\ref{pointedinftytannakaduality}.
\end{proof}

\section{Perfect complexes and schematization}\label{overfields}

Tannakian $\infty$-categories arise in a number of mathematical applications.  We have seen in Section~\ref{classicalcomparison} that the Tannaka duality statement for $\infty$-categories with the flat topology subsumes the classical statement.  It furthermore extends the 1-categorical duality by allowing the theorem to be defined over arbitrary commutative rings as opposed to just fields.  In this section we will discuss the example of the $\infty$-category of perfect complexes on a topological space $X$.  The $\infty$-category of perfect complexes on $X$ is a $k$-Tannakian $\infty$-category with respect to the positive topology for $k$ a field.  When $k$ is a field of characteristic zero, then the Tannakian dual of the $\infty$-category of perfect complexes on $X$ is precisely the schematization of $X$ introduced by To\"en in \cite{T4}.  

In Section~\ref{tannakadualityforinftycategories} we introduced the notion of a Tannakian $\infty$-category over a commutative ring spectrum $R$.  In fact one can define a whole family of Tannaka duality theorems based on the topology chosen on the site of $R$-algebras.  In Section~\ref{topologies} we concerned ourselves with three examples called the flat, finite and positive topologies.  The flat topology is the topology which links us to the classical theory of Tannaka duality over fields for the \textit{ffqc} topology (see Section 10 of \textit{loc.cit.})  and the use of the finite topology is rather limited.  The positive topology is however something new and in this section we would like to explore some uses of the positive topology and hence to the pointed positive Tannaka duality theorem for $\infty$-categories.  

We will fix a field $k$ throughout this section.  In \cite{T4}, To\"en introduced the notion of a schematic homotopy type.  Given two pointed stacks $F$ and $G$, let $\Map_*(F,G)$ denote the $\inftyz$-category of pointed maps between them.  

\begin{dfn}
Let $F$ and $G$ be two pointed stacks in $\St(k,ffqc)_{*}$.  A map $F\ra G$ is said to be a \textit{P-equivalence} if the induced map
\[    \Perf(F)\ra\Perf(G)  \]
is an equivalence of $\infty$-categories.  A pointed stack $F$ in $\St(k,ffqc)_{*}$ is said to be \textit{P-local} if for any $P$-equivalence $H_1\ra H_2$, the induced map
\[    \Map_{*}(H_2,F)\ra\Map_{*}(H_1,F)       \]
is an equivalence.
\end{dfn}

\begin{lem}
Let $\tau\in\{\geq 0,fl,fin\}$ and $T$ be a pointed $k$-Tannakian $\infty$-category with respect to $\tau$.  Then $\wtB\Fib_{*}(T)$ is $P$-local.
\end{lem}

\begin{proof}
Let $H_1\ra H_2$ be a $P$-equivalence.  From the pointed adjunction of Lemma~\ref{pointedadjunction} we have a chain of equivalences
\[  \Map_*(H_2,\wtB\Fib_*(T))\simeq\Map_*(T,\Perf(H_2))\simeq\Map_*(T,\Perf(H_1))\simeq\Map_*(H_1,\wtB\Fib_*(T))  \]
of $\inftyz$-categories.
\end{proof}

Let $A\in c\tu{CAlg}_{k}$ be a cosimplicial $k$-algebra in the category $\Ab$ of abelian groups.  We define the following prestack
\begin{align*}
\Spec A:\CAlg_{k}  &\ra\cS\\
           B &\mapsto\Map(A,B)
\end{align*}
where $\Map(A,B)_{n}:=\Hom(A_n,B)$.  This prestack can be shown to be a stack for the \textit{ffqc}-topology and thus defines a natural functor 
\[    \Spec: c\tu{CAlg}_{k}\ra\St(k,ffqc)  \]
between $\infty$-categories.  We are interested in objects which lie in the essential image of the $\Spec$ functor.

\begin{dfn}
Let $k$ be a commutative ring.  An \textit{affine stack} over $k$ is a stack in $\St(k,ffqc)$ which is equivalent to an object of the form $\Spec A$ for $A$ a cosimplicial $k$-algebra.
\end{dfn}

Let $s:*\ra F$ be a pointed stack in $\St(k,ffqc)$.  We define the prestack $\Omega_{*}F$ of loops at $s$ by the formula
\begin{align*}
\Omega_{*}F:\CAlg_{k} &\ra\cS\\
           x &\mapsto\Omega_{s(*)}F(x)
\end{align*}
where $\Omega_{s(*)}F(x)$ is the subsimplicial set of $\Map(\Delta^{1},F(x))$ which sends the endpoints $\{0,1\}$ of $\Delta^{1}$ to $s(*)$.

\begin{dfn}
Let $k$ be a field.  A pointed stack $F$ on the site $(\Aff_{k},ffqc)$ is said to be a \textit{pointed schematic homotopy type} over $k$ if it is $P$-local, connected and $\Omega_{*}F$ is an affine stack over $k$.
\end{dfn}

Let $\SHT_{k}$ denote the full subcategory of $\St(k,ffqc)$ spanned by the pointed schematic homotopy types.  Every pointed connected affine stack is naturally a pointed schematic homotopy type.  Furthermore, a stack $F$ is a pointed schematic homotopy type if and only if $F$ is a pointed connected stack in $\St(k,ffqc)$ such that the sheaf $\pi_{1}(F,x)$ is represented by an affine group scheme and for any $i > 1$, the sheaf $\pi_{i}(F,x)$ is represented by a unipotent affine group scheme (see Section 3.2 of \cite{T4}).

Let $k$ be a field and $F$ a pointed schematic homotopy type over $k$.  We will say that $F$ is of finite cohomological dimension if the abelian category $\Perf(i(F))^\heartsuit$ is of finite cohomological dimension.

\begin{prop}
Let $k$ be a field and $F$ a pointed schematic homotopy type over $k$.  Assume that $F$ is of finite cohomological dimension.  Then $\Perf_*(i(F))$ is a pointed positive $Hk$-Tannakian $\infty$-category.
\end{prop}

\begin{proof}
The stack $iF$ is clearly a positive $Hk$-Tannakian group stack in $\TGp^{fl}(Hk)$.  The result then follows from the positive $\infty$-Tannaka duality Theorem~\ref{pointedinftytannakaduality}.
\end{proof}

The schematization is an important schematic homotopy type which we define through the following universal property.

\begin{dfn}\label{schematization}
Let $k$ be a field and $X$ a fibrant simplicial set considered as a constant prestack.  Then the \textit{schematization} of $X$ over $k$ is a schematic homotopy type $\schx$ over $k$ together with a map $f:X\ra\schx$ satisfying the following universal property~: for any stack $F$ in $\St(k,ffqc)$, composition with $f$ induces an equivalence
\[    \Map_{\St(k,ffqc)}(\schx,F)\ra\Map_{\St(k,ffqc)}(X,F)  \]
of $\inftyz$-categories.
\end{dfn}

The main result in the theory of schematizations is Theorem 3.3.4 of \cite{T4} which states that for a pointed, connected simplicial set $(X,x)$, a schematization always exists.  The schematization of a pointed space admits a number of fundamental properties which follow from Definition~\ref{schematization} together with this existence result.  Firstly, the affine group scheme $\pi_1(\schx,x)$ is naturally isomorphic to the pro-algebraic completion of the discrete group $\pi_1(X,k)$ (thought of as a constant group scheme over $k$).  When $X$ is finite and simply connected, then for $i>1$, the group scheme $\pi_i(\schx,x)$ is naturally isomorphic to the pro-unipotent completion of the discrete group $\pi_{i}(X,x)$.  Finally, if $V$ is a local system of finite dimensional $k$-vector spaces on $X$, then $V$ corresponds to a linear representation of the affine group scheme $\pi_1(\schx,x)$.  This induces a local system $\mathcal{V}$ on the schematization $\schx$ (a sheaf of abelian groups $\mathcal{V}$ on $(\Aff_{k},ffqc)$ together with an action of $\pi_{1}(\schx)$ on $\mathcal{V}$) such that the map $X\ra\schx$ furnishes an isomorphism
\[    H^*(\schx,\mathcal{V})\ra H^*(X,V)  \]
in cohomology with local coefficients.  See \cite{T4} for further discussion.

Let $k$ be a commutative ring.  The category $C(k)$ of (unbounded) complexes of $k$-modules admits a cofibrantly generated model structure where the fibrations are degree-wise surjective morphisms and the equivalences are the quasi-isomorphisms (those maps inducing isomorphisms on homology groups) (see Theorem 2.3.11 of \cite{H1}).  Then for a topological space $X$, the category $C(X,k)$ of complexes of presheaves of $k$-modules on $X$ is a $C(k)$-enriched model category.  Here $C(k)$ is endowed with its usual monoidal structure and $C(X,k)$ is naturally tensored over $C(k)$ since the category of presheaves of $k$-modules is naturally tensored over the category of $k$-modules (if $F$ is a presheaf of $k$-modules on $X$ and $M$ is a $k$-module then $M\otimes F$ is defined to be the presheaf $U\mapsto M\otimes F(U)$).  The model category structure on $C(X,k)$ arises from a more general result of a model structure on the category $C(A)$ of complexes in any Grothendieck category $A$ (see \cite{H2}).

We define $\Perf(X,k)$ to be the subcategory of $LC(X,k)$ consisting of complexes of presheaves of $k$-modules which have locally constant cohomology sheaves.  The objects of $\Perf(X,k)$ are complexes of presheaves of $k$-modules which are, locally on $X$, quasi-isomorphic to a constant complex of presheaves associated with a bounded complex of finite rank projective $k$-modules.  Let $x:*\ra X$ be a point of $X$ and define $\Perf(k):=\Perf(*,k)$.  There exists a $Hk$-linear structure on the symmetric monoidal $\infty$-category $\Perf(X,k)$ from the symmetric monoidal functor $\Perf(k)\ra\Perf(X,k)$ and a $Hk$-linear symmetric monoidal functor 
\[   \omega_{x}:\Perf(X,k)\ra\Perf(k)    \]
induced by the point $x$.

\begin{thm}\label{perfistannakian}
Let $k$ be a field and $X$ a finite CW complex.  Then $(\Perf(X,k),\omega_{x})$ is a pointed $Hk$-Tannakian $\infty$-category with respect to the positive topology.  When $k$ is a field of characteristic zero, the dual of this pointed positive $Hk$-Tannakian $\infty$-category is the schematization of $X$. 
\end{thm}

\begin{proof}
We need to show that $\Perf(X,k)$ is a rigid $Hk$-tensor $\infty$-category and $\omega_{x}$ a positive fiber functor.  The first part is clear by the above discussion and by Example~\ref{derivedinftycategory}~: $\Perf(X,k)$ is stable since it is a stable subcategory of the derived $\infty$-category of an additive category.  Secondly, $\who_{x}:\Ind(\Perf(X,k))\ra\Mod(k)$ is the conservative forgetful functor and clearly creates a t-structure using the canonical t-structure on $\Perf(k)$ whilst preserving finite limits and colimits.

For the second part of the proof, let $k$ be of characteristic zero and consider the fully faithful classifying prestack functor
\[   \wtB:\Gp(k,ffqc)\ra\St(k,ffqc)_*.  \]  
We will say that a group stack $G$ in $\Gp(k,ffqc)$ is $P$-local if $\wtB G$ is $P$-local.  Let $L_P(\Gp(k,ffqc))$ be the localization of $\Gp(k,ffqc)$ with respect to the $P$-equivalences.  We can identify this $\infty$-category with the full subcategory of $\Gp(k,ffqc)$ spanned by the $P$-local objects.  

From Theorem 3.14 of \cite{KPT2}, the functor $\wtB$ induces an equivalence
\[    L_P(\Gp(k,ffqc))\ra\SHT_k    \]
of $\infty$-categories.  Thus by Corollary~3.19 of \textit{loc. cit.} there exists a Hopf $k$-algebra $B^P$ in the category $\Ab$ of abelian groups such that the map $\schx\ra\wtB\Spec B^P$ is an equivalence in $\St(k,ffqc)_*$.  Therefore, there exists an equivalence
\[    i\schx\ra i\wtB\Spec B^P  \]
in $\St(Hk,fl)_*$ where $i$ is the fully faithful inclusion of Proposition~\ref{fullyfaithfuli}. 

Now $\Fib_*(\Perf(X,k))=\End^\otimes(\omega_x)=\Spec B'$ for the Hopf $Hk$-algebra $B'\simeq \wh{\omega_x}\circ\wh{p_x}(k)$.  It remains to show that $B'$ is equivalent to $i\wtB\Spec B^P$.  However, this is clear from the explicit form of $B^P$ \cite{T4}.
\end{proof}

\newpage


\end{document}